\documentclass[a4paper]{amsart}

\usepackage[utf8]{inputenc}
\usepackage{amssymb}
\usepackage{amsthm}
\usepackage{amsmath}
\usepackage{mathtools}
\usepackage{enumitem}
\usepackage{nicefrac}
\usepackage{xcolor}

\usepackage{url}

\begin{document}
	
\theoremstyle{plain}
\newtheorem{lemma}{Lemma}
\numberwithin{lemma}{section}
\newtheorem{proposition}[lemma]{Proposition}
\newtheorem{corollary}[lemma]{Corollary}
\newtheorem{theorem}[lemma]{Theorem}
	
\theoremstyle{definition}
\newtheorem{definition}[lemma]{Definition}
\newtheorem{example}[lemma]{Example}

\theoremstyle{remark}
\newtheorem{remark}[lemma]{Remark}

\newcommand{\period}{\text{.}}
\newcommand{\comma}{\text{,}}

\newcommand{\rca}{\textsf{RCA}}
\newcommand{\wkl}{\textsf{WKL}}
\newcommand{\aca}{\textsf{ACA}}
\newcommand{\atr}{\textsf{ATR}}
\newcommand{\pa}{\textsf{PA}}
\newcommand{\idone}{\textsf{ID}_1}
\newcommand{\pica}{\Pi^1_1\textsf{-CA}_{0}}
\newcommand{\picaminus}{\Pi^1_1\textsf{-CA}^{-}_{0}}
\newcommand{\pitr}{\Pi^1_1\textsf{-TR}_{0}}
\newcommand{\deltaonecr}{\Delta^1_1\textsf{-CR}}
\newcommand{\deltaoneca}{\Delta^1_1\textsf{-CA}}
\newcommand{\zfc}{\textsf{ZFC}}
\newcommand{\on}{\textsf{On}}
\newcommand{\wo}{\textsf{WO}}
\newcommand{\lo}{\textsf{LO}}
\newcommand{\set}{\textsf{Set}}
\newcommand{\dil}{\textsf{Dil}}
\newcommand{\Lim}{\textsf{Lim}}
\newcommand{\id}{\textsf{Id}}
\newcommand{\n}{\mathbb{N}}
\newcommand{\z}{\mathbb{Z}}
\newcommand{\q}{\mathbb{Q}}
\newcommand{\rr}{\mathbb{R}}
\newcommand{\supp}{\textsf{supp}}
\newcommand{\fin}{\textsf{fin}}
\newcommand{\bachmann}{\vartheta(\varepsilon_{\Omega+1})}
\newcommand{\smallveblen}{\vartheta(\Omega^\omega)}
\newcommand{\phiomega}{\varphi(\omega, 0)}
\newcommand{\psiOmegaomega}{\psi(\Omega^\omega)}
\newcommand{\ot}{\textsf{OT}(\vartheta)}
\newcommand{\rng}{\textsf{rng}}

\newcommand{\restr}[2]{#1 \mathbin{\upharpoonright} #2}

\title{On inverse Goodstein sequences}
\author[P. Uftring]{Patrick Uftring}

\address{Patrick Uftring, University of W\"urzburg, Institute of Mathematics, Emil-Fischer-Stra{\ss}e~40, 97074 W\"urzburg, Germany}
\email{patrick.uftring@uni-wuerzburg.de}
\thanks{This work has been funded by the Deutsche Forschungsgemeinschaft (DFG, German Research Foundation) -- Project number 460597863.}

\begin{abstract}
	In the late 1980s, Abrusci, Girard and van de Wiele defined a variant of Goodstein sequences: the so-called inverse Goodstein sequence. In their work, they show that it terminates precisely at the Bachmann-Howard ordinal. This reveals that a proof of this fact requires substantial consistency strength.
	Moreover, the authors could show that sequences of this kind terminate even if the hereditary base change at the heart of their construction is replaced by a generalization using arbitrary dilators.
	
	It has been a conjecture by Andreas Weiermann that this more general result has a connection to Bachmann-Howard fixed points and is, therefore, equivalent to one of the most famous strong set existence principles from reverse mathematics: $\Pi^1_1$-comprehension. In this article, we prove this conjecture to be correct.
	Moreover, we show that the ordinal at which such sequences terminate is, in a fundamental way, isomorphic to the $1$-fixed point of their dilator, a new concept introduced by Freund and Rathjen. This yields explicit notation systems and a general method for specifying such ordinals.
	
	Also, using the notation systems provided by $1$-fixed points, we can reproduce the result that the Goodstein sequence terminates at the Bachmann-Howard ordinal in a weak system. Additionally, we perform a similar computation for a variant of Goodstein sequences, which terminates at a predicative~ordinal.
\end{abstract}

\maketitle

\section{Introduction}

One of the most fascinating insights of proof theory lies in the fact that there are number theoretic statements that require surprisingly strong principles in order to be proven. An example for such a statement is given by the \emph{termination of Goodstein sequences.}

In the 1940s, Goodstein devised the following construction (cf.~\cite{Goodstein44}): Take an arbitrary natural number (e.g.~$5$) and express it as a term in \emph{hereditary base-$2$ notation} ($2^{2^{2^0}} + 2^0$), i.e., a variant of base-$2$ notation where also all exponents are, hereditarily, given in base-$2$ notation. The resulting term only consists of numbers less or equal to $2$ combined by the operations of addition, multiplication, and exponentiation. Now, replace all occurrences of $2$ by $3$ ($3^{3^{3^0}} + 3^0$), evaluate this term ($28$) and subtract $1$ ($27$). Let us summarize this construction:
\begin{equation*}
	5 \quad \rightsquigarrow \quad 2^{2^{2^0}} + 2^0 \quad \rightsquigarrow \quad 3^{3^{3^0}} + 3^0 \quad \rightsquigarrow \quad 28 \quad \rightsquigarrow \quad 27\period
\end{equation*}
In the next step, we apply all of this to $27$ but instead of hereditary base-$2$ notation, we interpret $27$ as a term in hereditary base-$3$ notation and replace all occurrences of $3$ by $4$. In the following steps, we continue this pattern: If we are in the $n$-th step of our construction, we compute the hereditary base-$(n+1)$ notation of our number, replace all occurrences of $n+1$ by $n+2$, subtract $1$ from the result of the evaluation and repeat the process. We collect the numbers produced by each step in a sequence, called the \emph{Goodstein sequence} for our given starting value.

It is no surprise that for sufficiently large starting values, this construction yields greater and greater elements in each next step and the subtraction of $1$ is practically negligible; the effect of incrementing the hereditary base is simply too strong. However, Goodstein could prove the surprising result that his construction must always terminate, i.e.~reach $0$, for any starting value. In fact, he considered an even broader statement that allowed for a large class of hereditary base substitutions.

The usual proof assigns an element of the ordinal $\varepsilon_0$ to the result of each step and proves that the sequence in $\varepsilon_0$ that emerges by this process must be strictly descending. Thus, if some Goodstein sequence does not terminate, this yields the contradiction that $\varepsilon_0$ is ill founded.

Proof theorists, especially those that work in the field of \emph{reverse mathematics}, ask how strong statements of second order arithmetic like this are. To be more precise: Which axioms do we need to add to our reasoning in order to prove, e.g, the termination of Goodstein sequences? Of course, working in $\zfc$ does not require any additional axioms since the fact that $\varepsilon_0$ exists as an ordinal is a standard result of most common set theories. Therefore, logicians in reverse mathematics consider very weak logical systems like $\aca_0$, which only consist of basic axioms about natural numbers and weak inductive principles.

Kirby and Paris could show that $\pa$ (a system that is very similar to $\aca_0$) is not strong enough to prove that Goodstein sequences, as we presented them here, terminate for arbitrary starting values (cf.~\cite{KP82}, see also \cite[Theorem~2.6]{Rathjen15} for a more general version). However, $\aca_0$ is (by far) not the strongest system that reverse mathematics provides and already $\atr_0$, the next of the so-called \emph{Big Five}, easily proves Goodstein's original result. Thus, a natural question that we may ask is the following: Is there a way to reformulate the termination of Goodstein sequences so that its proof needs even stronger principles than what $\atr_0$ does provide?\footnote{It should be noted that the following approach is not the only positive answer to this question. See \cite{ADWW20} for a different variant of Goodstein sequences, whose termination also cannot be proven in~$\atr_0$.}

In \cite{Girard81}, Girard introduced so-called \emph{dilators}. These are endofunctors on the category of ordinals $\on$ that preserve pullbacks and direct limits. One can show that there is a unique dilator $G: \on \to \on$ that naturally extends the substitution of hereditary bases: Each morphism $G(\iota_{n}^{n+1})$ realizes the substitution of base $n+1$ by $n+2$ for any positive natural number $n \in \n$, where $\iota_{n}^{n+1}$ is the canonical embedding from $n$ to $n+1$. In the following, we call $G$ the \emph{Goodstein dilator}. Thus, we can define any Goodstein sequence $(g_n)_{n \in \n}$ with starting value $m \in \n$ in the following compact form:
\begin{itemize}
	\item $g_0 := m$,
	\item $g_{n+1}$: If $g_n > 0$, then we set $g_{n+1} := G(\iota_{n+1}^{n+2})(g_n) - 1$. Otherwise, if $g_n = 0$, we set $g_{n+1} := 0$.
\end{itemize}
From here, Abrusci, Girard, and van de Wiele (cf.~\cite{Abrusci87, AGV90}) constructed a procedure that makes use of the fact that $G$ is defined on \emph{all} ordinals: Instead of subtracting~$1$ after the base change, they \emph{add} $1$. Of course, with this modification our sequence can never reach $0$. Thus, we now consider the sequence to be terminated if a certain upper bound is reached:
\begin{definition}\label{def:zfc_inverse_d_sequence}
	For every dilator $D$ and every ordinal $\alpha$, we define $(A_{\alpha + \gamma}^{D, \alpha})_{\gamma \in \on}$ the \emph{increasing $D$-sequence with respect to $\alpha$} as follows: For all $\gamma \geq \alpha$, we have
	\begin{equation*}
		A_{\gamma}^{D, \alpha} :=
		\begin{cases}
			\sup_{\alpha \leq \delta < \gamma} (D(\iota_{\delta}^{\gamma})(A_{\delta}^{D, \alpha}) + 1) & \text{ if } A_{\delta}^{D, \alpha} < D(\delta) \text{ for all } \delta \text{ with } \alpha \leq \delta < \gamma \comma\\
			D(\gamma) & \text{ otherwise,}
		\end{cases}
	\end{equation*}
	where $\iota_\delta^\gamma$ denotes the canonical embedding from $\delta$ to $\gamma$.
\end{definition}
We say that the sequence terminates if we have reached an ordinal $\gamma$ satisfying $A_{\gamma}^{D, \alpha} = D(\gamma)$. We also call the increasing $G$-sequence (where $G$ is the Goodstein dilator) the \emph{inverse Goodstein sequence}.

We define a function that yields the ordinals at which our sequences terminate:
\begin{definition}
	We define a partial function $\nu : \subseteq\dil \times \on \to \on$ such that for any dilator $D$ and ordinal $\alpha$, $\nu(D, \alpha)$ is the smallest $\gamma \geq \alpha$ satisfying $A_{\gamma}^{D, \alpha} = D(\gamma)$ if it exists.
\end{definition}
The first main result by Abrusci, Girard, and van de Wiele is the following:
\begin{theorem}[$\zfc$]\label{thm:d_seq_terminates}
	For any dilator $D$, there is some (unique) ordinal $\alpha$ such that $\nu(D, 0) = \alpha$ holds.
\end{theorem}
Moreover, they explicitly computed the point at which the inverse Goodstein sequence terminates.
\begin{proposition}[$\zfc$]\label{prop:goodstein_bachmann}
	Let $G$ be the Goodstein dilator. We have the equality $\nu(G, 0) = \bachmann$, where $\bachmann$ is the Bachmann-Howard ordinal.
\end{proposition}
Since $\bachmann$ is much larger than the proof theoretic ordinal (see \cite{rathjen99} for an introduction to this topic) of $\atr_0$, this entails that $\atr_0$ is not able to prove this result. In fact, they showed this property for $\idone$, a system that is closely related to $\picaminus$, which is a weakened version of $\pica$, the strongest of the Big Five.

From the standpoint of reverse mathematics, these results are not completely satisfying: The deep insight of this field is that most natural theorems of mathematics are (over a weak base system of $\rca_0$) equivalent to one of the Big Five: $\rca_0$, $\wkl_0$, $\aca_0$, $\atr_0$, or $\pica$.

In second order arithmetic, the termination of the inverse Goodstein sequence can be formulated as a so-called $\Pi^1_1$-statement (``for all sets, some arithmetical statement holds''). Logic tells us that, over $\rca_0$, statements of this form cannot be equivalent to any of the Big Five (except for $\rca_0$ itself). Of course, we can already rule out all four systems below $\pica$ since $\atr_0$ cannot prove the termination. However, even a $\Pi^1_2$-statement (``for all sets, there exists a set such that some arithmetical statement holds'') is not enough for an equivalence with $\pica$ (see~\cite{Marcone96}).

A natural form of statements that \emph{can} be equivalent to such stronger set existence principles as $\pica$ is given by a quantification over dilators. An instance of such a statement is the following: A certain linear order $\vartheta(D)$, called the \emph{Bachmann-Howard fixed point of $D$}, is well founded for any dilator $D$. In fact, this is equivalent to $\pica$ over $\rca_0$ (cf.~\cite{Freund18Higher, Freund18PhD, FreundCategorical, Freund19, Freund20}). Clearly, Theorem~\ref{thm:d_seq_terminates} is also a (non-trivial) statement that quantifies over dilators. Hence, it has been conjectured by Andreas Weiermann that this theorem might also be equivalent to $\pica$ over $\rca_0$. In fact, we will see that it is a consequence of this statement about well founded Bachmann-Howard fixed points (see Section~\ref{sec:set_theory}). But we will also see that these fixed points do not always give precise bounds on $\nu(D, 0)$ for dilators~$D$.

The definition of Bachmann-Howard fixed points is based, as the name suggests, on a certain notation system for the Bachmann-Howard ordinal that uses the collapsing function $\vartheta$, which was used by Rathjen and Weiermann to show the characterization of Kruskal's Theorem using the Ackermann ordinal (cf.~\cite{RW93}). However, this is not the only possible way to notate this ordinal: In \cite{Buchholz86}, Buchholz uses a system that employs a different collapsing function $\psi$. In their recent publication~\cite{FR21}, Freund and Rathjen defined $\nu$-fixed points $\psi_\nu(D)$ for well orders $\nu$ and dilators~$D$ based on this collapsing function. The well foundedness of $\psi_\nu(D)$ for arbitrary $\nu$ and $D$ manages, similar to the notation system by Buchholz, to reach ordinals that are much larger than the Bachmann-Howard ordinal. In fact, this statement is equivalent to $\pitr$ over $\rca_0$, which lies beyond even $\pica$. If we restrict ourselves to $\nu := 1$, it still is equivalent to $\pica$ over $\rca_0$. Of course, for our purposes this latter statement is more interesting.

The main result of this paper is the following: For any dilator $D$, the order $\nu(D, 0)$ is exactly the same as the $1$-fixed point of $D$. This connection between increasing $D$-sequences and $1$-fixed points for dilators $D$ is so deep that we can already prove it in $\rca_0$. Thus, we can show that (a variant in second order arithmetic of) Theorem~\ref{thm:d_seq_terminates} is equivalent to $\pica$ over $\rca_0$. However, this connection goes beyond the realm of reverse mathematics and yields a general concise way to explicitly notate ordinals that are produced by increasing $D$-sequences.

It should also be noted that this article is only one of many works that profit from the ideas of abstract collapsing principles such as Bachmann-Howard fixed points and $\nu$-fixed points (cf.~\cite{aguilera2022functorial, FreundPatterns, FreundKruskalFriedman, Freund22, FRW22}).
Let us finish with a short outline of this paper:

In Section \ref{sec:set_theory}, we give two very different proofs of Theorem \ref{thm:d_seq_terminates}. The first proof uses Fodor's lemma, a useful result from set theory about regressive functions from stationary sets into uncountable regular cardinals (cf.~\cite{Fodor56}). The second proof uses Bachmann-Howard fixed points.

In Section \ref{sec:second_order_arithmetic}, we give a definition that faithfully represents increasing $D$-sequences in second order arithmetic.

In Section \ref{sec:one_fixed_points}, we show that this definition bears a strong connection with $1$-fixed points. Here, we also prove that a formulation of Theorem \ref{thm:d_seq_terminates} in the context of reverse mathematics is equivalent to $\pica$ over $\rca_0$.

In Section \ref{sec:inverse_Goodstein}, we use the correspondence between increasing $D$-sequences and $1$-fixed points in order to calculate $\nu(G, 0) = \bachmann$ for the Goodstein dilator $G$. Here, we also use connections between Bachmann-Howard fixed points and $1$-fixed points in order to give purely syntactical arguments for the equivalence of different notation systems of the Bachmann-Howard ordinal.

In Section \ref{sec:weak_inverse_Goodstein}, we consider a weak variant of the inverse Goodstein sequence and calculate $\nu(D, 0) = \varphi_\omega(0)$ for the corresponding dilator $D$. We use similar methods from the previous sections and purely syntactical arguments.

In Section \ref{sec:uniqueness}, we closely inspect properties of our second order definition for increasing $D$-sequences and study the strength of the uniqueness for variations of this definition. The main result of this section is that already the uniqueness of so-called \emph{regular} increasing $D$-sequences for all dilators $D$ is equivalent to $\pica$ over $\rca_0$.

\subsection*{Acknowledgments} The idea for this project and the conjecture that there might be a connection between increasing $D$-sequences and Bachmann-Howard fixed points which leads to an equivalence of their termination and $\Pi^1_1$-comprehension is due to Andreas Weiermann.

I would also like to thank my advisor Anton Freund and all participants of \emph{Trends in Proof Theory 2023} at the university of Ghent, where I had the opportunity to present early results. This work has been funded by the Deutsche Forschungsgemeinschaft (DFG, German Research Foundation) -- Project number~460597863.

\section{Results in set theory}\label{sec:set_theory}

As it turns out, later values of $\nu$, i.e.~$\nu(D, \alpha)$ for a dilator $D$ and $\alpha > 0$, can already be computed from $\nu(E, 0)$ for the right choice of $E$. In fact, we have the following correspondence (cf.~\cite[Proposition~1.3]{AGV90}):
\begin{equation*}
	\nu(D, \alpha) = \alpha + \nu(D \circ (\alpha + \id), 0)\period
\end{equation*}
With this in mind, our studies reduce to the behavior of $\nu(D, 0)$ for dilators $D$.
In order to get a first intuition for increasing $D$-sequences, we begin this section with the computation of $\nu(D, 0)$ for some simple (families of) dilators $D$:
\begin{example}\label{ex:various_d_sequence_fixed_points}\mbox{}
	\begin{enumerate}[label=(\roman*)]
		\item $D$ with $D(0) = 0$.
				
		$A_{0}^{D, 0} = \sup_{0 \leq \delta < 0} (D(\iota_{\delta}^{0})(A_{\delta}^{D, 0}) + 1) = 0 = D(0)$.\\
		Thus, we conclude $\nu(D, 0) = 0$. Using the same argument, we can also see that the first member of any $D$-sequence is equal to $0$.
		\item $D \equiv \beta$ for $\beta \in \on$, i.e., $D(\gamma) = \beta$ and $D(\iota_{\delta}^{\gamma}) = \iota_{\beta}^{\beta}$ for all $\delta < \gamma \in \on$.
		
		We show $A_{\gamma}^{D, 0} = \gamma$ for all $\gamma \leq \beta$. Assume that this claim has already been proven for all $\delta < \gamma$. Then, we have the following:
		\begin{equation*}
			A_{\gamma}^{D, 0} = \sup_{\delta < \gamma} (D(\iota_{\delta}^{\gamma})(A_{\delta}^{D, 0}) + 1) = \sup_{\delta < \gamma} (A_{\delta}^{D, 0} + 1) = \sup_{\delta < \gamma} (\delta + 1) = \gamma\period
		\end{equation*}
		We conclude $A_{\gamma}^{D, 0} = \gamma < \beta = D(\gamma)$ for all $\gamma < \beta$, but for $\gamma := \beta$, we have $A_{\beta}^{D, 0} = \beta = D(\beta)$. Thus, $\nu(D, 0) = \beta$.
		\item $D = D_1 + D_2$ with $D_1(0) = 0$ and under the assumption that $\nu(D_2, 0)$ exists.
		
		Let $\alpha := \nu(D_2, 0)$. We show inductively that for all $\gamma < \alpha$, we have $A_{\gamma}^{D, 0} = D_1(\gamma) + A_{\gamma}^{D_2, 0}$. Let $\gamma \leq \alpha$ and assume that our claim has already been shown for all $\delta < \gamma$. Then, we have:
		\begin{align*}
			A_{\gamma}^{D, 0} &= \sup_{\delta < \gamma} (D(\iota_{\delta}^{\gamma})(A_{\delta}^{D, 0}) + 1) = \sup_{\delta < \gamma} ((D_1 + D_2)(\iota_{\delta}^{\gamma})(D_1(\delta) + A_{\delta}^{D_2, 0}) + 1)\\
			&= \sup_{\delta < \gamma} (D_1(\gamma) + D_2(\iota_{\delta}^{\gamma})(A_{\delta}^{D_2, 0}) + 1) = D_1(\gamma) + A_{\gamma}^{D_2, 0}\period
		\end{align*}
		We conclude $A_{\gamma}^{D, 0} < D(\gamma)$ for all $\gamma < \alpha$ but for $\gamma := \alpha$, we have an equality $A_{\alpha}^{D, 0} = D(\alpha)$. Thus, $\nu(D, 0) = \nu(D_2, 0)$.
	\end{enumerate}
\end{example}

While the definition of dilators as endofunctors that preserve pullbacks and direct limits is nice and compact, we choose a different one that is more explicit and closer to the definition that we will use in second order arithmetic (cf.~\cite{Freund18PhD}):

\begin{definition}
	We define an endofunctor $[\cdot]^{<\omega}$ on $\set$, the category of sets, with
	\begin{align*}
		[S]^{<\omega} &:= \text{ the set of finite subsets of $S$, and}\\
		[f]^{<\omega}(F) &:= f(F)
	\end{align*}
	for sets $S$, $T$, and $F$ (where $F$ is a finite subset of $S$), and morphisms $f: S \to T$.
\end{definition}
\begin{definition}
	A dilator consists of
	\begin{enumerate}[label=(\roman*)]
		\item an endofunctor $D$ on the category of well orders $\wo$ that maps ordinals to ordinals
		\item a natural transformation $\supp^D: D \Rightarrow [\cdot]^{<\omega}$ such that any $\sigma \in D(\alpha)$ lies in the range of $D(\iota)$ for the canonical injection $\iota: \supp^D_\alpha(\sigma) \hookrightarrow \alpha$, for each well order $\alpha$.\footnote{For purposes of readability, we omit the forgetful functor $F: \wo \to \set$ that should be composed both after $D$ and before $[\cdot]^{<\omega}$ in the definition of $\supp^D$.} This property is called the \emph{support condition}.
	\end{enumerate}
\end{definition}
Now, we can give a concrete definition of the Goodstein dilator:
\begin{definition}\label{def:goodstein_dilator_set_theory}
	For the purposes of this paper, we will call the following dilator $G$ the \emph{Goodstein dilator}. Let $\alpha$ be some ordinal, then we define:
	\begin{equation*}
		G(\alpha) := (1 + \alpha)_{\omega} := \sup_{n < \omega} (1 + \alpha)_{n}\comma
	\end{equation*}
	where the subscript $n$ notates a tower of height $n$.
	Consider another ordinal $\beta$ together with an embedding $f : \alpha \to \beta$. Using the Cantor normal form, any element $\sigma \in G(\alpha)$ can uniquely be written as the following sum:
	\begin{equation*}
		{(1 + \alpha)^{\gamma_0} \cdot (1 + \delta_0) + \dots + (1 + \alpha)^{\gamma_{n-1}} \cdot (1 + \delta_{n-1})}
	\end{equation*}
	for ${n \in \n}$, ${\gamma_i \in G(\alpha)}$, and ${\delta_i \in \alpha}$ for all ${i < n}$ such that $\gamma_0 > \dots > \gamma_{n-1}$ holds. Now, the morphism $G(f)$ maps $\sigma$ to
	\begin{equation*}
		{(1 + \beta)^{G(f)(\gamma_0)} \cdot (1 + f(\delta_0)) + \dots + (1 + \beta)^{G(f)(\gamma_{n-1})} \cdot (1 + f(\delta_{n-1}))}\period
	\end{equation*}
	The support of $\sigma$ is given by
	\begin{equation*}
		\supp^G_\alpha(\sigma) := \bigcup \{\supp^G_\alpha(\gamma_i) \cup \{\delta_i\} \mid i < n\}\period
	\end{equation*}
	Finally, our definition can naturally be extended from ordinals to well orders.
\end{definition}

We continue with the first proof of Theorem \ref{thm:d_seq_terminates} using \emph{Fodor's Lemma}. This proof is already mentioned by Abrusci, Girard, and van de Wiele in both works \cite{Abrusci87} and~\cite{AGV90}. However, they do not give any details. The following construction is due to Anton Freund.
We recall some standard notions from set theory that are required for the formulation of Fodor's Lemma:
\begin{definition}[Prerequisites for Fodor's Lemma]\mbox{}
	\begin{itemize}
		\item A cardinal $\kappa$ is called \emph{regular} if and only if any embedding $f: \alpha \to \kappa$ with $\sup_{\beta \in \alpha} f(\beta) = \kappa$ already implies $\alpha = \kappa$.
		\item Given an ordinal $\kappa$, we call $C \subseteq \kappa$ a \emph{club} subset of $\kappa$ if and only if it is both closed and unbounded:
		\begin{itemize}
			\item $C$ is \emph{closed} if for any subset $D \subseteq C$ with $\sup(D) \in \kappa$ we already have $\sup(D) \in C$.
			\item $C$ is \emph{unbounded} if for any element $\alpha \in \kappa$, we can find some $\beta$ with $\alpha < \beta < \kappa$ and $\beta \in C$.
		\end{itemize}
		\item Given a cardinal $\kappa$, we call $S \subseteq \kappa$ a \emph{stationary} subset of $\kappa$ if and only if its intersection with any club subset of $\kappa$ is non-empty.
		\item A function $f: S \to T$ for subsets $S$ and $T$ of the ordinals is called \emph{regressive} if and only if $f(\alpha) < \alpha$ holds for every $\alpha \in S \setminus \{0\}$.
	\end{itemize}
\end{definition}
\begin{lemma}[$\zfc$, Fodor's Lemma]
	Given an uncountable and regular cardinal~$\kappa$, a stationary subset $S \subseteq \kappa$, and a regressive function $f: S \to \kappa$, we can find a stationary subset $S'$ of $\kappa$ with $S' \subseteq S$ and a value $\delta \in \kappa$ such that $f(\alpha) = \delta$ holds for any $\alpha \in S'$.
\end{lemma}
For a proof of this lemma, we refer to Fodor's original result \cite[Satz~2]{Fodor56}.
\begin{lemma}[$\zfc$]\label{lem:dilator_below_kappa}
	Let $D$ be a dilator. Then, there exists some cardinal $\kappa$ such that for any regular cardinal $\lambda \geq \kappa$, we have that $D(\alpha) < \lambda$ holds for any $\alpha < \lambda$.
\end{lemma}
\begin{proof}
	Let $\kappa$ be an uncountable cardinal with $D(\omega) < \kappa$. We show that this choice for $\kappa$ satisfies our claim.
	
	Let $\lambda \geq \kappa$ be some regular cardinal and assume $\alpha < \lambda$. Recalling the support condition for dilators, for any element $\gamma \in D(\alpha)$ there is some natural number $n \in \n$ together with an embedding $f : n \to \alpha$ such that there exists a unique $\sigma \in D(n)$ with $D(f)(\sigma) = \gamma$. Therefore, in order to find some upper bound on the cardinality of $D(\alpha)$, we simply count the amount of choices for $n$, $f$, and $\sigma$. Since such embeddings $f$ correspond to finite (descending) sequences of $\alpha$, and $D(n) \leq D(\omega)$ holds for any $n \in \n$, we have
	\begin{equation*}
		| D(\alpha) | \leq | D(\omega) | \times | \alpha^\omega |\period
	\end{equation*}
	By assumption, we have $D(\omega) < \kappa \leq \lambda$. Moreover, $\alpha^\omega < \lambda$ follows from $\alpha < \lambda$ since our regular $\lambda$ contains $\omega$ and is closed under products. We arrive at $|D(\alpha)| < \lambda$. Thus, we conclude $|D(\alpha)| < \lambda$ and, since $\lambda$ is a cardinal, $D(\alpha) < \lambda$.
\end{proof}

\begin{proof}[Proof of Theorem \ref{thm:d_seq_terminates} using Fodor's Lemma]
	As $\nu(D, 0)$ maps to the smallest possible value $\alpha$ with ${A_{\alpha}^{D, 0} = D(\alpha)}$ if it is defined, the uniqueness part is clear.
	
	Assume that $D$ is a dilator and, using Lemma \ref{lem:dilator_below_kappa}, let $\kappa$ be an uncountable regular cardinal with $D(\gamma) < \kappa$ for all $\gamma < \kappa$. Furthermore, assume that $A_{\gamma}^{D, 0} < D(\gamma)$ holds for any $\gamma < \kappa$. In the following, we will lead these assumptions to a contradiction.
	
	Let $\Lim_{\kappa}$ denote the club set of all limit ordinals below $\kappa$. We define a function $f : \Lim_{\kappa} \to \kappa$ such that $f(\gamma)$ is the smallest ordinal greater than any element of $\supp_{\gamma}(A_{\gamma}^{D, 0})$. We show that $f$ is regressive, which also makes sure that $f$ is well defined in the sense that it maps below $\kappa$.
	
	If the support of $A_{\gamma}^{D, 0}$ is empty, then $f(\gamma)$ maps to $0$. Since any element in $\Lim_{\kappa}$ must be greater than $0$, we know that $f$ is regressive for such $\gamma$. If $\supp_{\gamma}(A_{\gamma}^{D, 0})$ is not empty, then $f(\gamma)$ is equal to $\delta + 1$ where $\delta$ is the biggest element of this set. By definition of the support, we have $\delta < \gamma$. Moreover, since $\gamma$ is a limit ordinal, we even have $f(\gamma) = \delta + 1 < \gamma$. We conclude that $f$ is regressive on its whole domain.
	
	Now, we apply Fodor's Lemma. This yields some element $\delta \in \kappa$ such that $S := f^{-1}(\delta) \subseteq \Lim_{\kappa}$ is stationary with respect to $\kappa$. By definition of $f$, we have $\supp_{\gamma}(A_{\gamma}^{D, 0}) \subseteq \delta$ for any $\gamma \in S$. For our purposes, we are only interested in the fact that $S$ is unbounded in $\kappa$, i.e., for any $\gamma \in \kappa$, there is some $\gamma' \in S$ with $\gamma' > \gamma$. In order to make a later argument simpler, we restrict ourselves to $S' := S \setminus \delta$. Clearly, this subset of $S$ is still unbounded in $\kappa$. Using the fact that $\kappa$ is regular, we conclude that there is an embedding $g : \kappa \to S'$. By definition of $f$ and $S$, this embedding satisfies $\supp_{g(\gamma)}(A_{g(\gamma)}^{D, 0}) \subseteq \delta$ for any $\gamma \in \kappa$.
	
	Using the support condition, we can show that any of the elements in our $D$-sequence with index in $S'$ can be mapped to by something that lives in $D(\delta)$. The important detail of our argument is that this can be done in an order-preserving manner.
	
	In order to make this precise, let $h : S' \to D(\delta)$ be the function that maps any index $\gamma \in S'$ to the unique element $h(\gamma) \in D(\delta)$ with $D(\iota_{\delta}^{\gamma})(h(\gamma)) = A_{\gamma}^{D, 0}$. Such a value $h(\gamma)$ always exists because of the support condition in combination with $\supp_{\gamma}(A_{\gamma}^{D, 0}) \subseteq \delta$. For the existence of $\iota_{\delta}^{\gamma}$, we use the fact that, by definition, any considered element $\gamma \in S'$ is greater or equal to $\delta$.
	
	We want to show that $h$ is an embedding. Let $\gamma, \gamma' \in S'$ with $\gamma < \gamma'$ and assume $h(\gamma) \geq h(\gamma')$. Then, we have
	\begin{equation*}
		A^{D, 0}_{\gamma'} > D(\iota_\gamma^{\gamma'})(A^{D, 0}_\gamma) = D(\iota_\delta^{\gamma'})(h(\gamma)) \geq D(\iota_\delta^{\gamma'})(h(\gamma')) = A^{D, 0}_{\gamma'}\comma
	\end{equation*}
	which clearly is a contradiction.
	
	For the final steps of this proof, we combine $g$ and $h$ into ${h \circ g : \kappa \to D(\delta)}$. Recall that at the beginning, $\kappa$ was chosen to satisfy $D(\delta) < \kappa$ for $\delta < \kappa$. This renders the existence of an embedding like $h \circ g$ impossible. Because of this contradiction, we conclude that there must have been some ordinal $\alpha$ (even below $\kappa$) satisfying the equality $A_{\alpha}^{D, 0} = D(\alpha)$.
\end{proof}

\begin{remark}
	This proof also gives us a first upper bound on $\nu(D, 0)$ for a dilator~$D$: We have $\nu(D, 0) < \kappa$ where $\kappa$ is the least uncountable regular cardinal greater than $D(\omega)$. For the Goodstein dilator $G$ (and any other dilator that we consider in second order arithmetic), this entails that $\nu(G, 0)$ is countable.
\end{remark}

The next proof makes use of so-called \emph{Bachmann-Howard fixed points}. An argument using this technique has been conjectured by Weiermann.
The concept was first introduced by Freund (cf.~\cite{Freund18Higher, Freund18PhD, FreundCategorical, Freund19, Freund20}). While he mostly uses these fixed points in the context of reverse mathematics, he also shows that they exist in ``regular'' set theory. With our proof, we give further evidence that Bachmann-Howard fixed points are a valuable tool outside of second order arithmetic.
\begin{definition}
	Given a dilator $D$, we call an ordinal $\alpha$ a \emph{Bachmann-Howard fixed point} if there exists a map
	\begin{equation*}
		\vartheta: D(\alpha) \to \alpha
	\end{equation*}
	satisfying the following for any $\sigma, \tau \in D(\alpha)$:
	\begin{enumerate}[label=(\roman*)]
		\item If $\sigma < \tau$ and $\supp^D_\alpha(\sigma) \subseteq \vartheta(\tau)$, then we have $\vartheta(\sigma) < \vartheta(\tau)$.
		\item We have $\supp^D_\alpha(\sigma) \subseteq \vartheta(\sigma)$.
	\end{enumerate}
	We call $\vartheta$ a \emph{Bachmann-Howard collapse}.
\end{definition}

\begin{theorem}[$\zfc$]
	For any dilator $D$, there is a Bachmann-Howard fixed point.
\end{theorem}
For a proof of this theorem, we refer to Section 2 of \cite{Freund19}.
\begin{proof}[Proof of Theorem \ref{thm:d_seq_terminates} using Bachmann-Howard fixed points]
	Similar to the previous proof, the uniqueness of $\nu(D, 0)$ is clear. Let $\alpha$ be a Bachmann-Howard fixed point of $D$ with collapse $\vartheta: D(\alpha) \to \alpha$. We define a sequence $(\sigma_\beta)_{\beta < \alpha}$ in $D(\alpha)$ with $\sigma_\beta := D(\iota_\beta^\alpha)(A^D_\beta)$ for each $\beta < \alpha$.
	In order to derive a contradiction, we assume that $A^D_\beta < D(\beta)$ holds for any ordinal $\beta$. In particular, $\sup_{\beta < \alpha}(\sigma_\beta + 1) =: \tau < D(\alpha)$.
	
	The first important properties of $(\sigma_\beta)_{\beta < \alpha}$ are that it is a strictly increasing sequence and that we have $\supp^D_\alpha(\sigma_\beta) \subseteq \beta$ for any $\beta < \alpha$. For the former, we simply apply the definition:
	\begin{equation*}
		\sigma_\beta = D(\iota_\beta^\alpha)(\sup_{\gamma < \beta}(D(\iota_\gamma^\beta)(A^D_\gamma) + 1)) > D(\iota_\gamma^\alpha)(A^D_\gamma) = \sigma_\gamma
	\end{equation*}
	holds for $\gamma < \beta < \alpha$. The latter is given by $\supp^D_\alpha(\sigma_\beta) = \supp^D_\beta(A^D_\beta) \subseteq \beta$.
	
	Next, we put the sequence through our collapse and show that its order is preserved. For a proof by contradiction, assume that $\beta < \alpha$ is the smallest ordinal such that there exists a $\gamma < \beta$ with $\vartheta(\sigma_\gamma) \geq \vartheta(\sigma_\beta)$. Given $\beta$, we also choose $\gamma$ to be the smallest ordinal with this property. Because of the minimality of $\beta$, we know that $\rho \mapsto \vartheta(\sigma_\rho)$ for $\rho < \beta$ is an embedding. Thus, we have $\vartheta(\sigma_\rho) \geq \rho$ for any $\rho < \beta$. In particular, this holds for any element in $\supp^D_\alpha(\sigma_\gamma)$. Using the minimality of $\gamma$, we conclude $\rho \leq \vartheta(\sigma_\rho) < \vartheta(\sigma_\beta)$ for any $\rho \in \supp^D_\alpha(\sigma_\gamma) \subseteq \gamma$. Combining this with $\sigma_\gamma < \sigma_\beta$ yields the contradiction $\vartheta(\sigma_\gamma) < \vartheta(\sigma_\beta)$. Thus, $\beta \mapsto \vartheta(\sigma_\beta)$ for $\beta < \alpha$ is an embedding.
	
	Consider $\tau \in D(\alpha)$ with $\sigma_\beta < \tau$ for all $\beta < \alpha$ from the beginning. In particular, we have $\sigma_{\vartheta(\tau)} < \tau$ and, since every element in the support of the left hand side lies below $\vartheta(\tau)$, also $\vartheta(\sigma_{\vartheta(\tau)}) < \vartheta(\tau)$. However, since $\beta \mapsto \vartheta(\sigma_\beta)$ for $\beta < \alpha$ is an embedding, we have $\vartheta(\tau) \leq \vartheta(\sigma_{\vartheta(\tau)})$. Clearly, this is a contradiction.
\end{proof}

\begin{remark}\label{rem:bachmann_not_exact_introduction}
	This proof has the potential to give much better upper bounds on $\nu(D, 0)$ for dilators $D$ than the previous proof since the smallest Bachmann-Howard fixed point $\alpha$ depends on the whole dilator $D$ (instead of only $D(\omega)$). We have $\nu(D, 0) \leq \alpha$ for the (smallest such) fixed point $\alpha$. Notice that inequality is not strict this time. As we will see later, this proof already entails $\nu(G, 0) \leq \bachmann$ for the Goodstein dilator $G$ where $\bachmann$ is the Bachmann-Howard ordinal. Still, there are cases in which the fixed point is strictly larger than $\nu(D, 0)$. An example for this is the \emph{weak Goodstein dilator} $W$ that we will study in Section~\ref{sec:weak_inverse_Goodstein}. We will see that $\nu(W, 0) = \varphi_\omega(0)$ holds, whereas even the smallest Bachmann-Howard fixed point of $W$ cannot be smaller than $\vartheta(\Omega^\omega)$, which is much larger.
\end{remark}
We can also already state the following corollary. However, for now this result comes without an explicit value for $\nu(G, 0)$ where $G$ is the Goodstein dilator.
\begin{corollary}[$\zfc$]
	The inverse Goodstein sequence terminates.
\end{corollary}

\begin{proof}
	Apply Theorem \ref{thm:d_seq_terminates} to the Goodstein dilator.
\end{proof}

\section{Inverse $D$-sequences in second order arithmetic}\label{sec:second_order_arithmetic}
From now on, we are working in second order arithmetic. To be more precise: the system of $\rca_0$ from reverse mathematics. The standard literature on this topic is \cite{Simpson09} by Simpson. For the definition of dilators in second order arithmetic, we first consider the more general notion of \emph{predilators}, that are defined on linear orders:
\begin{definition}
	A predilator consists of
	\begin{enumerate}[label=(\roman*)]
		\item an endofunctor $D$ on the category of linear orders $\lo$ and
		\item a natural transformation $\supp^D: D \Rightarrow [\cdot]^{<\omega}$ such that any $\sigma \in D(X)$ lies in the range of $D(\iota)$ for the canonical injection $\iota: \supp^D_X(\sigma) \hookrightarrow X$, for each linear order $X$.\footnote{Again, we are omitting forgetful functors.}
	\end{enumerate}
	If it is clear from context, we omit the superscript $D$ in $\supp^D$. A predilator is called a \emph{dilator} if it preserves well orders.
\end{definition}
The support condition, i.e.~(ii), has the effect that $D$ is already defined by its behavior on \emph{finite} linear orders. This makes it possible to represent predilators and, hence, dilators in the language of second order arithmetic. More information about this can be found in \cite[Section~2]{Freund20}. Also, it should be noted that Girard's original definition of predilators also requires a further monotonicity condition: $D(f) \leq D(g)$ is supposed to hold for morphisms $f, g: X \to Y$ with $f(x) \leq g(x)$ for all $x \in X$. In the case of dilators, this property is implied.

In the following, we will need some notation around linear orders: 
Given a linear order $X$ and two finite sets $Y, Z \subseteq X$, we write $Y <_{\fin} Z$ (or $Y \leq_{\fin} Z$) if for any element $y \in Y$ there is some $z \in Z$ with $y < z$ (or $y \leq z$). For singleton subsets, we simply write the element instead of the set, e.g.~$\{y\} <_{\fin} Z$ may be written as $y <_{\fin} Z$ for $y \in X$ and a finite $Z \subseteq X$.

For two linear orders $X$ and $Y$ with $X \subseteq Y$, we write $\iota_X^Y$ for the canonical injection from $X$ into $Y$. Given a linear order $X$ together with an element $x \in X$, we define $\restr{X}{x}$ to be the suborder of all elements in $X$ strictly below $x$. Finally, given a function $f: X \to Y$ between linear orders $X$ and $Y$ with $x \in X$, we define $\restr{f}{(\restr{X}{x})}: \restr{X}{x} \to \restr{Y}{f(x)}$ with $\restr{f}{(\restr{X}{x})}(x') := f(x')$ for $x' \in \restr{X}{x}$. Notice that we are also restricting the codomain. This will become convenient during later proofs.

We define increasing $D$-sequences in the context of second order arithmetic. Since we do not have the notion of an ordinal in this setting (and also want to extend this idea to linear orders), we cannot simply define $\nu(D, 0)$ for (pre)dilators~$D$ in~$\rca_0$. Instead, we give linear orders that satisfy properties similar to that of $\nu(D, 0)$ a name: termination point.
\begin{definition}\label{def:soa_d_sequence_fixed_point}
	Let $D$ be a predilator. We call the linear order $X$ together with a sequence $(A_{x}^{D})_{x \in X}$ with $A_{x}^{D} \in D(\restr{X}{x})$ for any $x \in X$ a \emph{$D$-sequence termination point} if and only if the following are satisfied:
	\begin{enumerate}
		\item $A_{x}^{D}$ is the smallest value in $D(\restr{X}{x})$ greater than any of $D(\iota_{\restr{X}{y}}^{\restr{X}{x}})(A_{y}^{D})$ for~$y <_{X} x$.
		\item For any element $\sigma \in D(X)$, there is some $x \in X$ with $\sigma \leq D(\iota_{\restr{X}{x}}^{X})(A_{x}^{D})$.
	\end{enumerate}
	We call a $D$-sequence termination point \emph{regular} if and only if the following holds additionally:
	\begin{enumerate}
		\item[(3)] Consider the relation $\prec$ that is defined by
		\begin{equation*}
			x \prec y \quad :\Longleftrightarrow \quad x \in \supp_{\restr{X}{y}}(A_{y}^{D})
		\end{equation*}
		for any $x, y \in X$. We require that this relation comes with a height function $h: X \to \n$ that satisfies $h(x) < h(y)$ for any $x, y \in X$ with $x \prec y$.
	\end{enumerate}
	Moreover, we call a $D$-sequence termination point \emph{strong} if and only if it is regular and satisfies the following strengthening of condition (2):
	\begin{enumerate}
		\item[(4)] For any element $\sigma \in D(X)$, there is a smallest $x \in X$ with $\supp_{X}(\sigma) \leq_{\fin} x$ and $\sigma \leq D(\iota_{\restr{X}{x}}^{X})(A_{x}^{D})$.
		
	\end{enumerate}
\end{definition}
While it might already be clear how conditions (1) and (2) faithfully represent increasing $D$-sequences from set theory, conditions (3) and (4) appear to be rather ad hoc. Later, it will become apparent that they are required for a one-to-one correspondence with $1$-fixed points. For the moment, we can just accept them since they always hold in the most important case when the underlying order $X$ is well founded:

\begin{lemma}[$\rca_0$]\label{lem:wf_fp_strong}
	Given some predilator $D$, any well founded $D$-sequence termination point is already strong.
\end{lemma}

\begin{proof}
	Let $X$ be some well founded $D$-sequence termination point. For requirement (3), we first observe that $x \prec y$ implies $x < y$ for any $x, y \in X$ since $\supp_{\restr{X}{y}}(A_{y}^{D}) <_{\fin} y$ holds. Therefore, $\prec$ must be well founded. We proceed with the definition of our height function $h: X \to \n$. Note that this can be done quite easily by considering the tree that the relation $\prec$ induces for an arbitrary element and applying bounded K\H{o}nig's Lemma (cf.~\cite[Definition IV.1.3 and Lemma IV.1.4]{Simpson09}), which is available in $\wkl_0$, to this structure. However, we want to show that our claim can already be proven in $\rca_0$ using an argument that is quite similar to that of the proof for \cite[Proposition~19]{Uftring23}:
	
	Recall that supports of dilators are given by finite sets that we internally code as natural numbers. Thus, we can effectively reason about the elements of any support. Let $f: X \times \n \to \n$ be the following primitive recursive function:
	\begin{equation*}
		f(x, n) :=
		\begin{cases}
			x & \text{if } n = 0\comma\\
			\max \{f(x, n-1), l \mid \text{$k, l \in \n$, $k \leq f(x, n-1)$, and $l \prec k$} \} & \text{otherwise.}
		\end{cases}
	\end{equation*}
	Note that this function gives the \emph{code} of $x$ in the case $n = 0$ and interprets $\prec$ as a relation on natural numbers in the case $n > 0$. Using a simple induction, we can show that for any sequence $s \in X^*$ that begins with $x \in X$ and descends with respect to $\prec$, the value of $f(x, |s| - 1)$ is an upper bound to the code of any member in $s$. Assuming a reasonable coding of sequences, this can be turned into a function $g: X \to \n$ such that $g(x, n)$ is an upper bound to the codes of all sequences $s \in X^*$ of length $n$ that begin with $x \in X$ and descend with respect to $\prec$.
	
	Let $h: X \to \n$ be the program that for each $x \in X$ searches for and yields the least number $n \in \n$ such that there is no sequence $s \in X^*$ of length $n$ that begins with $x$ and descends with respect to $\prec$. For each potential candidate $n$, this program simply searches for such a sequence by considering all codes below $g(x, n)$. In the case where $n$ is minimal such that no sequence with this property can be found, it terminates and gives $n$ as its result.
	
	We prove that $h$ must always terminate: Assume that it does not. Then, there must be some $x \in X$ such that we can perform the following construction: Let $(x_n)_{n \in \n}$ be the sequence such that, for any $n \in \n$, $x_n$ is the smallest element in $X$ (with respect to the order on $X$) such that there exists a sequence $s \in X^*$ of length $n$ that begins with $x_n$ and descends with respect to $\prec$. With the help of $g$, we can effectively search for all such sequences and simply produce the minimum of their first member. Now, we claim that this yields an infinite descending sequence in $X$: Given an arbitrary number $n \in \n$, we know that there must be a sequence $s$ that has length $n+1$, starts with $x_{n+1}$, and descends with $\prec$. Clearly, the sequence $s'$ that is identical to $s$ but misses its first element has length $n$, descends in $\prec$, and begins with an element $y \in X$ such that $y \prec x_{n+1}$ and, thus, $y < x_{n+1}$ holds. With $x_n \leq y$, which clearly holds by definition of $(x_n)_{n \in \n}$, this results in a contradiction and, hence, our claim that $h$ must terminate.
	
	The final step in our proof of requirement (3) is to show the main property that we demand of $h$: For any $x, y \in X$ with $x \prec y$, we want that $h(x) < h(y)$ holds. Consider a sequence that realizes the value of $h(x)$, i.e., let $s \in X^*$ be a sequence of length $h(x) - 1$ that begins with $x$ and descends with respect to $\prec$. Now, we can extend $s$ into a sequence $s'$ that begins with $y$ and continues with $s$. Clearly, this sequence witnesses that $h(x) < h(y)$ must hold.
	
	For requirement (4), we first use (2) and assume that $x \in X$ satisfies the inequality $\sigma \leq D(\iota_{\restr{X}{x}}^{X})(A_{x}^{D})$. Now, if $\supp_{X}(\sigma) \leq_{\fin} x$ does not hold, let $y$ be the greatest element of $\supp_{X}(\sigma)$. Clearly, we have $x <_{X} y$. Using (1), we know that $A_{y}^{D}$ is greater than $D(\iota_{\restr{X}{x}}^{\restr{X}{y}})(A_{x}^{D})$. Thus, we have
	\begin{equation*}
		\sigma \leq D(\iota_{\restr{X}{x}}^{X})(A_{x}^{D}) = D(\iota_{\restr{X}{y}}^{X})(D(\iota_{\restr{X}{x}}^{\restr{X}{y}})(A_{x}^{D})) < D(\iota_{\restr{X}{y}}^{X})(A_{y}^{D})\period
	\end{equation*}
	Finally, since $X$ is well founded, we know that there is a smallest such value.
\end{proof}
Next, we make the vague feeling that our definition correctly represents increasing $D$-sequences concrete. First, it should be clear that for any (second order arithmetic) dilator $D$ on linear orders, there is a (set-theoretic) dilator $\tilde{D}$ on well orders such that we can find a natural isomorphism $\eta: \restr{D}{\wo} \Rightarrow \tilde{D}$. This also holds in the other direction if we require that $\tilde{D}(n)$ is countable for every $n \in \n$, i.e., for every set theoretic dilator $\tilde{D}$ (with this restriction on $\tilde{D}$ for all $n \in \n$), there is a second order arithmetic dilator $D$ such that $\eta$ exists.

\begin{lemma}[$\zfc$]\label{lem:zfc_fp_is_soa_fp}
	Given a (second order arithmetic) dilator $D$, consider its set-theoretic realization $\tilde{D}$. We abbreviate $\alpha := \nu(\tilde{D}, 0)$.
	There is a sequence $(A^D_\gamma)_{\gamma \in \alpha}$ realizing that $\alpha$ is a strong $D$-sequence termination point in the sense of Definition~\ref{def:soa_d_sequence_fixed_point}.
	Moreover, for any $\gamma \in \alpha$, the order represented by $A^D_\gamma$, i.e.~$\restr{D(\gamma)}{A^D_\gamma}$, is isomorphic to $A^{\tilde{D}, 0}_\gamma$.
\end{lemma}

\begin{proof}
	We write $\eta: \restr{D}{\wo} \Rightarrow \tilde{D}$ for the natural isomorphism that connects the second order representation $D$ with its set-theoretic counterpart $\tilde{D}$. We define $A^D_\gamma := \eta^{-1}_\gamma(A^{\tilde{D}, 0}_\gamma)$ for all $\gamma < \alpha$. In light of Lemma \ref{lem:wf_fp_strong}, we only have to show that this constitutes a $D$-sequence termination point, i.e., conditions (3) and (4) can be skipped.
	
	For condition (1), we assume that $\gamma < \alpha$ is the smallest value such that $A^D_\gamma$ violates the requirements. First, assume that there exists some $\delta < \gamma$ such that $A^D_\gamma \leq D(\iota_\delta^\gamma)(A^D_\delta)$ holds. We conclude $A^{\tilde{D}, 0}_\gamma = \eta_\gamma(A^D_\gamma) \leq \eta_\gamma(D(\iota_\delta^\gamma)(A^D_\delta)) = \tilde{D}(\iota_\delta^\gamma)(\eta_\delta(A^D_\delta)) = \tilde{D}(\iota_\delta^\gamma)(A^{\tilde{D}, 0}_\delta) < \tilde{D}(\iota_\delta^\gamma)(A^{\tilde{D}, 0}_\delta) + 1$. Thus, we arrive at the contradiction that already $A^{\tilde{D}, 0}_\gamma$ was not the correct supremum to all $D(\iota_\delta^\gamma)(A^{\tilde{D}, 0}_\delta)$ for~$\delta < \gamma$.
	
	Otherwise, assume that there is a value $\sigma \in D(\gamma)$ strictly smaller than $A^D_\gamma$ with $\sigma > D(\iota_\delta^\gamma)(A^D_\delta)$ for all $\delta < \gamma$. We conclude $\eta_\gamma(\sigma) > \eta_\gamma(D(\iota_\delta^\gamma)(A^D_\delta)) = \tilde{D}(\iota_\delta^\gamma)(\eta_\delta(A^{\tilde{D}, 0}_\delta)) = \tilde{D}(\iota_\delta^\gamma)(A^D_\delta)$. This entails $\eta_\gamma(\sigma) \geq \tilde{D}(\iota_\delta^\gamma)(A^D_\delta) + 1$. Moreover, $\eta_\gamma(\sigma) < \eta_\gamma(A^D_\gamma) = A^{\tilde{D}, 0}_\gamma$ holds. Again, we are left with the contradiction that $A^{\tilde{D}, 0}_\gamma$ is not the correct supremum to all $D(\iota_\delta^\gamma)(A^{\tilde{D}, 0}_\delta)$ for $\delta < \gamma$.
	
	For condition (2), let $\sigma \in D(\alpha)$ be arbitrary. Since $A^{\tilde{D}, 0}_\alpha = D(\alpha)$, there is some $\gamma < \alpha$ with $\eta_\alpha(\sigma) < \tilde{D}(\iota_\gamma^\alpha)(A^{\tilde{D}, 0}_\gamma) + 1 \leq \tilde{D}(\iota_\gamma^\alpha)(\eta_\gamma(A^D_\gamma)) + 1 = \eta_\alpha(D(\iota_\gamma^\alpha)(A^D_\gamma)) + 1$. Since $\eta_\alpha$ is an embedding, this entails our claim $\sigma \leq D(\iota_\gamma^\alpha)(A^D_\gamma)$.
\end{proof}
Later, in Corollary \ref{cor:strong_fp_unique}, we will see that strong $D$-sequence termination points are unique. Thus, we also get the reversal of Lemma \ref{lem:zfc_fp_is_soa_fp}, i.e., in $\zfc$, any $D$-sequence termination point is isomorphic to $\nu(\tilde{D}, 0)$, where $\tilde{D}$ is the set-theoretic realization of the coded dilator~$D$.

The last notion that we introduce in this section is that of a homomorphism between two $D$-sequences. This will be of importance when we talk about the uniqueness of termination points since this property can, of course, only ever hold up to isomorphism.
\begin{definition}\label{def:d_seq_homomorphism}
	Given a predilator $D$ and two linear orders $X$ and $Y$, we call a map $f: X \to Y$ a \emph{$D$-sequence homomorphism} from a termination point realized by $(A^D_x)_{x \in X}$ to one given by $(B^D_y)_{y \in Y}$ if and only if $f$ is an embedding and satisfies the equality
	\begin{equation*}
		D(\restr{f}{(\restr{X}{x})})(A^D_x) = B^D_{f(x)}
	\end{equation*}
	for any $x \in X$. We call $f$ a $D$-sequence \emph{isomorphism} if and only if $f$ is an isomorphism of linear orders.
\end{definition}
With the final corollary of this section, we show that this notion of isomorphism is the right one:
\begin{corollary}[$\rca_0$]
	Consider a predilator $D$ and two linear orders $X$ and $Y$ with $D$-sequence termination points $(A^D_x)_{x \in X}$ and $(B^D_y)_{y \in Y}$ such that there is an isomorphism~${f: X \to Y}$ from the former termination point to the latter. Then, $f^{-1}$ is also an isomorphism from the latter termination point to the former one.
\end{corollary}

\begin{proof}
	First, it is clear that $f^{-1}: Y \to X$ is an isomorphism of linear orders. Thus, we only have to show that the property $D(\restr{f^{-1}}{(\restr{Y}{y})})(B^D_y) = A^D_{f^{-1}(y)}$ holds for every $y \in Y$.
	
	Let $y \in Y$ be arbitrary and define $x \in X$ to be the unique element with $f(x) = y$. Then, we have
	\begin{align*}
		D(\restr{f^{-1}}{(\restr{Y}{y})})(B^D_y) &= D(\restr{f^{-1}}{(\restr{Y}{y})})(B^D_{f(x)})\\
		&= D(\restr{f^{-1}}{(\restr{Y}{y})} \circ \restr{f}{(\restr{X}{x})})(A^D_{x})\\
		&= D(\restr{f^{-1}}{(\restr{Y}{f(x)})} \circ \restr{f}{(\restr{X}{x})})(A^D_{x})\\
		&= D(\iota_{\restr{X}{x}}^{\restr{X}{x}})(A^D_x) = A^D_x = A^D_{f^{-1}(y)}\period\qedhere
	\end{align*}
\end{proof}

\section{Parallels with $1$-fixed points}\label{sec:one_fixed_points}
In this section, we explore the deep connection between increasing $D$-sequences and $1$-fixed points. We begin with the definition of the more general notation of a $\nu$-fixed point where $\nu$ is an arbitrary well order. This concept is due to Freund and Rathjen and inspired by Buchholz's ordinal notation systems (cf.~\cite{Buchholz86}):

\begin{definition}[\cite{FR21}]\label{def:nu_fp}
	Consider a well order $\nu$ together with a predilator $D$. A \emph{$\nu$-collapse} of $D$ is a linear order $X$ together with an embedding $\pi: X \to \nu \times D(X)$ such that two properties hold:
	Let the relation $\triangleleft$ on $X$ be defined by the equivalence
	\begin{equation*}
		s \triangleleft t \quad :\Longleftrightarrow \quad s \in \supp_X(\tau) \text{ for } \pi(t) = (\alpha, \tau)\period
	\end{equation*}
	The first condition is that $\triangleleft$ respects some height function $h: X \to \n$, i.e., it satisfies $h(s) < h(t)$ for any elements $s, t \in X$ with $s \triangleleft t$.	
	
	With this, we can apply recursion along this relation $\triangleleft$ in order to define two functions $G^D_\gamma: X \to [D(X)]^{<\omega}$ and $G_\gamma: D(X) \to [D(X)]^{<\omega}$ simultaneously:
	\begin{align*}
		G^D_\gamma(t) &:=
		\begin{cases}
			\{\tau\} \cup G_\gamma(\tau) & \text{if } \pi(t) = (\alpha, \tau) \text{ with } \alpha \geq \gamma\comma\\
			\emptyset & \text{if } \pi(t) = (\alpha, \tau) \text{ with } \alpha < \gamma\comma
		\end{cases}\\
		G_\gamma(\tau) &:= \bigcup \{G^D_\gamma(s) \mid s \in \supp_X(\tau)\}\period
	\end{align*}
	The second condition is the so-called \emph{range condition}:
	\begin{equation*}
		\rng(\pi) = \{(\alpha, \tau) \in \nu \times D(X) \mid G_\alpha(\tau) <_{\fin} \tau\}\period
	\end{equation*}
	We call the order $X$ a $\nu$-fixed point of $D$ if an embedding $\pi$ with these properties exists.
\end{definition}
Given a dilator $D$, we write $\psi_\nu(D)$ and $\pi: \psi_\nu(D) \to D(\psi_\nu(D))$ for the canonical $\nu$-fixed point of $D$ and its embedding, respectively. The existence is due to Theorem~2.9 from \cite{FR21}. Moreover, we will sometimes make use of an order $\psi_\nu^+(D)$ with $\psi_\nu(D) \subseteq \psi_\nu^+(D)$ such that there exists an isomorphism $\pi^+: \psi_\nu^+(D) \to D(\psi_\nu^+(D))$ with $\pi^+(x) = (D(\iota_{\psi_1(D)}^{\psi_1^+(D)}) \circ \pi)(x)$ for all $x \in \psi_\nu(D)$.
For our purposes, we only need $\nu$-fixed points with $\nu := 1$. The next lemma shows that, in this case, the definition of $G_0$ can be simplified significantly:
\begin{lemma}[$\rca_0$]\label{lem:1_fp_simplyfied_G}
	Let $D$ be a predilator and $X$ a linear order.
	For the definition of $1$-fixed points, we identify the codomain $1 \times D(X)$ of $\pi$ with $D(X)$ in the canonical way. We arrive at a notion that is equivalent to $1$-fixed points if we replace our current definition of the function $G_0$ with the following one:
	\begin{equation*}
		G_0(\tau) := \{\pi(s) \mid s \in \supp_X(\tau)\}\period
	\end{equation*}
\end{lemma}

\begin{proof}
	Let $G'_0$ be the original definition of $G_0$ as it is given by Definition \ref{def:nu_fp}. Consider the sets
	\begin{equation*}
		A := \{\tau \in D(X) \mid G'_0(\tau) <_{\fin} \tau\} \quad \text{and} \quad B := \{\tau \in D(X) \mid G_0(\tau) <_{\fin} \tau\}\period
	\end{equation*}
	Under the assumption that our collapse $\pi: X \to D(X)$ already satisfies the first condition for $1$-fixed points, we want to show that $\rng(\pi) = A$ holds if and only if $\rng(\pi) = B$ does. Since $G_0(t)$ is always a subset of $G'_0(t)$ for $t \in X$, we have $A \subseteq B$. Thus, under any of both assumptions, we conclude $\rng(\pi) \subseteq B$.
	
	Now, we take this common conclusion $\rng(\pi) \subseteq B$ and show that it entails $A = B$, which results in our claim. We prove $B \subseteq A$: Assume, for contradiction, that there is some $\tau \in B \setminus A$. Using $\Delta^0_1$-induction along the height function provided by the first condition for $1$-fixed points, we can assume that $\tau$ is minimal with this property, i.e., for all $\pi(s) \in B$ with $s \in \supp_X(\tau)$, we already have $\pi(s) \in A$.
	
	Consider some element $\sigma \in G'_0(\tau)$. If $\sigma = \pi(s)$ holds for some $s \in \supp_X(\tau)$, then we have $\sigma \in G_0(\tau) <_{\fin} \tau$. Otherwise, if $\sigma \in G'_0(\pi(s))$ does hold for some $s \in \supp_X(\tau)$, then $\pi(s) \in \rng(\pi) \subseteq B$ already implies $\pi(s) \in A$ by assumption. Thus, we have $\sigma \in G'_0(\pi(s)) <_{\fin} \pi(s) \in G_0(\tau) <_{\fin} \tau$. We conclude $G'_0(\tau) <_{\fin} \tau$ and, therefore, the contradiction $\tau \in A$.
\end{proof}
We come to the central result of this paper: The following theorem makes the deep connection between strong $D$-sequence termination points and $1$-fixed points apparent in form of a one-to-one correspondence between the two concepts.
\begin{theorem}[$\rca_0$]\label{thm:d_sequence_fp_is_1_fp}
	For any predilator $D$ and any linear order $X$, the following are equivalent:
	\begin{enumerate}[label=(\alph*)]
		\item The linear order $X$ is a strong termination point of the $D$-sequence $(A_{x}^{D})_{x \in X}$.
		\item The linear order $X$ together with some embedding $\pi : X \to D(X)$ is a $1$-fixed point.
	\end{enumerate}
	Moreover, the correspondence between sequence and embedding in both directions of the proof can be given by
	\begin{equation}\label{eq:definition_pi_from_goodstein}
		\pi(x) = D(\iota_{\restr{X}{x}}^{X})(A_{x}^{D})\tag{$*$}
	\end{equation}
	for every $x \in X$.
\end{theorem}

\begin{proof}\mbox{}
	
	``(a) $\Rightarrow$ (b)'': Starting from the strong $D$-sequence, we construct the $1$-fixed point using $\pi$ as defined in (\ref{eq:definition_pi_from_goodstein}). We begin by showing that $\pi$ is an embedding. Let $x, y \in X$ be arbitrary with $x < y$, then:
	\begin{equation*}
		\pi(x) = D(\iota_{\restr{X}{x}}^{X})(A_{x}^{D}) = D(\iota_{\restr{X}{y}}^{X}) (D(\iota_{\restr{X}{x}}^{\restr{X}{y}})(A_{x}^{D})) < D(\iota_{\restr{X}{y}}^{X})(A_{y}^{D}) = \pi(y)\period
	\end{equation*}
	Here, we used the property that $A_{y}^{D}$ is greater than $D(\iota_{\restr{X}{x}}^{\restr{X}{y}})(A_{x}^{D})$, which can be derived from requirement (1).
	
	Next, we prove that the relation $\triangleleft$ has a compatible height function $h: \n \to X$. In fact, this is already given by requirement (3). Let $y \in X$ be arbitrary. Then we have the following:
	\begin{equation*}
		\supp_{X}(\pi(y)) = \supp_{X}(D(\iota_{\restr{X}{y}}^{X})(A_{y}^{D})) = \supp_{\restr{X}{y}}(A_{y}^{D})\period
	\end{equation*}
	Thus, we conclude $x \triangleleft y \Leftrightarrow x \prec y$ for all $x, y \in X$, which lets us simply copy the height function from the assumed $D$-sequence termination point.
	
	Now, we prove that $G_0(\pi(x)) <_{\fin} \pi(x)$ holds for any $x \in X$. We only need to make sure that $\pi(y) <_{D(X)} \pi(x)$ holds for any $y \in \supp_{X}(\pi(x))$. However, with $\supp_{X}(\pi(x)) <_{\fin} x$, we immediately have $y <_{X} x$ and, therefore, $\pi(y) <_{D(X)} \pi(x)$.
	
	Finally, we have to show that $G_0(\sigma) <_{\fin} \sigma$ for $\sigma \in D(X)$ already implies the existence of some $x \in X$ with $\sigma = \pi(x)$. Using requirement (4), let $x \in X$ be the smallest element satisfying both $\supp_{X}(\sigma) \leq_{\fin} x$ and $\sigma \leq \pi(x)$. We proceed by case distinction:
	
	If $\supp_{X}(\sigma) <_{\fin} x$, then we claim that already $\sigma = \pi(x)$ holds. In order to derive a contradiction, assume the strict inequality $\sigma < \pi(x) = D(\iota_{\restr{X}{x}}^{X})(A_{x}^{D})$. Using our assumption $\supp_{X}(\sigma) <_{\fin} x$, there is a unique $\tau \in D(\restr{X}{x})$ satisfying $\sigma = D(\iota_{\restr{X}{x}}^{X})(\tau)$. We conclude $\tau < A_{x}^{D}$. Moreover, assume that there is some $y < x$ with $\tau \leq D(\iota_{\restr{X}{y}}^{\restr{X}{x}})(A_{y}^{D})$. This yields $\sigma \leq D(\iota_{\restr{X}{y}}^{X})(A_{y}^{D}) = \pi(y)$. The element $y$ can be chosen to satisfy $\supp_{X}(\sigma) \leq_{\fin} y$ by simply redefining $y := \max(\{y\} \cup \supp_X(\sigma))$ if it should not be the case. Note that this keeps the property $y < x$ intact since we assumed $\supp_X(\sigma) <_{\fin} x$ to hold. But now, this entails that (4) wrongly claimed $x$ to be the smallest element with this property. We conclude that $\tau$ is an element in $D(\restr{X}{x})$ smaller than $A_{x}^{D}$ but greater than any $D(\iota_{\restr{X}{y}}^{\restr{X}{x}})(A_{y}^{D})$ for $y < x$. This contradicts (1), which claims $A_{x}^{D}$ to be the smallest element with this property. Finally, we see that our assumption $\sigma < D(\iota_{\restr{X}{x}}^{X})(A_{x}^{D})$ at the beginning must have been wrong, and we arrive at our claim $\sigma = D(\iota_{\restr{X}{x}}^{X})(A_{x}^{D}) = \pi(x)$.
	
	If $\supp_{X}(\sigma) \geq_{\fin} x$, then there must be some element $y \in \supp_{X}(\sigma)$ with $x \leq y$. We derive $\sigma \leq \pi(x) \leq \pi(y) \in G_0(\sigma)$. Clearly, this violates $G_0(\sigma) <_{\fin} \sigma$.
	
	``(b) $\Rightarrow$ (a)'': We begin by showing that $\supp_{X}(\pi(x)) <_{\fin} x$ holds for any $x \in X$. Assume that this is not the case for some $x \in X$. Then, we find $y \in \supp_{X}(\pi(x))$ with $x \leq y$. Since $\pi$ is an embedding, we conclude $\pi(x) \leq \pi(y)$. Finally, this violates $G_0(\pi(x)) <_{\fin} \pi(x)$. With this, we can define $A_{x}^{D}$ to be the unique element satisfying $\pi(x) = D(\iota_{\restr{X}{x}}^{X})(A_{x}^{D})$ for any $x \in X$. We continue by showing that this sequence $(A_{x}^{D})_{x \in X}$ satisfies all necessary requirements of a strong $D$-sequence termination point according to Definition \ref{def:soa_d_sequence_fixed_point}:
	
	Condition (1): By definition, it is clear that $A_{x}^{D}$ lives in $D(\restr{X}{x})$ for any $x \in X$. Next, we see that
	\begin{equation*}
		D(\iota_{\restr{X}{x}}^{X})(A_{x}^{D}) = \pi(x) > \pi(y) = D(\iota_{\restr{X}{y}}^{X})(A_{y}^{D}) = D(\iota_{\restr{X}{x}}^{X}) D(\iota_{\restr{X}{y}}^{\restr{X}{x}})(A_{y}^{D})
	\end{equation*}
	holds for any $ y <_{X} x$. Therefore, since $D(\iota_{\restr{X}{x}}^{X})$ is an embedding, we conclude $A_{x}^{D} > D(\iota_{\restr{X}{y}}^{\restr{X}{x}})(A_{y}^{D})$ for any $y <_{X} x$. Now, suppose that there is some element $\sigma \in D(\restr{X}{x})$ smaller than $A_{x}^{D}$ with the same property. Let $\tau := D(\iota_{\restr{X}{x}}^{X})(\sigma)$. We show $G_0(\tau) <_{\fin} \tau$. Assume, for contradiction, $G_0(\tau) \not<_{\fin} \tau$, then, there is some $\rho \in G_0(\tau)$ with $\rho \geq \tau$. From the definition of $G_0$, we conclude that $\rho = \pi(s)$ holds for some $s \in \supp_{X}(\tau)$.
	
	We proceed by case distinction: If we are in the case $s < x$, then we have
	\begin{equation*}
		D(\iota_{\restr{X}{x}}^{X})(\sigma) = \tau \leq \pi(s) = D(\iota_{\restr{X}{s}}^{X})(A_{s}^{D}) = D(\iota_{\restr{X}{x}}^{X}) D(\iota_{\restr{X}{s}}^{\restr{X}{x}})(A_{s}^{D})\period
	\end{equation*}
	Thus, $\sigma \leq D(\iota_{\restr{X}{s}}^{\restr{X}{x}})(A_{s}^{D})$ for $s < x$, which entails that $\sigma$ was not a counterexample to $A_{x}^{D}$ being the smallest element greater than each of $D(\iota_{\restr{X}{y}}^{\restr{X}{x}})(A_{y}^{D})$ for $y < x$.
	
	Otherwise, we are in the case $x \leq s$. But this clearly violates $\supp_{X}(\tau) <_{\fin} x$, which means that $\sigma$ could not have been an element of $D(\restr{X}{x})$. These contradictions yield $G_0(\tau) <_{\fin} \tau$. Now, the condition on the range of $\pi$ tells us that there is some element $y \in X$ with $\pi(y) = \tau$. Thus, we have
	\begin{equation*}
		D(\iota_{\restr{X}{x}}^{X}) (D(\iota_{\restr{X}{y}}^{\restr{X}{x}})(A_{y}^{D})) = D(\iota_{\restr{X}{y}}^{X})(A_{y}^{D}) = \pi(y) = \tau = D(\iota_{\restr{X}{x}}^{X})(\sigma)\period
	\end{equation*}
	We conclude $\sigma = D(\iota_{\restr{X}{y}}^{\restr{X}{x}})(A_{y}^{D})$, which again entails that $\sigma$ was not a counterexample to $A_{x}^{D}$.
	
	Conditions (2) and (4): Let $\sigma \in D(X)$ be arbitrary. We prove the conditions by case distinction on the truth of $G_0(\sigma) <_{\fin} \sigma$:
	
	If the relation holds, then there is some $x \in X$ with $\pi(x) = \sigma$. Moreover, we have $\supp_{X}(\sigma) <_{\fin} x$. Therefore, since $\pi$ is an embedding, $x$ is the smallest value satisfying both $\supp_{x}(\sigma) \leq_{\fin} x$ and $\sigma \leq \pi(x) = D(\iota_{\restr{X}{x}}^{X})(A_{x}^{D})$.
	
	If the relation does not hold, there is some $x \in \supp_{X}(\sigma)$ with $\pi(x) \geq \sigma$, just like above. Assume that $x$ is the largest element in $\supp_{X}(\sigma)$. Since $\pi$ is an embedding, $\pi(x) \geq \sigma$ still holds. This entails $\supp_{X}(\sigma) \leq_{\fin} x$ and $\sigma \leq \pi(x) = D(\iota_{\restr{X}{x}}^{X})(A_{x}^{D})$. Now, since $x$ is an element of $\supp_{X}(\sigma)$, any smaller candidate $y < x$ will not satisfy $\supp_{X}(\sigma) \leq_{\fin} y$. Therefore, $x$ is the claimed value.
	
	Condition (3): In the proof of direction ``(a) $\Rightarrow$ (b)'', we have already seen the equivalence $x \triangleleft y \Leftrightarrow x \prec y$ for $x, y \in X$. Therefore, if $\triangleleft$ has a height function, then so does $\prec$.
\end{proof}

\begin{remark}\label{rem:cond_1_cases}
	Inspection of the proof for condition (1) yields the following: Let $x \in X$ with $\sigma \in D(\restr{X}{x})$ be such that $\sigma < A^D_x$ holds. Then, we do not only know that there must be some $y < x$ with $\sigma \leq D(\iota_{\restr{X}{y}}^{\restr{X}{x}})(A^D_y)$ but also that $y$ can be chosen to satisfy one of the following:
	\begin{itemize}
		\item $\sigma = D(\iota_{\restr{X}{y}}^{\restr{X}{x}})(A^D_y)$, or
		\item $y \in \supp_X(D(\iota_{\restr{X}{x}}^X)(\sigma)) = \supp_{\restr{X}{x}}(\sigma)$.
	\end{itemize}
	This observation will play a key role in keeping the proof of Proposition \ref{prop:if_D_sequence_non_iso_fp_linear_orders} short.
\end{remark}
With Theorem \ref{thm:d_sequence_fp_is_1_fp}, we get existence and uniqueness results for free as they have already been proved for $1$-fixed points.
\begin{lemma}[$\rca_0$]\label{lem:unique_hom}
	For any predilator $D$, there is a unique homomorphism between any two strong $D$-sequence termination points.
\end{lemma}

\begin{proof}
	Consider two $D$-sequence termination points given by the linear orders $X$ and $Y$ together with the sequences $(A^D_x)_{x \in X}$ and $(B^D_y)_{y \in Y}$, respectively. Let $(X, \pi)$ and $(Y, \kappa)$ be the corresponding $1$-fixed points given by Theorem \ref{thm:d_sequence_fp_is_1_fp}. For the proof of existence, Proposition 2.1 from \cite{FR21} yields an embedding $f: X \to Y$ satisfying $\kappa \circ f = D(f) \circ \pi$. For any $x \in X$, we have
	\begin{align*}
		D(\iota_{\restr{Y}{f(x)}}^{Y} \circ \restr{f}{(\restr{X}{x})})(A^D_x) &= D(f \circ \iota_{\restr{X}{x}}^X)(A^D_x) = D(f)(\pi(x)) = \kappa(f(x))\\
		&= D(\iota_{\restr{Y}{f(x)}}^{Y})(B^D_{f(x)})\period
	\end{align*}
	Since $D(\iota_{\restr{Y}{f(x)}}^{Y})$ is an embedding, we conclude $D(\restr{f}{(\restr{X}{x})})(A^D_x) = B^D_{f(x)}$ for any~$x \in X$.
	
	Now, for the uniqueness, let $f, g: X \to Y$ be two homomorphisms from the former to the latter $D$-sequence. Let $h \in \{f, g\}$ be one of them. For any $x \in X$, we have
	\begin{align*}
		D(h)(\pi(x)) &= D(h \circ \iota_{\restr{X}{x}}^{X})(A^D_x) = D(\iota_{\restr{Y}{h(x)}}^{Y} \circ \restr{h}{(\restr{X}{x})})(A^D_x) = D(\iota_{\restr{Y}{h(x)}}^{Y})(B^D_{h(x)})\\
		&= \kappa(h(x))\period
	\end{align*}
	Thus, we have $D(h) \circ \pi = \kappa \circ h$ for both choices of $h \in \{f, g\}$. Invoking Proposition 2.1 from \cite{FR21} again, we know that any such embedding must be unique.
\end{proof}

\begin{corollary}[$\rca_0$]\label{cor:strong_fp_unique}
	For any predilator $D$, there exists a strong $D$-sequence termination point $X$ which is unique up to isomorphism.
\end{corollary}

\begin{proof}
	For the existence, we invoke Theorem 2.9 from \cite{FR21}, which yields a $1$-fixed point for any predilator $D$, that can be turned into a strong $D$-sequence termination point using the direction ``(b) $\Rightarrow$ (a)'' of our Theorem \ref{thm:d_sequence_fp_is_1_fp}.
	
	Now, for the uniqueness, consider some predilator $D$ together with two strong $D$-sequence termination points given by linear orders $X$ and $Y$ together with sequences $(A_{x}^{D})_{x \in X}$ and $(B_{y}^{D})_{y \in Y}$, respectively. Applying Lemma \ref{lem:unique_hom} twice, we get two $D$-sequence homomorphisms $f: X \to Y$ and $g: Y \to X$. Now, $g \circ f$ is a homomorphism from the \emph{former} $D$-sequence termination point to itself. Likewise, $f \circ g$ is a homomorphism from the \emph{latter} termination point to itself. By Lemma \ref{lem:unique_hom}, these must be unique. Finally, since already the identities on both $X$ and $Y$ are $D$-sequence homomorphisms, this entails that $g \circ f$ and $f \circ g$ must be the identities on $X$ and~$Y$, respectively. We conclude that $f$ (and also $g$) is a $D$-sequence isomorphism.
\end{proof}
We finish this section with a theorem that proves Weiermann's conjecture about the strength of Theorem \ref{thm:d_seq_terminates} from the introduction. By exploiting the connection with $1$-fixed points and the results due to Freund and Rathjen that do the heavy lifting for us, its proof only takes a few lines.
\begin{theorem}[$\rca_0$]\label{thm:pica_wf_fp}
	The following are equivalent:
	\begin{enumerate}[label=(\alph*)]
		\item $\Pi^1_1$-comprehension.
		\item Any dilator $D$ has a well founded (strong) $D$-sequence termination point.
	\end{enumerate}
\end{theorem}

\begin{proof}
	Theorem 1.6 from \cite{FR21} proves over $\rca_0$ that $\Pi^1_1$-comprehension is equivalent to the statement that any dilator $D$ has a well founded $1$-fixed point. Combining this with our Theorem \ref{thm:d_sequence_fp_is_1_fp} yields the claim. Recall that, by Lemma \ref{lem:wf_fp_strong}, any well founded $D$-sequence termination point is already strong.
\end{proof}

\section{The inverse Goodstein sequence}\label{sec:inverse_Goodstein}
We begin this chapter with a definition of the Goodstein dilator in second order arithmetic. It is, of course, similar to Definition \ref{def:goodstein_dilator_set_theory}. We do this since we have to redefine it as an endofunctor on linear orders. Moreover, we call it \emph{predilator} instead of \emph{dilator} since it is not immediately clear in weak systems that this functor preserves well orders. In fact, this property is not provable in $\aca_0$, as we will see in Proposition \ref{prop:goodstein_dilator}.
\begin{definition}
	We define an endofunctor $G$ on the category of linear orders, as follows:
	Given a linear order $X$, the set $G(X)$ is given by terms
	\begin{equation*}
		\sigma := (1 + X)^{\gamma_0} \cdot (1 + \delta_0) + \dots + (1 + X)^{\gamma_{n-1}} \cdot (1 + \delta_{n-1})
	\end{equation*}
	for $\gamma_i \in G(X)$ and $\delta_i \in X$ with $i < n$. Moreover, we demand $\gamma_0 > \dots > \gamma_{n-1}$. The order is given by simultaneous recursion as follows: Consider another term
	\begin{equation*}
		\tau := (1 + X)^{\gamma'_0} \cdot (1 + \delta'_0) + \dots + (1 + X)^{\gamma'_{m-1}} \cdot (1 + \delta'_{m-1})
	\end{equation*}
	for $\gamma'_i \in G(X)$ and $\delta'_i \in X$ with $i < m$. We have $\sigma < \tau$ if and only if one of the following holds:
	\begin{enumerate}[label=(\roman*)]
		\item The term $\tau$ is a proper extension of $\sigma$, i.e., $m > n$ holds with $\gamma_i = \gamma'_i$ and $\delta_i = \delta'_i$ for all $i < n$.
		\item There is some $i < \min(n, m)$ with $\gamma_i < \gamma'_i$ or both $\gamma_i = \gamma'_i$ and $\delta_i < \delta'_i$. Moreover, we have $\gamma_j = \gamma'_j$ and $\delta_j = \delta'_j$ for all $j < i$.
	\end{enumerate}
	Given a morphism $f: X \to Y$ between two linear orders $X$ and $Y$, the morphism $G(f)$ maps our element $\sigma$ from above to
	\begin{equation*}
		G(f)(\sigma) := (1 + Y)^{G(f)(\gamma_0)} \cdot (1 + f(\delta_0)) + \dots + (1 + Y)^{G(f)(\gamma_{n-1})} \cdot (1 + f(\delta_{n-1}))\period
	\end{equation*}
	Finally, we define $\supp_X(\sigma) := \bigcup \{\supp_X(\gamma_i) \cup \{\delta_i\} \mid i < n\}$. We call $G$ the \emph{Goodstein predilator}. Moreover, the increasing $G$-sequence is also known as the \emph{inverse Goodstein sequence}.
\end{definition}
A standard argument shows that $G(X)$ is a linear order and that $G(f)$ is an embedding for any linear orders $X$ and $Y$ with $f: X \to Y$. With this, one can easily check in $\rca_0$ that $G$ is a predilator.
Moreover, the restriction of this predilator to well orders is naturally isomorphic to the dilator given by Definition~\ref{def:goodstein_dilator_set_theory} in the introduction.

As already mentioned, the system $\aca_0$ does not suffice for showing that $G$ is a dilator. For this, we need the slightly stronger system $\aca'_0$. First, recall that $\aca_0$ can be defined by adding the axiom that $X \mapsto \omega^X$ preserves well orders to~$\rca_0$ (cf.~\cite[Theorem~2.6]{Hirst94}).
\begin{definition}\label{def:exponentiation}
	Given a linear order $X$, we define a linear order $\omega^X$ as follows: The set is given by finite sequences $\langle x_0, \dots, x_{n-1} \rangle \in X^*$ with $x_i \in X$ for all $i < n$ and $x_i \geq x_j$ for all $i < j < n$.
	
	Given two sequences $x := \langle x_0, \dots, x_{n-1} \rangle$ and $y := \langle y_0, \dots, y_{m-1} \rangle$ in $\omega^X$, we have~$x < y$ if and only if one of the following holds:
	\begin{enumerate}[label=(\roman*)]
		\item The sequence $y$ is a proper extension of $x$, i.e., $m > n$ holds together with $x_i = y_i$ for all $i < n$.
		\item There is some $i < \min(n, m)$ with $x_i < y_i$ and $x_j = y_j$ for all $j < i$.
	\end{enumerate}
\end{definition}
Now, we use a definition of $\aca'_0$ that is given by an equivalence to the actual definition (cf.~\cite[Definition~4.2]{MM09} and \cite[Theorem~5.22]{MM11}).
\begin{definition}
	For any well order $X$, we define the order $\omega^{\langle n, X\rangle}$ for $n \in \n$ inductively as follows: $\omega^{\langle 0, X\rangle} := X$ and $\omega^{\langle n + 1, X\rangle} := \omega^{(\omega^{\langle n, X\rangle})}$.
	Over $\rca_0$, the system of $\aca'_0$ can be defined by adding the following principle as an axiom: $\omega^{\langle n, X\rangle}$ is well founded for any well order $X$ and number $n \in \n$.
\end{definition}

\begin{proposition}[$\rca_0$]\label{prop:goodstein_dilator}
	The following are equivalent:
	\begin{enumerate}[label=(\alph*)]
		\item $\aca'_0$
		\item The Goodstein predilator is a dilator.
	\end{enumerate}
\end{proposition}

\begin{proof}
	``(a) $\Rightarrow$ (b)'': Given some well order $X$, we define $X' := X \cup \{\top\}$ where $\top$ is a new top element. Moreover, we define a family of functions: For every $n \in \n$, we define $f_n : X' \to \omega^{\langle 2n + 1, X'\rangle}$ as follows:
	\begin{equation*}
		f_0(x) := \langle x \rangle \quad \text{ and } \quad f_{n+1}(x) := \langle \langle f_n(x) \rangle \rangle\period
	\end{equation*}
	Using a simple induction, it can be seen that $f_n$ is an embedding for any $n \in \n$.
	Next, we have want to define functions $g_n : G(X) \to \omega^{\langle 2n + 1, X' \rangle}$ for every $n \in \n$ with $g_0 \equiv \langle\rangle$ and
	\begin{multline*}
		g_{n+1}((1+X)^{\gamma_0} \cdot (1 + \delta_0) + \dots + (1+X)^{\gamma_{m-1}} \cdot (1 + \delta_{m-1}))\\:= \langle \langle g_n(\gamma_0), f_n(\delta_0)\rangle, \dots, \langle g_n(\gamma_{m-1}), f_n(\delta_{m-1})\rangle, \langle f_n(\top) \rangle\rangle\period
	\end{multline*}
	We want to show that every $g_n$ is well defined, i.e., that every element in the range of $g_n$ satisfies the restriction imposed on elements in $\omega^{\langle 2n + 1, X' \rangle}$, and that is an embedding into $\omega^{\langle 2n + 1, X'\rangle}$. However, neither of these requirements hold at the moment, in general. We need to restrict the domains of $g_n$. For this restriction, we make use of a height function $h : G(X) \to \n$ with $h(0) := 1$ and
	\begin{equation*}
		h((1+X)^{\gamma_0} \cdot (1 + \delta_0) + \dots + (1+X)^{\gamma_{n-1}} \cdot (1 + \delta_{n-1})) := 1 + \max \{h(\gamma_i) \mid i < n\}\period 
	\end{equation*}
	Let us write $\restr{G(X)}{n}$ for the suborder of $G(X)$ that only consists of elements $\sigma \in G(X)$ with $h(\sigma) \leq n$. For each $n$, we restrict the domain of $g_n$ to $\restr{G(X)}{n}$.
	Using $\Pi^0_1$-induction, we show for every $n \in \n$ that $g_n$ (now restricted to arguments $\sigma$ with $h(\sigma) \leq n$) is an embedding into $\omega^{\langle 2n + 1, X'\rangle}$ and produces values that are greater or equal to anything in the range of $f_n$.
	
	For $n = 0$, there is no $\sigma \in G(X)$ with $h(\sigma) \leq n$. Assume $n > 0$ and let $\sigma$ be of the form $(1 + X)^{\gamma_0} \cdot (1 + \delta_0) + \dots + (1 + X)^{\gamma_{m-1}} \cdot (1 + \delta_{m-1})$ for ${\gamma_i \in G(X)}$ and $\delta_i \in X$ with~$i < m$. We begin by showing that $g_n$ maps $\sigma$ to an element of $\omega^{\langle 2n + 1, X' \rangle}$. For this, we begin by considering $\langle g_{n-1}(\gamma_i), f_{n-1}(\delta_i) \rangle$ for any $i < m$. By induction hypothesis, we know that $g_{n-1}(\gamma_i)$ must be greater or equal to $f_{n-1}(\delta_i)$. Moreover, we have that both values live in $\omega^{\langle 2n - 1, X'\rangle}$. We conclude $\langle g_{n-1}(\gamma_i), f_{n-1}(\delta_i) \rangle \in \omega^{\langle 2n, X' \rangle}$. Next, we consider two such elements $\langle g_{n-1}(\gamma_i), f_{n-1}(\delta_i) \rangle$ and $\langle g_{n-1}(\gamma_j), f_{n-1}(\delta_j) \rangle$ with $i < j < m$. Clearly, the former is greater or equal to the latter since $\gamma_i > \gamma_j$ holds and $g_{n-1}$ is an embedding on values with height less than $n$. Finally, $\langle g_{n-1}(\gamma_i), f_{n-1}(\delta_i) \rangle$ is greater than $\langle f_{n-1}(\top) \rangle$ since $g_{n-1}(\gamma_i)$ is greater or equal to $f_{n-1}(\top)$. We conclude that $g_n(\sigma)$ is an element of $\omega^{\langle 2n + 1, X' \rangle}$.
	
	Next, we want to show that $g_n(\sigma) \geq f_n(x)$ holds for any $x \in X'$. By definition, we have $g_n(\sigma) \geq \langle \langle f_{n-1}(\top) \rangle \rangle = f_n(\top) \geq f_n(x)$.
	The only remaining claim of our induction step is that $g_n$ is an embedding on elements of height less or equal to $n$. For this, consider a further element $\tau \in G(X)$ with $h(\tau) \leq n$ of the form $\tau = (1 + X)^{\gamma'_0} \cdot (1 + \delta'_0) + \dots + (1 + X)^{\gamma'_{k - 1}} \cdot (1 + \delta'_{k - 1})$ with $\gamma'_i \in G(X)$ and $\delta'_i \in X$ for $i < k$. Now, assume $\sigma < \tau$. If $\tau$ is an extension of $\sigma$, then $g_n(\sigma)$ and $g_n(\tau)$ coincide at the first $m$-many positions. At the $(m+1)$-th position, $g_n(\sigma)$ has the value $\langle f_{n-1}(\tau) \rangle$ and $g_n(\tau)$ has the value $\langle g_{n-1}(\gamma'_m), f_{n-1}(\delta'_m) \rangle$. From the previous considerations, we know that the former lies strictly below the latter.
	
	If $\tau$ is not an extension of $\sigma$, then there is some index $i < \min(m, k)$ such that we have $(1 + X)^{\gamma_i} \cdot (1 + \delta_i) < (1 + X)^{\gamma'_i} \cdot (1 + \delta'_i)$ and both $\gamma_j = \gamma'_j$ and $\delta_j = \delta'_j$ hold for all $j < i$. Hence, the first $i$-many positions of $g_n(\sigma)$ and $g_n(\tau)$ coincide. If we have $\gamma_i < \gamma'_i$, then $g_{n-1}$ being an embedding on values with height less than $n$ implies $g_{n-1}(\gamma_i) < g_{n-1}(\gamma'_i)$ and, therefore, $\langle g_{n-1}(\gamma_i), f_{n-1}(\delta_i) \rangle < \langle g_{n-1}(\gamma'_i), f_{n-1}(\delta'_i) \rangle$. Otherwise, if $\gamma_i = \gamma'_i$ and $\delta_i < \delta'_i$ hold, the same is implied by $f_{n-1}$ being an embedding. In both cases, we conclude $g_n(\sigma) < g_n(\tau)$.
	
	We want to make sure that there is a bound on the height of each element in a descending sequence in $G(X)$. For this, we show by induction along $h(\sigma) + h(\tau)$ for two terms $\sigma$ and $\tau$ of $G(X)$ that $h(\sigma) < h(\tau)$ implies $\sigma < \tau$. If $h(\sigma) = 1$, then we have both $\sigma = 0$ and $h(\tau) \neq 1$. Thus, $\tau > 0$ holds. Assume now that both $h(\sigma)$ and $h(\tau)$ are greater than $1$, i.e., both $\sigma$ and $\tau$ are greater than $0$. Using the previous explicit representations for $\sigma$ and $\tau$, we know by induction hypothesis both $h(\gamma_0) \geq \dots \geq h(\gamma_{m-1})$ and $h(\gamma'_0) \geq \dots \geq h(\gamma'_{k-1})$. Therefore, $h(\gamma_0) = h(\sigma) - 1 < h(\tau) - 1 = h(\gamma'_0)$. By induction hypothesis, we conclude $\gamma_0 < \gamma'_0$. Finally, this implies $\sigma < \tau$.
	
	For the final steps of this direction, consider some descending sequence $(\sigma_n)_{n \in \n}$ in $G(X)$. By the previous result, we know $h(\sigma_n) \leq h(\sigma_0)$ for any $n \in \n$. Therefore, we conclude that $(g_{h(\sigma_0)}(\sigma_n))_{n \in \n}$ is a descending sequence in $\omega^{\langle 2h(\sigma_0) + 1, X' \rangle}$. Thus, if $\aca'_0$ holds and $\omega^{\langle n, X' \rangle}$ is well founded for any $n \in \n$, then so is $G(X)$.
	
	``(b) $\Rightarrow$ (a)'': Let $X$ be a linear order. We want to show that $\omega^{\langle n, X \rangle}$ is well founded under the assumption that the Goodstein predilator is a dilator. Let $X + \n$ denote the disjoint union of $X$ and $\n$ such that we consider any $x \in X$ to be strictly smaller than any $n \in \n$ and we have two embeddings $\iota_0: X \to X + \n$ and $\iota_1: \n \to X + \n$ with disjoint ranges.
	We define a family $f_n: \omega^{\langle n, X \rangle} \to G(X)$ for $n \in \n$ as follows:
	\begin{align*}
		f_0(x) &:= (1 + X)^0 \cdot (1 + \iota_0(x))\comma\\
		f_{n+1}(\langle \rangle) &:= 0\period
	\intertext{Let $\sigma := \langle \sigma_0, \dots, \sigma_{m-1} \rangle \in \omega^{\langle n+1, X \rangle}$ be arbitrary with $m > 0$ and let $k \leq m$ be the greatest value such that $\sigma_i = \sigma_j$ holds for any $i < j < k$. Then, we set:}
		f_{n+1}(\sigma) &:= (1 + X)^{f_n(\sigma_0)} \cdot (1 + \iota_1(k)) + f_{n+1}(\langle \sigma_k, \dots, \sigma_{n-1} \rangle)\period
	\end{align*}
	Using $\Pi^0_1$-induction along $n$ and the usual arguments, we can see that $f_n$ is an embedding for any $n \in \n$:
	
	For $n = 0$, it is clear that $f_n$ is an embedding. Assume that our claim has already been shown for $n \in \n$ and let $\sigma := \langle \sigma_0, \dots, \sigma_{m-1} \rangle$ and $\tau := \langle \tau_0, \dots, \tau_{k-1} \rangle$ be arbitrary elements of $\omega^{\langle n+1, X \rangle}$ with $\sigma < \tau$. By side induction along $m + k$, we show that this entails $f_{n+1}(\sigma) < f_{n+1}(\tau)$. If $\sigma = \langle \rangle$ holds, our claim is clear. Assume that $\sigma$ and $\tau$ are both different from $\langle \rangle$ and let both $i \leq m$ and $j \leq k$ be maximal such that $\langle \sigma_0, \dots, \sigma_{i-1} \rangle$ and $\langle \tau_0, \dots, \tau_{j-1} \rangle$, respectively, only consist of members that are pairwise equal. By assumption, we conclude that both $i$ and $j$ must be positive. From $\sigma < \tau$, we know that $\sigma_0 \leq \tau_0$ must hold. If this inequality is strict, then by induction hypothesis, we have $f_n(\sigma_0) < f_n(\tau_0)$ and our claim follows immediately. Otherwise, we have $\sigma_0 = \tau_0$ and easily derive that $i \leq j$ must hold. If this inequality is strict, then we have $(1 + X)^{f_n(\sigma_0)} \cdot (1 + \iota_1(i)) < (1 + X)^{f_n(\tau_0)} \cdot (1 + \iota_1(j))$ and our claim follows. Otherwise, assume that $i = j$ holds. By definition of $i$ and $j$ together with their maximality, we see that $\sigma' := \langle \sigma_i, \dots, \sigma_{m-1} \rangle < \langle \tau_j, \dots, \tau_{k-1} \rangle =: \tau'$ must hold. By side induction, we find $f_{n+1}(\sigma') < f_{n+1}(\tau')$ and, hence, $f_{n+1}(\sigma) < f_{n+1}(\tau)$ by definition of $f_{n+1}$.
	
	Therefore, if the Goodstein predilator is a dilator, then $\omega^{\langle n, X \rangle}$ is well founded for any $n \in \n$ and the characteristic axioms of $\aca'_0$ are satisfied.
\end{proof}
\begin{corollary}[$\rca_0$]
	The system $\pica$ proves that the termination point of the inverse Goodstein sequence is well founded.
\end{corollary}
\begin{proof}
	Since $\aca'_0$ is implied by $\pica$, this follows by combining Theorem \ref{thm:pica_wf_fp} and Proposition \ref{prop:goodstein_dilator}.
\end{proof}
Similar to Proposition \ref{prop:goodstein_bachmann}, we will show in Proposition \ref{prop:goodstein_bachmann_second_order} that the strong $G$-sequence termination point is \emph{equimorphic} to the Bachmann-Howard ordinal for the Goodstein predilator $G$, i.e., both linear orders embed into each other. In sufficiently strong systems, this entails an isomorphism.

In second order arithmetic, there are no unique representations of ordinals. Thus, we have to decide for a notation system. We use the one published by Rathjen and Weiermann in \cite{RW93} closely matching the presentation given in \cite{FreundSecond}. We choose this system because it is easy to work with, well established, and it gives us the opportunity to show how the connection between Bachmann-Howard fixed points and $1$-fixed points yields a general method for embedding Rathjen-style notation systems into Buchholz-style notation systems. We will also use a similar method in Section \ref{sec:weak_inverse_Goodstein}.
\begin{definition}
	We simultaneously define a set of terms $\ot$ with an order and a function $E: \ot \to [\ot]^{<\omega}$:
	\begin{enumerate}[label=(\roman*)]
		\item $\Omega \in \ot$.
		\item $\vartheta(\sigma) \in \ot$ for any $\sigma \in \ot$.
		\item $\omega^{\gamma_0} + \dots + \omega^{\gamma_{n-1}} \in \ot$ if $\gamma_i \in \ot$ for $i < n$ and $\gamma_i \geq \gamma_j$ for $i < j < n$ hold. Additionally, we require $n > 1$ if $n > 0$ and the first element $\gamma_0$ is equal to $\Omega$ or $\vartheta(\sigma)$ for some $\sigma \in \ot$.
	\end{enumerate}
	Instead of an empty sum in (iii) for $n = 0$, we often write $0$.
	Next, the definition of $E$ is given by
	\begin{equation*}
		E(\Omega) := \emptyset\comma \quad E(\vartheta(\sigma)) := \{\vartheta(\sigma)\}\comma \quad E(\omega^{\gamma_0} + \dots + \omega^{\gamma_{n-1}}) := \bigcup_{i < n} E(\gamma_i)\period
	\end{equation*}
	Finally, we define the order on $\ot$:
	\begin{itemize}
		\item If $\sigma := \Omega$, then $\sigma < \tau$ holds for $\tau \in \ot$ if and only if
		\begin{itemize}
			\item $\tau = \omega^{\gamma_0} + \dots + \omega^{\gamma_{n-1}}$ and $\Omega \leq \gamma_0$ hold.
		\end{itemize}
		\item If $\sigma := \vartheta(\sigma')$, then $\sigma < \tau$ holds for $\tau \in \ot$ if and only if
		\begin{itemize}
			\item $\tau = \Omega$ holds, or
			\item $\tau = \vartheta(\tau')$ holds together with one of
			\begin{itemize}
				\item $\sigma \leq_{\fin} E(\tau')$, or
				\item $\sigma < \tau$ and $E(\sigma) <_{\fin} \tau$, or
			\end{itemize}
			\item $\tau = \omega^{\gamma_0} + \dots + \omega^{\gamma_{m-1}}$ and $\sigma \leq \gamma_0$.
		\end{itemize}
		\item If $\sigma := \omega^{\gamma_0} + \dots + \omega^{\gamma_{n-1}}$, then $\sigma < \tau$ holds for $\tau \in \ot$ if and only if
		\begin{itemize}
			\item $\tau = \Omega$ with $\gamma_0 < \Omega$, or
			\item $\tau = \vartheta(\tau')$ with $\gamma_0 < \tau$, or
			\item $\tau = \omega^{\delta_0} + \dots + \omega^{\delta_{m-1}}$ with $\sigma < \tau$ like in Definition \ref{def:exponentiation}.
		\end{itemize}
	\end{itemize}
	We define $\bachmann$, our representation of the Bachmann-Howard ordinal, to be the suborder
	\begin{equation*}
		\bachmann := \{\sigma \in \ot \mid \sigma < \Omega\}\period
	\end{equation*}
	We call terms $\Omega$, $\vartheta(\sigma)$, and $\omega^\sigma$ for $\sigma \in \ot$ \emph{indecomposable}. All other terms are called \emph{decomposable}.
\end{definition}
While it is possible to prove statements about $\bachmann$ using structural induction, sometimes we need to apply induction on two terms at the same time. For cases like this, it is convenient to have a length function at hand. We define $l: \bachmann \to \n$ with
\begin{align*}
	l(\Omega) &:= 0\comma\\
	l(\vartheta(\sigma)) &:= l(\sigma) + 1\comma\\
	l(\omega^{\gamma_0} + \dots + \omega^{\gamma_{n-1}}) &:= 1 + \sum_{i < n} l(\gamma_i)
\end{align*}

One can show in $\rca_0$ that $\ot$ is a linear order and, in $\zfc$, that it is isomorphic to the Bachmann-Howard ordinal from set theory (cf.~\cite{Buchholz2017} for more on the relation between different such notation systems).
By defining $\omega^\Omega := \Omega$ and $\omega^{\vartheta(\sigma)} := \vartheta(\sigma)$ for $\sigma \in \ot$, we can write any term in $\ot$ as the sum of powers of $\omega$.
\begin{definition}[Addition, multiplication, exponentiation]
	Consider two terms $\sigma := \omega^{\gamma_0} + \dots + \omega^{\gamma_{n-1}}$ and $\tau := \omega^{\delta_0} + \dots + \omega^{\delta_{m-1}}$. We define addition:
	\begin{equation*}
		\sigma + \tau := \omega^{\gamma_0} + \dots + \omega^{\gamma_i} + \omega^{\delta_0} + \dots + \omega^{\delta_{m-1}}\comma
	\end{equation*}
	where $i < n$ is the largest index with $\gamma_i \geq \delta_0$ (or $i := -1$ if it does not exist).
	Moreover, we need multiplication and exponentiation but only if they are of the following form:
	\begin{align*}
		\Omega \cdot \sigma &:= \omega^{\Omega + \gamma_0} + \dots + \omega^{\Omega + \gamma_{n-1}}\comma\\
		\Omega^\sigma \cdot \tau &:= \omega^{\Omega \cdot \sigma + \delta_0} + \dots + \omega^{\Omega \cdot \sigma + \delta_{m-1}}\period
	\end{align*}
\end{definition}
\noindent
It can easily be seen in $\rca_0$ that addition is associative.

Let us restate Proposition \ref{prop:goodstein_bachmann} in the context of reverse mathematics. Instead of constructing an isomorphism, we are satisfied with an equimorphism. This is already enough to transfer the property of being well founded between the two~orders. If we assume that one of the orders is well founded, then $\atr_0$ tells us that both are isomorphic (cf.~\cite[Theorem~5.2]{FriedmanHirst}).
\begin{proposition}[$\rca_0$]\label{prop:goodstein_bachmann_second_order}
	Any strong $G$-sequence termination point of the Goodstein dilator $G$ is equimorphic to $\bachmann$. In particular, there is a well founded $G$-sequence termination point if and only if $\bachmann$ is well founded.
\end{proposition}
Whenever we have a result that involves an isomorphism or an equimorphism with a well-known ordinal like $\bachmann$, we can derive consequences from ordinal analysis as the following:
\begin{corollary}[$\rca_0$]
	The system of $\picaminus$, i.e.~$\aca_0$ together with the principle of $\Pi^1_1$-comprehension for formulas without set parameters, is unable to prove that the inverse Goodstein sequence terminates.
\end{corollary}
\begin{proof}
	Similar to $\idone$, which was used in the original result (cf.~\cite[Proposition~14]{Abrusci87} and \cite{AGV90}), the proof-theoretic ordinal of $\picaminus$ is also the Bachmann-Howard ordinal (cf.~\cite{sep-proof-theory}). Thus, Proposition \ref{prop:goodstein_bachmann_second_order} implies our claim.
\end{proof}
The proof of Proposition~\ref{prop:goodstein_bachmann_second_order} requires two embeddings: One from $\bachmann$ into our termination point and one from our termination point into $\bachmann$. The latter is rather short and will be given at the end of the section. Our main effort lies in proving the former. This proof consists of several steps and we shall now give a short overview of them in order to avoid confusion. The goal of this direction is an embedding from $\bachmann$ into $\psi_1(G)$, the canonical $1$-fixed point of the Goodstein dilator $G$, which is isomorphic to the termination point of $G$ by Theorem~\ref{thm:d_sequence_fp_is_1_fp}.
\begin{itemize}
	\item We begin by defining two normal forms for our notation system $\ot$ that are given by orders on terms $\mathfrak{N} \supset \mathfrak{C}$. (Definition~\ref{def:two_omega_normal_forms})
	\item We prove that any term in $\ot$ can be expressed in a normal form as term in $\mathfrak{N}$. (Lemma~\ref{lem:omega_normal})
	\item We construct an embedding from $\mathfrak{N}$ to $\mathfrak{C}$. With our previous result, this yields an embedding from $\ot$ to $\mathfrak{C}$. Moreover, we show that our function maps every element from $\bachmann \subset \ot$ into a certain suborder $C$ of $\mathfrak{C}$. (Lemma~\ref{lem:embedding_omega_normal_collapsed})
	\item We show that in the case of our considered dilator $G$, the order $\psi_1(\omega \circ G)$ can be embedded into $\psi_1(G)$. (Lemma~\ref{lem:subset_of_psi_G_is_fp_of_omega_G})
	\item We see that the suborder $C$, that $\bachmann$ embeds into, is a so-called \emph{initial Bachmann-Howard fixed point} of $G$. By its initiality, it embeds into ${\psi_1(\omega \circ G)}$, which is a Bachmann-Howard fixed point of $G$ by \cite[Theorem~4.2]{FR21}. Finally, since $\psi_1(\omega \circ G)$ embeds into $\psi_1(G)$, this yields our embedding from $\bachmann$ into $\psi_1(G)$.
\end{itemize}
We continue with two definitions of $\Omega$-normal forms for $\ot$, which we already mentioned in the first step of our overview:
\begin{definition}\label{def:two_omega_normal_forms}
	We consider the following normal forms for elements $\sigma \in \ot$:
	\begin{enumerate}[label=(\roman*)]
		\item $\Omega$-normal form:
		\begin{equation*}
			\sigma = \Omega^{\gamma_0} \cdot \delta_0 + \dots + \Omega^{\gamma_{n-1}} \cdot \delta_{n-1}\comma
		\end{equation*}
		where the sequence $\gamma_0, \dots, \gamma_{n-1}$ is strictly descending, $\gamma_i$ is also in $\Omega$-normal form and $0 < \delta_i < \Omega$ holds, for all $i < n$. Moreover, for any term ${\vartheta(\tau) \in E(\sigma)}$, $\tau$ must also be in $\Omega$-normal form. In the case of $\sigma = 0$, we still write $0$ instead of an empty sum.
		\item $\Omega$-normal form with collapsed coefficients,
		\\
		i.e., under the same restrictions as above but we add the following condition to the inductive definition: Every $\delta_i$ is a $\vartheta$-term for $i < n$.
	\end{enumerate}
	
	We interpret both normal forms as sets of terms $\mathfrak{N}$ and $\mathfrak{C}$, respectively. Moreover, we define linear preorders, i.e.~relations that are reflexive, transitive, and connected, on these sets by simply evaluating the terms to elements of $\ot$ and reflecting the order that we defined there.
\end{definition}

One can easily see that $\Omega$-normal forms (and, therefore, also $\Omega$-normal forms with collapsed coefficients) are unique, thus making the preorders on $\mathfrak{N}$ and $\mathfrak{C}$ proper linear orders.

\begin{lemma}[$\rca_0$]\label{lem:omega_normal}
	Any term of $\ot$ can be written in $\Omega$-normal form.
\end{lemma}

\begin{proof}
	For $\gamma = 0$, we simply consider the empty sum. For $\gamma > 0$, we prove this in three steps. First, we show that any $\gamma \in \ot$ with $\gamma \geq \Omega$ can be written as $\gamma = \Omega + \delta$ for some $\delta \in \ot$. If $\gamma$ is indecomposable, then we have either $\gamma = \Omega$ or $\gamma = \omega^{\rho}$ for $\rho > \Omega$. We choose $\delta := 0$ or $\delta := \gamma$, respectively. If $\gamma$ is decomposable, then we have $\gamma = \gamma_0 + \gamma_1$ for $\gamma_0, \gamma_1 \in \ot$ with $\Omega \leq \gamma_0$ and $l(\gamma_0) < l(\gamma)$. Thus, we know by induction on the length of terms that there is some $\delta_0 \in \ot$ with $\gamma_0 = \Omega + \delta_0$. We conclude $\gamma = \Omega + (\delta_0 + \gamma_1)$.
	
	Secondly, we show that any $\gamma \in \ot$ can be written as $\gamma = \Omega \cdot \delta + \rho$ for ${\delta, \rho \in \ot}$ with $\rho < \Omega$. This is clear for values of $\gamma$ below $\Omega$. Otherwise, write $\gamma = \omega^{\rho_0} + \dots + \omega^{\rho_{n-1}}$ and let $i \leq n$ be the greatest value such that $\rho_j \geq \Omega$ holds for all $j < i$. Using the result from above, we can write any such $\rho_j$ as the sum $\Omega + \rho'_j$ with $\rho'_j \in \ot$ for $j < i$. We conclude $\gamma = \Omega \cdot (\omega^{\rho'_0} + \dots + \omega^{\rho'_{i-1}}) + \omega^{\rho_{i}} + \dots + \omega^{\rho_{n-1}}$.
	
	Thirdly, we show the claim. Any value $\gamma \in \ot$ with $0 < \gamma < \Omega$ can be written as $\Omega^0 \cdot \gamma$. Also, $\Omega$ can be expressed as $\Omega^{\Omega^0 \cdot 1} \cdot 1$. Now, if $\gamma$ is decomposable and, therefore, of the form $\omega^{\rho_0} + \dots + \omega^{\rho_{n-1}}$ for $n > 1$, then let $i \leq n$ be the greatest value such that $\rho_j \geq \Omega$ holds for all $j < i$, just like above. Using the previous paragraph, we can write any $\rho_j$ as $\Omega \cdot \delta_j + \eta_j$ with $\eta_j < \Omega$ for $j < i$. It is important for the later induction that we have $l(\delta_j) < l(\gamma)$ for $j < i$. Moreover, we calculate $\omega^{\rho_j} = \omega^{\Omega \cdot \delta_j + \eta_j} = \Omega^{\delta_j} \cdot \omega^{\eta_j}$ for $j < i$. Thus, we can write $\gamma$ as
	\begin{equation*}
		\gamma = \Omega^{\delta_0} \cdot \omega^{\eta_0} + \dots + \Omega^{\delta_{i-1}} \cdot \omega^{\eta_{i-1}} + \Omega^0 \cdot \omega^{\rho_{i}} + \dots + \Omega^0 \cdot \omega^{\rho_{n-1}}
	\end{equation*}
	with $\delta_j \geq \delta_k$ for $j < k < i$, $\eta_j < \Omega$ for $j < i$, and $\rho_{j} < \Omega$ for $j$ with $i \leq j < n$. Using our definition of multiplication and by applying this procedure inductively (along the length of terms) on $\delta_j$ for $j < i$ and all terms $\tau$ with $\vartheta(\tau) \in E(\gamma)$, this yields our claim.
\end{proof}

\begin{lemma}[$\rca_0$]\label{lem:embedding_omega_normal_collapsed}
	There is an embedding from $\ot$ into $\mathfrak{C}$ and an embedding from $\bachmann$ into the suborder of all $\vartheta$-terms in $\mathfrak{C}$.\footnote{In the context of $\mathfrak{N}$ or $\mathfrak{C}$, a ``$\vartheta$-term'' is to be understood as a term of the form $\Omega^0 \cdot \vartheta(\delta)$ for some $\delta \in \mathfrak{N}$ or $\delta \in \mathfrak{C}$, respectively.}
\end{lemma}

\begin{proof}
	Recall that $\mathfrak{N}$ is the order of terms written in $\Omega$-normal form. Moreover, $\mathfrak{C}$ is the order of terms written in $\Omega$-normal form with collapsed coefficients. We construct an embedding $f : \mathfrak{N} \to \mathfrak{C}$. By Lemma \ref{lem:omega_normal}, there is an embedding from $\ot$ into $\mathfrak{N}$. If we compose $f$ after this embedding, the resulting map yields our first claim.
	
	For the definition of $f$, we employ a second function $g : \mathfrak{M} \to \mathfrak{C}$ where $\mathfrak{M}$ is the restriction of $\mathfrak{N}$ to terms whose value is below $\Omega$ (for ease of notation, we simply write $\delta$ for any term $\Omega^0 \cdot \delta \in \mathfrak{M}$).
	As we are no longer allowed to use $\omega$-terms, we code natural numbers by nested $\vartheta$-terms as follows:
	\begin{equation*}
		\overline{0} := 0 \quad \text{and} \quad \overline{n+1} := \Omega^0 \cdot \vartheta(\overline{n})\period
	\end{equation*}
	Both $f$ and $g$ are defined by simultaneous recursion:
	\begin{align*}
		f(0) &:= \Omega^{\overline{3}}\comma\\
		f(\Omega^{\gamma_0} \cdot \delta_0 + \dots + \Omega^{\gamma_{n-1}} \cdot \delta_{n-1}) &:= \Omega^{f(\gamma_0)} \cdot \vartheta(g(\delta_0)) + \dots + \Omega^{f(\gamma_{n-1})} \cdot \vartheta(g(\delta_{n-1}))\comma
	\end{align*}
	where we always assume $n > 0$ in the second case. A priori, it is not clear that $f$ produces values that satisfy $f(\gamma_0) > \dots > f(\gamma_{n-1})$ in the second line of its definition. This will follow from $f$ being an embedding, which we have not shown, yet. For the time being, we simply omit this restriction but keep in mind that we need to be careful here.
	For the definition of $g$, we consider three cases:
	\begin{itemize}
		\item $g(0) := 0$,
		\item $g(\vartheta(\tau)) := \Omega^{\overline{2}} \cdot \vartheta(f(\tau))$,
		\item $g(\omega^{\gamma} + \delta) := \Omega^{\overline{1}} \cdot \vartheta(g(\gamma)) + \Omega^{\overline{0}} \cdot \vartheta(g(\delta))$ if $\delta < \omega^{\gamma + 1}$ and either $\delta > 0$ or $\gamma$ is not a $\vartheta$-term.
	\end{itemize}
	Now that our definition is complete, a simple induction shows that every element that $f$ maps to lives in $\mathfrak{C}$ (without the restriction on exponents). The definition of $f(0)$ implies $\Omega^{\overline{3}} \leq f(\tau)$ for any $\tau \in \mathfrak{N}$. This entails that any element in the range of $g$ lies below any element in the range of $f$.
	Let us now show the following by simultaneous induction on the sum $l(\sigma) + l(\tau)$ for terms $\sigma$ and $\tau$:
	\begin{enumerate}[label=(\alph*)]
		\item $\vartheta(g(\sigma)) < \vartheta(g(\tau))$ for $\sigma, \tau \in \mathfrak{M}$ with $\sigma < \tau$,
		\item $f(\sigma) < f(\tau)$ for $\sigma, \tau \in \mathfrak{N}$ with $\sigma < \tau$,
		\item $E(\sigma) <_{\fin} \vartheta(\tau) \Rightarrow E(g(\sigma)) <_{\fin} \vartheta(f(\tau))$ for $\sigma \in \mathfrak{M}$ and $\tau \in \mathfrak{N}$,
		\item $E(\sigma) <_{\fin} \vartheta(\tau) \Rightarrow E(f(\sigma)) <_{\fin} \vartheta(f(\tau))$ for $\sigma, \tau \in \mathfrak{N}$,
		\item $\vartheta(\sigma) \leq_{\fin} E(\tau) \Rightarrow \vartheta(f(\sigma)) \leq_{\fin} E(g(\tau))$ for $\sigma \in \mathfrak{N}$ and $\tau \in \mathfrak{M}$,
		\item $\vartheta(\sigma) \leq_{\fin} E(\tau) \Rightarrow \vartheta(f(\sigma)) \leq_{\fin} E(f(\tau))$ for $\sigma, \tau \in \mathfrak{N}$,
		\item $\vartheta(f(\sigma)) < \vartheta(f(\tau))$ for $\sigma, \tau \in \mathfrak{N}$ with $\vartheta(\sigma) < \vartheta(\tau)$.
	\end{enumerate}
	We give a short overview of these claims:
	Our main goal is to prove (a) and (b), i.e.~that $f$ and $\vartheta \circ g$ are embeddings. For this, we also need to show (g), i.e.~that $\vartheta \circ f$ behaves like an embedding if we consider the order induced by $\vartheta$ on~$\mathfrak{N}$. Finally, all other parts of the induction (c), (d), (e), (f) concern the sets $E(\sigma)$ for $\sigma \in \mathfrak{N}$, which we need during the proof of (g). The statements are ordered in a way such that we may use earlier claims without having to rely on the induction hypothesis.

	At the end, the first statement of our lemma will simply follow from (b). The second statement, i.e., that we can embed $\bachmann$ into the suborder of $\mathfrak{C}$ that only consists of $\vartheta$-terms, is given by the map $\sigma \mapsto \vartheta(g(\sigma))$, i.e.~(a).
	\begin{enumerate}[leftmargin=*, label=(\alph*)]
		\item If $\sigma = 0$, the claim is clear. If $\sigma = \vartheta(\sigma')$ and $\tau = \vartheta(\tau')$, then (g) yields $\vartheta(f(\sigma')) < \vartheta(f(\tau'))$. This entails $\vartheta(\Omega^{\overline{2}} \cdot \vartheta(f(\sigma'))) < \vartheta(\Omega^{\overline{2}} \cdot \vartheta(f(\tau')))$.
		
		If $\sigma = \vartheta(\sigma')$ and $\tau = \omega^{\gamma} + \delta$, then $\sigma \leq \gamma$ must hold. From (a), we get $\vartheta(g(\sigma)) \leq \vartheta(g(\gamma))$ and, thus, $\vartheta(g(\sigma)) < \vartheta(g(\omega^{\gamma} + \delta))$.
		In the reverse case, if $\sigma = \omega^{\gamma} + \delta$ and $\tau = \vartheta(\tau')$, then $\gamma < \tau$ and $\delta < \tau$ hold. From (a), we have $\vartheta(g(\gamma)) < \vartheta(g(\tau))$ and $\vartheta(g(\delta)) < \vartheta(g(\tau))$. In order to derive $\vartheta(g(\omega^{\gamma} + \delta)) < \vartheta(g(\tau))$, we want to show $g(\omega^{\gamma} + \delta) < g(\tau)$, which is clear by definition of $g$, and $E(g(\omega^{\gamma} + \delta)) <_{\fin} \vartheta(g(\tau))$. In order to prove the latter inequality, we consider the elements of $E(g(\omega^{\gamma} + \delta))$: $\overline{0}$, $\overline{1}$, $\vartheta(g(\gamma))$, and $\vartheta(g(\delta))$. For the former two, this is straightforward. For the latter two, we have already shown the inequalities above using (a).
		
		Finally, consider $\sigma = \omega^{\gamma} + \delta$ and $\tau = \omega^{\gamma'} + \delta'$. Clearly, we have $\gamma < \tau$ and $\delta < \tau$, which entails $\vartheta(g(\gamma)) < \vartheta(g(\tau))$ and $\vartheta(g(\delta)) < \vartheta(g(\tau))$ using~(a). This establishes $E(g(\sigma)) <_{\fin} \vartheta(g(\tau))$. The only thing left for showing the inequality $\vartheta(g(\sigma)) < \vartheta(g(\tau))$ is $g(\sigma) < g(\tau)$. For this, there are two cases: If we have $\gamma < \gamma'$, then by a), $\vartheta(g(\gamma)) < \vartheta(g(\gamma'))$ holds. Otherwise, if we have $\gamma = \gamma'$ and $\delta < \delta'$, then by (a), $\vartheta(g(\delta)) < \vartheta(g(\delta'))$ holds. In both cases, we arrive at our claim $g(\sigma) < g(\tau)$.
		\item By induction hypothesis (b), we know that the $\Omega$-exponents in $f(\sigma)$ and $f(\tau)$ are strictly descending. They are, therefore, proper $\Omega$-normal forms (with collapsed coefficients).
		
		If $\tau$ is a proper extension of $\sigma$, then the same holds for $f(\tau)$ and $f(\sigma)$. Otherwise, let $i$ be the smallest index such that the $i$-th summand $\Omega^{\gamma} \cdot \delta$ of $\sigma$ is strictly less than the $i$-th summand $\Omega^{\gamma'} \cdot {\delta'}$ of $\tau$. If we have $\gamma < \gamma'$, then by induction hypothesis (b), this yields $f(\gamma) < f(\gamma')$ and, therefore, the inequality $\Omega^{f(\gamma)} \cdot \vartheta(g(\delta)) < \Omega^{f(\gamma')} \cdot \vartheta(g(\delta'))$. If we have $\gamma = \gamma'$ and $\delta < \delta'$, then by (a), this yields $\vartheta(g(\delta)) < \vartheta(g(\delta'))$ and thus $\Omega^{f(\gamma)} \cdot \vartheta(g(\delta)) < \Omega^{f(\gamma')} \cdot \vartheta(g(\delta'))$. Note that the lengths of $\delta$ and $\delta'$ may be identical to those of $\sigma$ and $\tau$, respectively, since these are only shorthand for $\Omega^0 \cdot \delta$ and $\Omega^0 \cdot \delta'$. This is no problem, however, since we are not invoking the induction hypothesis but simply the the fact that we have already shown (a) for this sum of lengths. For reasons like this, it is important that our claims have their specific order.
		\item If $\sigma = 0$, this is clear. If $\sigma = \vartheta(\sigma')$, we need to show that both $\overline{2}$ and $\vartheta(f(\sigma'))$ are strictly below $\vartheta(f(\tau))$. The former is clear. The latter follows from $\sigma \in E(\sigma)$ and (g). If $\sigma = \omega^{\gamma} + \delta$, then we have to check that all of $\overline{0}$, $\overline{1}$, $\vartheta(g(\gamma))$, and $\vartheta(g(\delta))$ are strictly below $\vartheta(f(\tau))$. Similar to before, this is clear for the former two. For the latter two, we recall $g(\gamma), g(\delta) < f(\tau)$ and arrive at the claim using (d) together with the fact $E(\gamma) \cup E(\delta) \subseteq E(\sigma)$.
		\item Any element of $E(f(\sigma))$ is either equal to $\overline{3}$ or $\vartheta(g(\sigma'))$ for a $\sigma' \in \mathfrak{M}$ with $E(\sigma') \subseteq E(\sigma)$. In the former case, it is easy to see that $\overline{3} <_{\fin} \vartheta(f(\tau))$ holds. In the latter case, we apply (c), which yields $E(g(\sigma')) <_{\fin} \vartheta(f(\tau))$ and, thus, $\vartheta(g(\sigma')) < \vartheta(f(\tau))$.
		\item Clearly, $\tau$ cannot be equal to zero. If $\tau = \vartheta(\tau')$, then we have $\vartheta(\sigma) \leq \vartheta(\tau')$, which implies $\vartheta(f(\sigma)) \leq \vartheta(f(\tau))$ using (g). From here, the claim follows immediately. If $\tau = \omega^{\gamma} \cdot \delta$, then we have $\vartheta(\sigma) \leq_{\fin} E(\gamma)$ or $\vartheta(\sigma) \leq_{\fin} E(\delta)$. In both cases, (e) yields $\vartheta(f(\sigma)) \leq_{\fin} E(g(\gamma))$ or $\vartheta(f(\sigma)) \leq_{\fin} E(g(\delta))$. Either implies our claim.
		\item As before, $\tau$ cannot be zero. Let $\tau = \Omega^{\gamma_0} \cdot \delta_0 + \dots + \Omega^{\gamma_{n-1}} \cdot \delta_{n-1}$. If we have $\vartheta(\sigma) \leq_{\fin} E(\tau)$, then there must be some index $i < n$ satisfying one of the inequalities $\vartheta(\sigma) \leq_{\fin} E(\gamma_i)$ or $\vartheta(\sigma) \leq_{\fin} E(\delta_i)$. In the former case, we apply~(f), which yields $\vartheta(f(\sigma)) \leq_{\fin} E(f(\gamma_i)) \subseteq E(f(\tau))$. In the latter case, we apply~(e), which entails $\vartheta(f(\sigma)) \leq_{\fin} E(g(\delta_i))$ and, thus, the inequality ${\vartheta(f(\sigma)) < \vartheta(g(\delta_i)) \leq_{\fin} E(f(\tau))}$.
		\item If $\sigma < \tau$ and $E(\sigma) <_{\fin} \vartheta(\tau)$, then by (b), we have $f(\sigma) < f(\tau)$ and from (d), we get $E(f(\sigma)) <_{\fin} \vartheta(f(\tau))$. This entails $\vartheta(f(\sigma)) < \vartheta(f(\tau))$. Otherwise, if $\vartheta(\sigma) \leq_{\fin} E(\tau)$, then by (f), we get $\vartheta(f(\sigma)) \leq_{\fin} E(f(\tau))$ and, thus, the same claim $\vartheta(f(\sigma)) < \vartheta(f(\tau))$.\qedhere
	\end{enumerate}
\end{proof}

\begin{lemma}[$\rca_0$]\label{lem:subset_of_psi_G_is_fp_of_omega_G}
	Let $G$ be the Goodstein dilator. There is a $1$-fixed point of $\omega \circ G$ that can be defined on a suborder of $\psi_1(G)$.
\end{lemma}

\begin{proof}
	Let $\pi: \psi_1(G) \to G(\psi_1(G))$ be the collapse of the canonical $1$-fixed point~$\psi_1(G)$ (and let $\pi^+: \psi_1^+(G) \to G(\psi_1^+(G))$ be the corresponding isomorphism).
	Before we begin, we need substitutes for the natural numbers in $\psi_1(G)$ (similar to those used in the proof of Lemma \ref{lem:embedding_omega_normal_collapsed}). We define
	\begin{equation*}
		\overline{0} := \pi^{-1}(0) \quad \text{and} \quad \overline{n + 1} := \pi^{-1}((1 + \psi_1(G))^0 \cdot (1 + \overline{n}))
	\end{equation*}
	for any $n \in \n$.
	Using induction, we can simultaneously see that $\overline{n} < \overline{n + 1}$ holds for all $n \in \n$ and that all our applications of $\pi^{-1}$ are valid.
	
	Consider the following suborder $S^+ \subseteq \psi_1^+(G)$:
	The order $S^+$ contains elements $s$ of the form
	\begin{equation*}
		\pi^+(s) := (1 + \psi_1^+(G))^{s_0} \cdot (1 + \overline{m_0}) + \dots + (1 + \psi_1^+(G))^{s_{n-1}} \cdot (1 + \overline{m_{n-1}})
	\end{equation*}
	with $s_i \in G(\psi_1^+(G))$ and $m_i \in \n$ for all $i < n$ such that the additional condition $\supp^G_{\psi_1^+(G)}(\pi^+(s)) \subseteq S^+$ holds.
	
	Using induction along the height of terms, we can see that it is decidable whether some element lives in $S^+$ or not. We will define our $1$-fixed point on $S := S^+ \cap \psi_1(G)$. For the collapse, we use the following function $\kappa^+: S^+ \to (\omega \circ G)(S^+)$:
	\begin{equation*}
		\kappa^+(s) := \omega^{s'_0} \cdot (1 + m_0) + \dots + \omega^{s'_{n-1}} \cdot (1 + m_{n-1})\comma\footnotemark{}
	\end{equation*}
	where $s$ is just our exemplary element from the definition of $S^+$ and $s'_i \in G(S^+)$ are uniquely given by $G(\iota_{S^+}^{\psi_1^+(G)})(s'_i) = s_i$ for all $i < n$. Note that the latter is possible because of the restriction the we imposed on the support of $s$ during the definition of $S^+$.\footnotetext{For ease of notation, we write this term for the sequence in $\omega^{G(S^+)}$ given by $(1 + m_0)$-many repetitions of $s'_0$, followed by $(1 + m_{1})$-many repetitions of $s'_1$, \dots, until it finishes with $(1 + m_{n-1})$-many repetitions of $s'_{n-1}$.}
	
	It can be seen rather quickly that $\kappa^+$ is an isomorphism. Now, we define the map $\kappa: S \to (\omega \circ G)(S)$ to be the unique embedding such that the equality ${\kappa^+ \circ \iota_S^{S^+} = (\omega \circ G)(\iota_S^{S^+}) \circ \kappa}$ holds. The existence of $\kappa$ is guaranteed since our support satisfies $\supp_{S^+}^{\omega \circ G}(\kappa^+(s)) \subseteq \psi_1(G)$ and, therefore, $\supp_{S^+}^{\omega \circ G}(\kappa^+(s)) \subseteq S$ holds for any $s \in S$.

	In order to prove our claim, we only have to show that $\kappa$ makes $S$ a $1$-fixed point. The first requirement is that we can construct a height function for $S$ that respects the relation $\triangleleft_S$ that satisfies $s \triangleleft_S t$ if and only if $s \in \supp_S^{\omega \circ G}(\kappa(t))$ holds, for arbitrary $s, t \in S$. For our height function, we simply restrict $h: \psi_1(G) \to \n$ of $\psi_1(G)$ to $S$. Now, let $s, t \in S$ be arbitrary with $s \triangleleft_S t$. Then, $s$ is an element of $\supp_S^{\omega \circ G}(\kappa(t)) = \supp_{S^+}^{\omega \circ G}(\kappa^+(t))$. Assume that $\pi^+(t)$ is of the form
	\begin{equation*}
		\pi^+(t) := (1 + \psi_1^+(G))^{t_0} \cdot (1 + \overline{k_0}) + \dots + (1 + \psi_1^+(G))^{t_{l-1}} \cdot (1 + \overline{k_{l-1}})
	\end{equation*}
	and let $t'_i$ be such that $G(\iota_{S^+}^{\psi_1^+(G)})(t'_i) = t_i$ holds for all $i < l$. Thus, we have
	\begin{equation*}
		\kappa^+(t) := \omega^{t'_0} \cdot (1 + k_0) + \dots + \omega^{t'_{l-1}} \cdot (1 + k_{l-1})\period
	\end{equation*}
	From $s \in \supp_{S^+}^{\omega \circ G}(\kappa^+(t))$, we conclude $s \in \supp_{S^+}^{G}(t'_i)$ for some $i < l$. Thus, $s \in \supp_{\psi_1^+(G)}^{G}(\pi^+(t))$ holds. Since we have $t \in S \subseteq \psi_1(G)$, this entails that $s$ is an element of $\supp_{\psi_1(G)}^{G}(\pi(t))$, i.e., $s \triangleleft t$ and, finally, $h(s) < h(t)$.
	
	Next, we have to verify the condition on the range of $\kappa$. For this, we consider some arbitrary $\sigma \in (\omega \circ G)(S)$ and prove that there exists some $s \in S$ with $\kappa(s) = \sigma$ if and only if $G^S_0(\sigma) < \sigma$ holds, for $G^S_0(\sigma) := \{\kappa(s) \mid s \in \supp_S^{\omega \circ G}(\sigma)\}$.
	
	For the first direction, we take some arbitrary $s \in S$ and prove $G^S_0(\kappa(s)) < \kappa(s)$. Let $t \in S$ be such that $\kappa(t) \in G^S_0(\kappa(s))$ holds. We want to derive $\kappa(t) < \kappa(s)$. By definition, we have $t \in \supp_S^{\omega \circ G}(\kappa(s))$. Using the same argument as above, this entails $t \triangleleft s$ and, hence, $t < s$. We conclude $\kappa(t) < \kappa(s)$.
	
	Finally, for the other direction, we assume $G^S_0(\sigma) < \sigma$ and try to find some $s \in S$ with $\kappa(s) = \sigma$. Now, since $\kappa^+$ is an isomorphism, it is straightforward to find some $s \in S^+$ with $\kappa^+(s) = (\omega \circ G)(\iota_S^{S^+})(\sigma)$. Also, if we can show $s \in \psi_1(G)$, which entails $s \in S$, then the definition of $\kappa$ yields our claim $\kappa(s) = \sigma$.
	
	Since the support of $\sigma$ is a subset of $S$, the same must holds for $s$ and also for $\pi^+(s)$ with respect to $\psi_1(G)$. For the latter, we use the fact that $\overline{m_i}$ is an element of $\psi_1(G)$ for each $i < n$. Thus, we can find a unique $\tau \in G(\psi_1(G))$ with $G(\iota_{\psi_1(G)}^{\psi_1^+(G)})(\tau) = \pi^+(s)$. In the next paragraph, we prove $G_0(\tau) <_{\fin} \tau$. If this holds, we can find a unique $t \in \psi_1(G)$ with $\pi(t) = \tau$. With this, we will conclude $s = t \in \psi_1(G)$ using the fact $\pi^+ \circ \iota_{\psi_1(G)}^{\psi_1^+(G)} = G(\iota_{\psi_1(G)}^{\psi_1^+(G)}) \circ \pi$.
	
	We prove $G_0(\tau) <_{\fin} \tau$. Let $x \in \psi_1(G)$ be arbitrary with $\pi(x) \in G_0(\tau)$. Our claim is that $\pi(x) < \tau$ holds, but we will show $x < s$ first. By assumption, we have $x \in \supp_{\psi(G)}^{G}(\tau) = \supp_{\psi_1^+(G)}^G(\pi^+(s))$. Therefore, we have either $x = \overline{m_i}$ or $x \in \supp_{\psi_1^+(G)}^G(s_i)$ for some index $i < n$. In the first case
	\begin{equation*}
		\pi^+(\overline{m_i}) < \pi^+(\overline{m_i + 1}) = (1 + \psi_1^+(G))^0 \cdot (1 + \overline{m_i}) \leq \pi^+(s)
	\end{equation*}
	implies $x = \overline{m_i} < s$. In the second case, we have
	\begin{equation*}
		x \in \supp_{\psi_1(G)}^G(s_i) \subseteq \supp_{S^+}^{\omega \circ G}(\kappa^+(s)) = \supp_S^{\omega \circ G}(\sigma)\period
	\end{equation*}
	Thus, $\kappa(x) < \sigma$ holds by our assumption $G^S_0(\sigma) < \sigma$. By definition of $\kappa$ and $s$, this entails $\kappa^+(x) < \kappa^+(s)$ and, hence, $x < s$.
	
	Now that we have seen that both cases yield $x < s$, we can derive our actual claim. The inequality implies $\pi^+(x) < \pi^+(s)$. Now, by definition of $\tau$ and the fact $\pi^+ \circ \iota_{\psi_1(G)}^{\psi_1^+(G)} = G(\iota_{\psi_1(G)}^{\psi_1^+(G)}) \circ \pi$, we arrive at $\pi(x) < \tau$.
\end{proof}
For the proof of Proposition \ref{prop:goodstein_bachmann_second_order}, we will employ results about Bachmann-Howard fixed points. We have already given a set-theoretic definition in Section \ref{sec:set_theory}. Now, we need a variant that is compatible with second order arithmetic:
\begin{definition}
	Given a predilator $D$, we call a linear order $X$ a \emph{Bachmann-Howard fixed point} if there exists a map
	\begin{equation*}
		\vartheta: D(X) \to X
	\end{equation*}
	satisfying the following for any $\sigma, \tau \in D(X)$:
	\begin{enumerate}[label=(\roman*)]
		\item If $\sigma < \tau$ and $\supp^D_X(\sigma) <_{\fin} \vartheta(\tau)$, then we have $\vartheta(\sigma) < \vartheta(\tau)$.
		\item We have $\supp^D_X(\sigma) <_{\fin} \vartheta(\sigma)$.
	\end{enumerate}
	We call $\vartheta$ a \emph{Bachmann-Howard collapse}.
\end{definition}

\begin{proof}[Proof of Proposition \ref{prop:goodstein_bachmann_second_order}]
	Let $C$ denote the suborder of $\vartheta$-terms in $\mathfrak{C}$. With the help of Lemma~\ref{lem:embedding_omega_normal_collapsed}, we can embed $\bachmann$ into $C$. We begin by constructing a Bachmann-Howard collapse and define $\kappa : G(C) \to \mathfrak{C}$:
	\begin{multline*}
		\kappa((1+C)^{\gamma_0} \cdot (1 + c_0) + \dots + (1+C)^{\gamma_{n-1}} \cdot (1 + c_{n-1})) :=\\ \Omega^{\kappa(\gamma_0)} \cdot c_0 + \dots + \Omega^{\kappa(\gamma_{n-1})} \cdot c_{n-1}\comma
	\end{multline*}
	where $\gamma_i \in G(C)$ and $c_i \in C$ holds for all $i < n$ with $\gamma_0 > \dots > \gamma_{n-1}$.
	Let us show that the map $\gamma \mapsto \vartheta(\kappa(\gamma))$ is a collapse. For this, consider two elements $\sigma, \tau \in C$ with $\sigma < \tau$ and $\supp_{C}(\sigma) <_{\fin} \vartheta(\kappa(\tau))$. First, we can easily see that $\kappa(\sigma) < \kappa(\tau)$ holds using induction on $l(\sigma) + l(\tau)$. In order to derive $\vartheta(\kappa(\sigma)) < \vartheta(\kappa(\tau))$, we need to show $E(\kappa(\sigma)) <_{\fin} \vartheta(\kappa(\tau))$. This easily follows from $E(\kappa(\sigma)) = \supp_{C}(\sigma)$.
	
	For the second condition of Bachmann-Howard collapses, we have to prove $\supp_{C}(\sigma) <_{\fin} \vartheta(\kappa(\sigma))$ for $\sigma \in G(C)$. Using $E(\kappa(\sigma)) = \supp_{C}(\sigma)$ again, this becomes apparent immediately.
	
	Now, we show that our Bachmann-Howard fixed point is initial. For this, we use Theorem 2.9 of \cite{Freund22}: The requirements are that our collapse $\gamma \mapsto \vartheta(\kappa(\gamma))$ is surjective and that there exists a function $l : C \to \n$ with $l(c) < l(\vartheta(\kappa(\sigma)))$ for any $\sigma \in G(C)$ and $c \in \supp_{C}(\sigma)$. The former can be seen using a short induction and for the latter, we can simply take the length function $l$ associated with $\bachmann$.
	
	Next, Theorem 4.2 of \cite{FR21} says that any $1$-fixed point of $\omega \circ G$ is a Bachmann-Howard fixed point of $G$. Since $C$ is initial, it embeds into any $1$-fixed point of $\omega \circ G$ and, using Lemma \ref{lem:subset_of_psi_G_is_fp_of_omega_G}, into (a suborder of) $\psi_1(G)$.
	
	For the converse direction, we use our Bachmann-Howard collapse $\kappa: G(C) \to C$ from before. According to Theorem 4.1 of \cite{FR21}, $C$ must have a $1$-fixed point of $G$ as suborder. Because of uniqueness, any $1$-fixed point of $G$ can be embedded into~$C$. Also, as a suborder of $\mathfrak{C}$, the order $C$ clearly embeds into $\bachmann$. We conclude that any $1$-fixed point of $G$ can be embedded into~$\bachmann$.
\end{proof}

\section{The weak inverse Goodstein sequence}\label{sec:weak_inverse_Goodstein}
There is a weak variant of Goodstein sequences that has been studied (cf.~\cite{Cichon83}). In contrast to the original construction, the base change in the weak variant is only a regular one, i.e., we apply a non-hereditary base change. E.g., if our initial value is $5$, then the second member of the ensuing weak Goodstein sequence computes as follows:
\begin{equation*}
	5 \quad \rightsquigarrow \quad 2^{2} + 2^0 \quad \rightsquigarrow \quad 3^{2} + 3^0 \quad \rightsquigarrow \quad 10 \quad \rightsquigarrow \quad 9\period
\end{equation*}
That this procedure terminates, i.e.~eventually reaches zero for any initial value, cannot be proved in $\rca_0$. However, already adding the assumption that $\omega^\omega$ is well founded suffices for a proof.

Similar to the original observation from the introduction, the regular base change can also be realized by a (pre)dilator:
\begin{definition}
	We define an endofunctor $W$ on the category of linear orders, as follows:
	Given a linear order $X$, the set $W(X)$ contains terms
	\begin{equation*}
		x := (1 + X)^{m_0} \cdot (1 + \delta_0) + \dots + (1 + X)^{m_{n-1}} \cdot (1 + \delta_{n-1})
	\end{equation*}
	for $m_i \in \n$ and $\delta_i \in X$ with $i < n$. Moreover, demand $m_0 > \dots > m_{n-1}$. Consider another term
	\begin{equation*}
		y := (1 + X)^{m'_0} \cdot (1 + \delta'_0) + \dots + (1 + X)^{m'_{m-1}} \cdot (1 + \delta'_{m-1})
	\end{equation*}
	in $W(X)$. We have $x < y$ if and only if one of the following holds:
	\begin{enumerate}[label=(\roman*)]
		\item The term $y$ is a proper extension of $x$, i.e., $m > n$ holds with $m_i = m'_i$ and $\delta_i = \delta'_i$ for all $i < n$.
		\item There is some $i < \min(n, m)$ with $m_i < m'_i$ or both $m_i = m'_i$ and $\delta_i < \delta'_i$. Moreover, we have $m_j = m'_j$ and $\delta_j = \delta'_j$ for all $j < i$.
	\end{enumerate}
	Given a morphism $f: X \to Y$ between two linear orders $X$ and $Y$, we map our element $x$ from above to
	\begin{equation*}
		f(x) := (1 + Y)^{m_0} \cdot (1 + f(\delta_0)) + \dots + (1 + Y)^{m_{n-1}} \cdot (1 + f(\delta_{n-1}))\period
	\end{equation*}
	Finally, we define $\supp_X(x) := \{\delta_i \mid i < n\}$. We call $W$ the \emph{weak Goodstein predilator} and the increasing $W$-sequence is the \emph{weak inverse Goodstein sequence}.
\end{definition}
With this, we can express the construction of weak Goodstein sequences as follows: Given an initial number $m \in \n$, we define our sequence $(w_n)_{n \in \n}$ by
\begin{itemize}
	\item $w_0 := m$,
	\item $w_{n+1}:$ If $w_n > 0$, then we set $w_{n+1} := W(\iota_{n+1}^{n+2})(w_n) - 1$. Otherwise, if $w_n = 0$, we set $w_{n+1} := 0$.
\end{itemize}
Of course, we are now interested in the termination point of the increasing $W$-sequence. We can already expect it to be smaller than $\bachmann$. In fact, it will turn out to be much smaller.
The following definition provides a notation system for $\varphi_\omega(0)$ directly inspired by~\cite{RW11}.
\begin{definition}
	We define the term system $\phiomega$ simultaneously with a head function $h: \phiomega \to \n$, a set of indecomposable terms $H \subseteq \phiomega$, and an order:
	\begin{itemize}
		\item $0$ with $h(0) := 0$ and $0 \notin H$.
		\item $x := x_0 + \dots + x_{n-1}$ for $n > 1$, $x_i \in H$ for all $i < n$ with $x_0 \geq \dots \geq x_{n-1}$. We have $h(x) := 0$ and $x \notin H$.
		\item $t := \varphi(n, x)$ for $n \in \n$, $x \in \phiomega$ and $h(x) \leq n$. We have both $h(t) := n$ and $t \in H$.
	\end{itemize}
	The linear order on $\phiomega$ is given as follows:
	\begin{itemize}
		\item $0 < x$ for all $x \in \phiomega$ with $x \neq 0$,
		\item $x_0 + \dots + x_{n-1} < y_0 + \dots + y_{m-1}$ if and only if either $n < m$ and $x_i = y_i$ for all $i < n$, or there is some $i < n$ with $x_i < y_i$ and $x_j = y_j$ holds for all~$j < i$,
		\item $x_0 + \dots + x_{n-1} < \varphi(m, y)$ if and only if $x_0 < \varphi(m, y)$,
		\item $\varphi(n, x) < y_0 + \dots + y_{m-1}$ if and only if $\varphi(n, x) \leq y_0$,
		\item $\varphi(n, x) < \varphi(m, y)$ if and only if one of the three holds:
		\begin{itemize}
			\item $n < m$ and $x < \varphi(m, y)$,
			\item $n = m$ and $x < y$, or
			\item $n > m$ and $\varphi(n, x) < y$.
		\end{itemize}
	\end{itemize}
\end{definition}

The following proposition is very similar to Propositions~\ref{prop:goodstein_bachmann} and \ref{prop:goodstein_bachmann_second_order}: We calculate the termination point of $W$ using equimorphisms:
\begin{proposition}[$\rca_0$]\label{prop:weak_goodstein_phiomega}
	Any strong $W$-sequence termination point is equimorphic to $\varphi(\omega, 0)$. In particular, there is a well founded $W$-sequence termination point if and only if $\varphi(\omega, 0)$ is well founded.
\end{proposition}
Again, we can ask ordinal analysis which systems may or may not prove the termination of weak inverse Goodstein sequences. E.g., let us consider the systems of $\deltaonecr$ and $\deltaoneca$, i.e., extensions of $\rca_0$ with full induction together with a comprehension rule and an comprehension axiom, respectively, for $\Delta^1_1$-formulas:
\begin{corollary}[$\rca_0$]
	The system of $\deltaonecr$ cannot prove that the weak inverse Goodstein sequence terminates, i.e., that there is a well founded $W$-sequence termination point. However, already the system of $\deltaoneca$ can prove this fact.
\end{corollary}

\begin{proof}
	See \cite{sep-proof-theory}, we use the proof-theoretic ordinals of $\deltaonecr$ and $\deltaoneca$, which are equal to $\varphi_\omega(0)$ and $\varphi_{\varepsilon_0}(0)$, respectively.
\end{proof}

Of course, our results can be transferred to set theory since $\rca_0$ is sound:

\begin{corollary}[$\zfc$]
	The weak inverse Goodstein sequence, i.e.~the increasing $\widetilde{W}$-sequence (in the sense of Definition \ref{def:zfc_inverse_d_sequence}), terminates with $\nu(\widetilde{W}, 0) = \varphi_\omega(0)$.
\end{corollary}
Similar to before, we need to construct two embeddings: One from $\phiomega$ into our termination point and another from the termination point into $\phiomega$. Again, we actually work with the $1$-fixed point of $W$ instead of its termination point using Theorem~\ref{thm:d_sequence_fp_is_1_fp}. Previously, the latter direction was almost trivial. Now, this is not the case anymore. However, in Lemma~\ref{lem:embed_1_fp_weak_phiomega}, where we prove this result, we can directly use ideas from \cite{RW93}.

The former direction is more involved. Like in the previous section, we give a short overview of our steps:
\begin{itemize}
	\item We begin by proving that $\phiomega$ embeds into the initial Bachmann-Howard fixed point of a certain predilator $D$. (Lemma~\ref{lem:phiomega_embeds_into_bhfp_of_D})
	\item We define a notation system $\mathfrak{P}$ together with a suborder $\psiOmegaomega$. As the name suggests, this is supposed to be isomorphic to the linear order in Buchholz's notation. However, our version of $\psiOmegaomega$ is already in a certain normal form, similar to the $\Omega$-normal form with collapsed coefficients from the previous section. (Definition~\ref{def:notation_system_frak_P_and_psiOmegaomega})
	\item We show that a suborder of $\psiOmegaomega$ is a $1$-fixed point of $\omega \circ D$. Using the connection between Bachmann-Howard fixed points and $1$-fixed points, this yields an embedding from $\phiomega$ into $\psiOmegaomega$. (Lemma~\ref{lem:embed_phiomega_psiOmegaomega})
	\item We construct an embedding from $\psiOmegaomega$ into the $1$-fixed point of $W$. Together with the previous step, this yields our claim. (Lemma~\ref{lem:embed_psiOmegaomega_1_fp_weak})
\end{itemize}

\begin{lemma}[$\rca_0$]\label{lem:phiomega_embeds_into_bhfp_of_D}
	The order $\phiomega$ embeds into an initial Bachmann-Howard fixed point of the predilator $D(X) := \omega^X + \omega \times X$.
\end{lemma}
We understand this predilator to map any linear order $X$ to terms
\begin{enumerate}[label=(\roman*)]
	\item $\langle 0, \sigma \rangle$ with $\sigma \in \omega^X$,
	\item $\langle 1, \langle n, x \rangle \rangle$ with $n \in \n$ and $x \in X$.
\end{enumerate}
The order is given lexicographically. Morphisms $f: X \to Y$ map terms $\langle 0, \sigma \rangle$ to $\langle 0, (\omega^f)(\sigma) \rangle$ and terms $\langle 1, \langle n, x \rangle \rangle$ to $\langle 1, \langle n, f(x) \rangle \rangle$. Finally, the former term has support $\supp_{X}^{\omega^X}(\sigma)$ and the latter term has support $\{x\}$.
\begin{proof}
	Let $\vartheta: D(\vartheta(D)) \to \vartheta(D)$ be the collapse for the canonical (and initial) Bachmann-Howard fixed point $\vartheta(D)$.
	We define a map $f: \phiomega \to D(\vartheta(D))$:
	\begin{align*}
		f(0) &:= \langle 0, \langle \rangle \rangle\comma\\
		f(x_0 + \dots + x_{n-1}) &:= \langle 0, \langle \vartheta(f(x_0)), \dots, \vartheta(f(x_{n-1})) \rangle \rangle\comma\\
		f(\varphi(n, x)) &:= \langle 1, \langle n, \vartheta(f(x)) \rangle \rangle\period
	\end{align*}
	Again, we will only see during the following induction that $f$ is well defined on sums since $\omega^X$ requires (weakly) descending sequences for linear orders $X$.
		
	Let us show that $\vartheta \circ f$ is an embedding: Assume $x, y \in \phiomega$ with $x < y$ and assume that we have already shown that $\vartheta \circ f$ is an embedding for any two terms whose sum of heights is shorter than that of $x$ and $y$. The case where $x = 0$ holds is clear. If $x$ and $y$ are sums with $x = x_0 + \dots + x_{n-1}$, we have $x_i < y$ for any $i < n$ by transitivity. By induction hypothesis, this yields $\vartheta(f(x_i)) < \vartheta(f(y))$ for any $i < n$. Combining this with $f(x) < f(y)$, which clearly follows from $x < y$, the induction hypothesis, and both $x$ and $y$ being sums, we conclude $\vartheta(f(x)) < \vartheta(f(y))$.
	
	Let $x = x_0 + \dots + x_{n-1}$ and $y = \varphi(m, y')$. This implies $x_i < \varphi(m, y')$ for $i < n$. By induction hypothesis, we have $\vartheta(f(x_i)) < \vartheta(f(\varphi(m, y')))$ for ${i < n}$. Thus, the inequalities $\supp_{\vartheta(D)}(f(x)) <_{\fin} \vartheta(f(\varphi(m, y'))$ and $f(x) < f(\varphi(m, y'))$ yield our claim ${\vartheta(f(x)) < \vartheta(f(y))}$.
	
	Let $x = \varphi(n, x')$ and $y = y_0 + \dots + y_{m-1}$. This implies $x \leq y_0$. By induction hypothesis, we have $\vartheta(f(x)) \leq \vartheta(f(y_0))$ and, thus, $\vartheta(f(x)) \leq_{\fin} \supp(f(y))$. This results in $\vartheta(f(x)) < \vartheta(f(y))$.
	
	Let $x = \varphi(n, x')$ and $y = \varphi(m, y')$. In the first case, we have $n < m$ and $x' < y$. From the former, we have $f(x) < f(y)$. From the latter and the induction hypothesis, we get $\supp_{\vartheta(D)}(f(x)) <_{\fin} \vartheta(f(y))$. This results in $\vartheta(f(x)) < \vartheta(f(y))$.
	In the second case, we have $n = m$ and $x' < y'$. By induction hypothesis, we have $\vartheta(f(x')) < \vartheta(f(y'))$. This yields both $f(x) < f(y)$ and the chain of inequalities $\supp_{\vartheta(D)}(f(x)) <_{\fin} \vartheta(f(y')) \leq_{\fin} \supp_{\vartheta(D)}(f(y)) <_{\fin} \vartheta(f(y))$. Thus, we get ${\vartheta(f(x)) < \vartheta(f(y))}$.
	In the third and last case, we have $n > m$ and $x < y'$. By induction hypothesis, this yields $\vartheta(f(x)) < \vartheta(f(y')) \leq_{\fin} \supp_{\vartheta(D)}(f(y))$. Finally, we arrive at $\vartheta(f(x)) < \vartheta(f(y))$.
\end{proof}

\begin{remark}
	There is another dilator $E$ with $E(X) := 1 + (X \times \omega) \times X$ such that we can find an embedding (even an isomorphism) from $\varphi(\omega, 0)$ to $\vartheta(E)$ (cf.~\cite[Theorem~1.4]{FreundPredicative}). However, for technical reasons, we prefer $\omega^X$ in $D(X)$ over the pairs $X \times X$ present in $E(X)$ since the former comes with information on the order of its elements, whereas in the latter, we do not know how the first and second element in a pair are ordered.
\end{remark}

\begin{definition}\label{def:notation_system_frak_P_and_psiOmegaomega}
	We define the term system $\mathfrak{P}$ together with an order and a suborder $\psiOmegaomega \subseteq \mathfrak{P}$ simultaneously:
	\begin{itemize}
		\item $\mathfrak{P}$ contains elements $x := \Omega^{m_0} \cdot \psi(x_0) + \dots + \Omega^{m_{n-1}} \cdot \psi(x_{n-1})$ for $n \in \n$, $m_i \in \n$ and $x_i \in \psiOmegaomega$ for all $i < n$. Moreover, we require $m_i \geq m_j$ for $i < j < n$ and $x_i \geq x_j$ for indices $i < j < n$ with $m_i = m_j$.
		\item The order is the usual lexicographical one. For the coefficients, we have $\psi(x) < \psi(y)$ if and only if $x < y$ holds.
		\item If $x_i < x$ for all $i < n$, then $x \in \psiOmegaomega$.
	\end{itemize}
	As before, we denote the empty sum by $0$. We may write $\psi(x)$ as abbreviation for $\Omega^0 \cdot \psi(x)$ given $x \in \psiOmegaomega$.
	We define addition in the usual way: Given elements $x := \Omega^{m_0} \cdot \psi(x_0) + \dots + \Omega^{m_{n-1}} \cdot \psi(x_{n-1})$ and $y := \Omega^{k_0} \cdot \psi(y_0) + \dots + \Omega^{k_{l-1}} \cdot \psi(y_{l-1})$ from $\mathfrak{P}$, we have
	\begin{equation*}
		x + y := \Omega^{m_0} \cdot \psi(x_0) + \dots + \Omega^{m_i} \cdot \psi(x_i) + \Omega^{k_0} \cdot \psi(y_0) + \dots + \Omega^{k_{l-1}} \cdot \psi(y_{l-1})\comma
	\end{equation*}
	where $i < n$ is the largest index with $\Omega^{m_i} \cdot \psi(x_i) \geq \Omega^{k_0} \cdot \psi(y_0)$ (or $i := -1$ if such an index does not exist).
\end{definition}

\begin{lemma}[$\rca_0$]\label{lem:addition}
	Addition is associative.
	For any $x, y, z \in \mathfrak{P}$ with $y < z$, we have $x + y < x + z$.
\end{lemma}

\begin{proof}
	Straightforward induction and case distinction.
\end{proof}

\begin{lemma}[$\rca_0$]\label{lem:embed_phiomega_psiOmegaomega}
	There is an embedding from $\phiomega$ into $\psiOmegaomega$.
\end{lemma}

\begin{proof}
	Let $D$ be the predilator with $D(X) := \omega^X + \omega \times X$ from Lemma~\ref{lem:phiomega_embeds_into_bhfp_of_D}.
	We show that a suborder of $\psiOmegaomega$ is a $1$-fixed point of $\omega \circ D$. With \cite[Theorem~4.2]{FR21},the fact that every $1$-fixed point of $\omega \circ D$ is a Bachmann-Howard fixed point of $D$, this yields our claim: $\phiomega$ embeds into an initial Bachmann-Howard fixed point~$\vartheta(D)$, which then embeds into our $1$-fixed point of $\omega \circ D$, which is a suborder of $\psiOmegaomega$.
	
	In order to ease notation, let us write $\langle x_0, \dots, x_{n-1} \rangle$ and $\varphi(n, x)$ for the elements $\langle 0, \langle x_0, \dots, x_{n-1} \rangle \rangle$ and $\langle 1, (n, x) \rangle$, respectively, in $D(X)$.
	Next, we define a function $f: \omega^{\psiOmegaomega} \to \psiOmegaomega$ with
	\begin{align*}
		&f(\langle \rangle) := 0\comma\\
		&f(\langle x_0, \dots, x_{n-1} \rangle) := x_0 + \psi(x_0) + \dots + \psi(x_{n-1})\period
	\end{align*}
	Any value of $f$ is a valid element of $\psiOmegaomega$ since $x_i \leq x_0 \leq f(\langle x_0, \dots, x_{n-1} \rangle)$ holds for any $i < n$. Moreover, $f$ is an embedding: Consider the two sequences $x := \langle x_0, \dots, x_{n-1} \rangle$ and $y := \langle y_0, \dots, y_{m-1} \rangle$ of $\omega^{\psiOmegaomega}$. Assume $x < y$. If $x = 0$ holds, the claim is clear. Otherwise, if $x > 0$ holds, then we know $x_0 \leq y_0$. Additionally, we have $\psi(x_0) + \dots + \psi(x_{n-1}) < \psi(y_0) + \dots + \psi(y_{m-1})$. Combining both using Lemma \ref{lem:addition} yields $f(x) < f(y)$.
	
	The range of $f$ is a set: Given some arbitrary element $x \in \psiOmegaomega$, we go through all the finitely many possible sequences $\langle x_0, \dots, x_{n-1} \rangle \in \omega^{\psiOmegaomega}$ such that ${x = y + \psi(x_0) + \dots + \psi(x_{n-1})}$ holds for some $y \in \mathfrak{P}$. Then, $x$ is in the range of $f$ if and only if $x = f(\langle x_0, \dots, x_{n-1} \rangle)$ holds for any of these.
	
	For the $1$-fixed point, we define suborders $S^+ \subseteq \mathfrak{P}$ and $S \subseteq \psiOmegaomega$ with $S \subseteq S^+$ as follows: $S^+$ contains terms of the form
	\begin{align*}
		s :=\ &\Omega^{m_0 + 2} \cdot \psi(s_0) + \dots + \Omega^{m_{n-1} + 2} \cdot \psi(s_{n-1})\\
		+\ &\Omega \cdot \psi(f(t_0)) + \dots + \Omega \cdot \psi(f(t_{k-1}))
	\end{align*}
	for $n, k \in \n$ and numbers $m_0 \geq \dots \geq m_{n-1}$ with $s_i \in S$ for all $i < n$ and $t_0 \geq \dots \geq t_{k-1} \in \omega^S$.
	Notice that it can be decided whether some element lies in $S^+$ since the range of $f$ is a set.
	If $s_i < s$ holds for all $i < n$ and also $(t_j)_0 < s$ for all $j < k$ with $t_j \neq \langle\rangle$, then $s$ is an element of $S$.
	
	We want to verify that $S \subseteq \psiOmegaomega$ holds. Given our $s \in S$ from above, the only requirement left for this is $f(t_j) < s$ for all $j < k$. Thus, let $j < k$ be arbitrary. If $t_j = \langle \rangle$, we have $f(t_j) = 0 < \Omega \cdot \psi(f(t_j)) \leq s$. Otherwise, we already know $(t_j)_0 < s$ by assumption. Now, since $s$ only consists of summands $\Omega^\gamma \cdot \delta$ with $\gamma > 0$, we can add arbitrary $\psi$-terms to the left hand side of this inequality as we like. Thus, $f(t_j) < s$ holds. 
	
	We define a map $\pi: S \to (\omega \circ D)(S)$ such that we have the following for our exemplary $s$ from above:
	\begin{align*}
		\pi(s) := \langle& \varphi(m_0, s_0), \dots, \varphi(m_{n-1}, s_{n-1}), t_0, \dots, t_{k-1} \rangle\period
	\end{align*}
	First, the values $t_i$ for $i < k$ are well defined since $f$ is an embedding. Second, we have $\varphi(m_i, s_i) \geq_{D(S)} \varphi(m_j, s_j)$ for $i < j < n$ (recall, that this is in the context of $D(S)$, so we have a simple lexicographical order on these terms), $t_i \geq_{D(S)} t_j$ for $i < j < k$ (since $f$ is an embedding), and $\varphi(m_i, s_i) >_{D(S)} t_j$ for $i < n$ and $j < k$ (since $\varphi$-terms are always greater than sequences in the context of $D(S)$). We conclude that the value of $\pi(s)$ is an element of $(\omega \circ D)(S)$. Using the usual arguments, it can easily be seen that $\pi$ is an embedding.
	
	For the height function demanded by $1$-fixed points, we take the amount of symbols in our terms. Now, we show that restricting $\pi$ to terms in $S$ makes it a collapse. First, it is clear that any element of the support $\supp_S(\pi(s))$ is strictly smaller than $s \in S$ from our previous considerations and the fact that ${(t_i)_j \leq (t_i)_0 < s}$ holds for any $i < k$ (such that $t_i$ is different from $0$ and $j$ is less than the length of $t_i$).
	
	Finally, consider some term $\sigma \in (\omega \circ D)(S)$ with $G_0(\sigma) <_{\fin} \sigma$. From our previous considerations on the order of elements in $D(S)$, we know that it must be of the expected form
	\begin{equation*}
		\sigma = \langle \varphi(m_0, s_0), \dots, \varphi(m_{n-1}, s_{n-1}), t_0, \dots, t_{k-1}\rangle
	\end{equation*}
	with $\varphi(m_i, s_i) \geq_{D(S)} \varphi(m_j, s_j)$ for $i < j < n$, $t_i \geq_{D(S)} t_j$ for $i < j < k$, and $\varphi(m_i, s_i) \geq_{D(S)} t_j$ for $i < n$ and $j < k$. Therefore, we can already find some element $s \in S^+$ with $\pi(s) = \sigma$ (if we allow this extension of the domain of $\pi$ for the moment). Moreover, from $G_0(\sigma) <_{\fin} \sigma$ and $\pi$ being an embedding, we know that $s_i$ and $(t_j)_0$ must be strictly below $s$ for $i < n$ and $j < k$ such that $t_j \neq 0$. By definition, this results in $s \in S$.
\end{proof}

\begin{lemma}[$\rca_0$]\label{lem:embed_psiOmegaomega_1_fp_weak}
	There is an embedding from $\psiOmegaomega$ into any $1$-fixed point of the predilator $W$.
\end{lemma}

\begin{proof}
	Assume that $\pi: \psi_1(W) \to W(\psi_1(W))$ is the collapse of~$\psi_1(W)$.
	Let us define a map $f: \omega^{\psiOmegaomega} \to \psiOmegaomega$ similar to before:
	\begin{align*}
		&f(\langle \rangle) := 0\comma\\
		&f(\langle x_0, \dots, x_{n-1} \rangle) := x_0 + \psi(x_1) + \dots + \psi(x_{n-1})\period
	\end{align*}
	Notice, that we are omitting $\psi(x_0)$ in the sum of $\psi$-terms this time. Using the argument from before, it is clear that $f$ maps into $\psiOmegaomega$. However, omitting $\psi(x_0)$ has the effect that $f$ is not an embedding and not even injective anymore.\footnote{$f(\langle \Omega \cdot \psi(0) + \psi(0), \psi(0)\rangle) = \Omega \cdot \psi(0) + \psi(\psi(0)) = f(\langle \Omega \cdot \psi(0), \psi(0)\rangle)$} Still, for elements $x, y \in \psiOmegaomega$ with $x \neq 0 \neq y$, we have that $x < y$ together with $x_0 = y_0$ implies $f(x) < f(y)$ as a short invocation of Lemma~\ref{lem:addition} reveals.\footnote{This is already wrong if we replace the assumption $x_0 = y_0$ with $x_0 \leq y_0$ as the previous footnote showed.}	
	Additionally, we define $g: \omega^{\psiOmegaomega} \to \mathfrak{P}$ as follows:
	\begin{equation*}
		g(\langle x_0, \dots, x_{n-1} \rangle) := \psi(x_0) + \dots + \psi(x_{n-1})\period
	\end{equation*}
	Using distributivity laws, any element $x \in \psiOmegaomega$ can uniquely be understood as
	\begin{equation*}
		x = \Omega^{m_0} \cdot g(\sigma_0) + \dots + \Omega^{m_{n-1}} \cdot g(\sigma_{n-1})
	\end{equation*}
	for $n \in \n$, natural numbers $m_0 > \dots > m_{n-1}$ and elements $\sigma_i \in \omega^{\psiOmegaomega}$  with $\sigma_i \neq \langle \rangle$ for all $i < n$. Strictly speaking, this is, of course, not a valid term. We only use this in order to ease the notation of the following function $h: \psiOmegaomega \to W(\psi_1^+(W))$:
	\begin{align*}
		&h(\Omega^{m_0} \cdot g(\sigma_0) + \dots + \Omega^{m_{n-1}} \cdot g(\sigma_{n-1})) :=\\
		&\qquad\phantom{{}+{} }(1 + \psi_1^+(W))^{2m_0 + 2} \cdot (\pi^{-1}_+ \circ h)((\sigma_0)_0)\\
		&\qquad+ (1 + \psi_1^+(W))^{2m_0 + 1} \cdot (\pi^{-1}_+ \circ h \circ f)(\sigma_0)\\
		&\qquad+ \dots\\
		&\qquad+ (1 + \psi_1^+(W))^{2m_{n-1} + 2} \cdot (\pi^{-1}_+ \circ h)((\sigma_{n-1})_0)\\
		&\qquad+ (1 + \psi_1^+(W))^{2m_{n-1} + 1} \cdot (\pi^{-1}_+ \circ h \circ f)(\sigma_{n-1})\comma
	\end{align*}
	where $\pi_+: \psi_1^+(W) \to W(\psi_1^+(W))$ is the canonical isomorphism.\footnote{We choose to write $\pi_+$ instead of $\pi^+$ in order to avoid the confusing notation $\pi^{+-1}$.} This function is defined by recursion along the length of terms (i.e.~the amount of symbols). For the application of $h((\sigma_i)_0)$ for $i < n$, this already suffices. For the application of $h(f(\sigma_i))$ for $i < n$, we have to prove that the length of $f(\sigma_i)$ is strictly less than that of the original argument: For this, let $l: \psiOmegaomega \to \n$ be such a length function. Given $i < n$, let $k$ be the length of $\sigma_i$. We have
	\begin{align*}
		l(f(\sigma_i)) &= l((\sigma_i)_0 + \psi((\sigma_i)_1) + \dots + \psi((\sigma_i)_{k-1}))\\
		&\leq l((\sigma_i)_0) + 1 + l(\psi((\sigma_i)_1) + \dots + \psi((\sigma_i)_{k-1}))\\
		&< l(\psi((\sigma_i)_0)) + 1 + l(\psi((\sigma_i)_1) + \dots + \psi((\sigma_i)_{k-1}))\\
		&= l(\psi((\sigma_i)_0) + \psi((\sigma_i)_1) + \dots + \psi((\sigma_i)_{k-1}))\\
		&= l(g(\sigma_0))
	\end{align*}
	which is less than or equal to the length of our original argument. (In the case of $n = 1$ and $m_0 = 0$, it is equal. Otherwise, it is strictly less.)
	
	Using the usual induction hypothesis, we show that $f$ is an embedding: Assume elements $x, y \in \psiOmegaomega$ with $x < y$. If $y$ is an extension of $x$, then clearly $h(y)$ is an extension of $h(x)$ and we are finished. Otherwise, let $\Omega^m \cdot g(\sigma)$ and $\Omega^k \cdot g(\tau)$ be the first summands in $x$ and $y$, respectively, at the same position such that the former is strictly less than the latter. We need to show
	\begin{multline*}
		(1 + \psi_1^+(W))^{2m+2} \cdot (\pi^{-1}_+ \circ h)(\sigma_0) + (1 + \psi_1^+(W))^{2m+1} \cdot (\pi^{-1}_+ \circ h \circ f)(\sigma)\\
		< (1 + \psi_1^+(W))^{2k+2} \cdot (\pi^{-1}_+ \circ h)(\tau_0) + (1 + \psi_1^+(W))^{2k+1} \cdot (\pi^{-1}_+ \circ h \circ f)(\tau)\period
	\end{multline*}
	If $\Omega^m \cdot g(\sigma) < \Omega^k \cdot g(\tau)$ holds due to $m < k$, we are clearly done. Otherwise, assume $m = k$ and $\sigma_0 < \tau_0$. By induction hypothesis, we have $h(\sigma_0) < h(\tau_0)$. Since $\pi^{-1}_+$ is an embedding, we quickly arrive at our claim. Finally, assume $m = k$, $\sigma_0 = \tau_0$ and $\sigma < \tau$. With this, both first summands of the inequality are equal and the only thing left to show is
	\begin{equation*}
		(1 + \psi_1^+(W))^{2m+1} \cdot (\pi^{-1}_+ \circ h \circ f)(\sigma) < (1 + \psi_1^+(W))^{2m+1} \cdot (\pi^{-1}_+ \circ h \circ f)(\tau)\period
	\end{equation*}
	Recalling our thoughts on $f$, we know that $\sigma_0 = \tau_0$ and $\sigma < \tau$ entail $f(\sigma) < f(\tau)$. Using the induction hypothesis and the fact that $\pi^{-1}_+$ is an embedding, we arrive at our claim.
	
	Finally, we need to show that $h$ already maps into $\pi(\psi_1(W))$. Starting from the usual argument $x = \Omega^{m_0} \cdot g(\sigma_0) + \dots + \Omega^{m_{n-1}} \cdot g(\sigma_{n-1})$, assume that the claim already holds for all terms of smaller height. With this, we have that both $h((\sigma_i)_0)$ and $h(f(\sigma_i))$ live in $\pi(\psi_1(W))$ for $i < n$. Thus, all coefficients of $h(x)$, i.e., $(\pi^{-1}_+ \circ h)((\sigma_i)_0)$ and $(\pi^{-1}_+ \circ h \circ f)(\sigma_i)$ are elements of $\psi_1(W)$. We conclude that $h(x)$ is an element of $W(\psi_1(W))$ (or can be interpreted as such).
	
	The final step for proving that $h(x)$ is already an element of $\pi(\psi_1(W))$ is done by showing that $(\sigma_i)_0$ and $f(\sigma_i)$ are strictly smaller than $x$ for $i < n$. In the case of the former, this is relatively straightforward: Since $x$ is a term of $\psiOmegaomega$, the conditions on its coefficients immediately transport to $(\sigma_i)_0 < x$. In the case of the latter, we start from $(\sigma_i)_0 < x$ and, similar to the proof of the previous lemma, simply add $\psi$-terms to the left hand side until we have reached $f(\sigma_i) < x$. Notice that the inequality stays intact since all summands on the right hand side are of the form $\Omega^m \cdot \psi(\sigma)$ with $m > 0$. Hence, the inequality is not affected by adding summands of the form $\Omega^0 \cdot \psi(\sigma)$ to the left hand side. We conclude that $\pi \circ h$ yields the embedding claimed by our lemma.
\end{proof}

\begin{lemma}[$\rca_0$]\label{lem:phiomega_facts}
	Let $x, y, z \in \phiomega$ be terms such that $y$ appears as a strict subterm of $z$, i.e., $y$ was used at some point in defining $z$. Then, $x \leq y$ implies the inequality $x < z$.
\end{lemma}

\begin{proof}
	Straightforward induction.
\end{proof}

\begin{lemma}[$\rca_0$]\label{lem:embed_1_fp_weak_phiomega}
	There is an embedding from any $1$-fixed point of $W$ into the order $\phiomega$.
\end{lemma}

\begin{proof}
	We basically follow (a simplified version of) the construction due to Rathjen and Weiermann in \cite[Lemma~3.11]{RW93}. Given a $1$-fixed point $\psi_1(W)$ with collapse $\pi: \psi_1(W) \to W(\psi_1(W))$, we define a map $f: W(\psi_1(W)) \to \phiomega$ as follows
	\begin{align*}
		f(0) &:= 0\comma\\
		f(\sigma + (1 + \psi_1(W))^m \cdot (1 + x)) &:= \varphi(m, f(\sigma) + (f \circ \pi)(x) + 1)\period
	\end{align*}
	Since we have not defined addition on terms of $\phiomega$, we explicitly state that the sum $f(\sigma) + (f \circ \pi)(x) + 1$ is replaced by $(f \circ \pi)(x) + 1$ if $(f \circ \pi)(x)$ gets greater than $f(\sigma)$. The additional ``$+1$'' ensures that the validity of the resulting term does not depend on the value of $(f \circ \pi)(x)$.
	
	We want to show that $\sigma < \tau$ with $G_0(\sigma) <_{\fin} \sigma$ implies $f(\sigma) < f(\tau)$. The case where $\sigma = 0$ holds, is clear. Also, if $\tau$ is an extension of $\sigma$, then $f(\sigma)$ (or a larger term, in case the sum collapses) appears as a subterm in $f(\tau)$. Thus, we get $f(\sigma) < f(\tau)$ by applying Lemma \ref{lem:phiomega_facts}. For the remaining case, we prove the following two statements by simultaneous induction of the sum of the heights of two terms $\sigma, \tau \in W(\psi_1(W))$:
	\begin{enumerate}
		\item If $\sigma = \rho + (1 + \psi_1(W))^{m_0} \cdot (1 + x_0) + \dots + (1 + \psi_1(W))^{m_{n-1}} \cdot (1 + x_{n-1})$ and $\tau = \rho + (1 + \psi_1(W))^{k} \cdot (1 + y)$ with $\sigma < \tau$ and $G_0(\sigma) <_{\fin} \tau$, then $f(\sigma) < f(\tau)$ holds.
		\item If $\sigma < \tau$ and $G_0(\sigma) <_{\fin} \sigma$, then $f(\sigma) < f(\tau)$ holds.
	\end{enumerate}
	For (1), let $\sigma$ and $\tau$ be as in the assumption. We proceed by induction on $n$: If $n = 0$, then clearly $\tau$ is an extension of $\sigma$ and, hence, $f(\sigma) < f(\tau)$ holds using the same argument as before.
	
	Otherwise, we have both $n > 0$ and $k \geq m_{n-1}$. Assume $k > m_{n-1}$. For
	\begin{align*}
		f(\tau) &= \varphi(k, f(\rho) + (f \circ \pi)(y) + 1)\\
		&> \varphi(m_{n-1}, f(\rho + (1 + \psi_1(W))^{m_0} \cdot (1 + x_0) + \dots\\
		&\qquad+ (1 + \psi_1(W))^{m_{n-2}} \cdot (1 + x_{n-2})) + (f \circ \pi)(x_{n-1}) + 1)\\
		&= f(\sigma)
	\end{align*}
	we need to show both
	\begin{equation*}
		f(\tau) > f(\rho + (1 + \psi_1(W))^{m_0} \cdot (1 + x_0) + \dots + (1 + \psi_1(W))^{m_{n-2}} \cdot (1 + x_{n-2}))
	\end{equation*}
	and $f(\tau) > (f \circ \pi)(x_{n-1})$. The former immediately follows from applying (1) inductively. For the latter, we use (2) together with $\pi(x_{n-1}) \in G_0(\sigma) <_{\fin} \tau$ and $G_0(\pi(x_{n-1})) <_{\fin} \pi(x_{n-1})$, where we recall that $G_0(\pi(x)) <_{\fin} \pi(x)$ holds for any element $x \in \psi_1(W)$.
	
	Assume $k = m_{n-1}$. Since the exponents of any element in $W(\psi_1(W))$ are strictly descending, we conclude $n = 1$. Thus, $k = m_{n-1} = m_0$ and $x_0 < y$. For ${f(\tau) = \varphi(k, f(\rho) + (f \circ \pi)(y) + 1) > \varphi(m_0, f(\rho) + (f \circ \pi)(x_0) + 1) = f(\sigma)}$, we need to show $(f \circ \pi)(x_0) < (f \circ \pi)(y)$. For this, we apply (2) to $\pi(x_0) < \pi(y)$, which follows from $x_0 < y$ since $\pi$ is an embedding. Moreover, since $G_0(\pi(x)) <_{\fin} \pi(x)$ holds for any $x \in \psi_1(W)$, the assumption $G_0(\pi(x_0)) <_{\fin} \pi(x_0)$, that is required by (2), is also satisfied.
	
	For (2), the cases where $\sigma = 0$ holds or where $\tau$ is an extension of $\sigma$ have already been done.
	Suppose $\sigma = \rho + (1 + \psi_1(W))^{m_0} \cdot (1 + x_0) + \dots + (1 + \psi_1(W))^{m_{n-1}} \cdot (1 + x_{n-1})$ and $\tau = \rho + (1 + \psi_1(W))^{k_0} \cdot (1 + y_0) + \dots + (1 + \psi_1(W))^{k_{l-1}} \cdot (1 + y_{k-1})$ with $(1 + \psi_1(W))^{m_0} \cdot (1 + x_0) < (1 + \psi_1(W))^{k_0} \cdot (1 + y_0)$. Let $\tau' := \rho + (1 + \psi_1(W))^{k_0} \cdot (1 + y_0)$. We have $\sigma < \tau' \leq \tau$. Since $\tau$ is equal to or an extension of $\tau'$, we have $f(\tau') \leq f(\tau)$ using the same argument as before. In order to derive $f(\sigma) < f(\tau')$, we can use (1) since $G_0(\sigma) <_{\fin} \sigma < \tau'$ holds.
	
	Finally, since any element $\sigma \in \pi(\psi_1(W))$ satisfies $G_0(\sigma) <_{\fin} \sigma$, we conclude that $f \circ \pi$ is the embedding claimed by our lemma.
\end{proof}

\begin{proof}[Proof of Proposition \ref{prop:weak_goodstein_phiomega}]
	We apply Theorem \ref{thm:d_sequence_fp_is_1_fp}, which provides the one-to-one correspondence between $1$-fixed points and strong $D$-sequence termination points. For the embedding from $\phiomega$ to the termination point, we invoke Lemmas \ref{lem:embed_phiomega_psiOmegaomega} and~\ref{lem:embed_psiOmegaomega_1_fp_weak}. For the other direction, we make use of Lemma \ref{lem:embed_1_fp_weak_phiomega}. Also, for the second part of the theorem, recall that any well founded $D$-sequence termination point is already strong (Lemma~\ref{lem:wf_fp_strong}).
\end{proof}

\section{Different ways to loose uniqueness}\label{sec:uniqueness}
In this section, we explore some thoughts on the uniqueness of our $D$-sequence termination points. But before that, we consider the case of $1$-fixed points and what happens if we replace the height function by well foundedness. This connects to the short discussion below Definition~1.4 in \cite{FR21}.
\begin{lemma}[$\rca_0$]\label{lem:1_fp_old_def_wkl}
	Consider a possible variation of the definition of $1$-fixed points where $\triangleleft$ is only required to be well founded (but where we still use the simplified version of $G_0$ from Lemma \ref{lem:1_fp_simplyfied_G}). Then, $\rca_0$ proves that $1$-fixed points are unique if and only if $\wkl$ holds.
\end{lemma}

\begin{proof}
	Let us denote the concatenation of two sequences $s$ and $t$ by $s * t$.
	For the easy direction, we assume $\wkl_0$. Consider some predilator $D$ together with a $1$-fixed point on $X$. We want to define a height function $h: X \to \n$ such that $h(x)$ is the length of the smallest sequence $s \in X^*$ with $s_0 = x$ and $s_{i+1} \triangleleft s_i$ for all $i < h(x) - 1$, for any $x \in X$. Since for any fixed length, we can give all such sequences effectively in the form of a finite list (compare the proof of Lemma~\ref{lem:wf_fp_strong}), we only need to make sure that the search of $h(x)$ terminates, i.e., that there is such a sequence $s$. For this, consider some arbitrary $x \in X$ and the smallest tree $T$ with $\langle x \rangle \in T$ such that $s * \langle y \rangle \in T$ implies $s * \langle y, z \rangle \in T$ for all $z \triangleleft y$. Clearly, this tree is finitely branching. Moreover, it can be effectively bounded (cf.~\cite[Definition~IV.1.3]{Simpson09}). Thus, an application of the bounded K\H{o}nig's Lemma (cf.~\cite[Lemma~IV.1.4]{Simpson09}) yields an infinite path in $T$, which cannot exist because $\triangleleft$ is well founded. Finally, $h(x) < h(y)$ holds for any two $x, y \in X$ with $x \triangleleft y$ by design of $h$ (again, see Lemma~\ref{lem:wf_fp_strong} for more details on a similar construction). Now, that we have a height function, this $1$-fixed point satisfies all requirements of the original definition, which entails uniqueness (cf.~\cite[Corollary~2.2]{FR21}).
	
	For the other direction, assume $\rca_0$ and $\lnot \wkl$. This means that we have access to some infinite binary tree $T$ without infinite path. In order to keep everything simple, we assume that any node has either $0$ or $2$ children. We define a linear order on $T$ such that $s < t$ holds for any two $s, t \in T$ if and only if $t$ is a proper extension of $s$ or both $s$ and $t$ are incomparable but $s$ is strictly lexicographically smaller than $t$. (Unfortunately, this is ill founded, so the following predilator cannot be a dilator:) Let $L$ be the set of leaves and $N := T \setminus L$. We consider a functor $D$ that maps any linear order $X$ to the set of terms:
	\begin{itemize}
		\item $l \in L$,
		\item $t(x, y)$ for $t \in N$ and $x, y \in X$.
	\end{itemize}
	The order is given as follows: For two elements of $L$, we just copy the order from~$L$. Similarly, if one of the elements is a leaf $l$ and the other is of the form $t(x, y)$, we simply compare $l$ and $t$. If we have two elements $t(x, y)$ and $t'(x', y')$, we use a lexicographical ordering by first comparing $t$ and $t'$, then $x$ and $x'$, and finally $y$ and $y'$. Any morphism $D(f)$ for $f: X \to Y$ maps leaves to themselves and $t(x, y)$ to $t(f(x), f(y))$. Clearly, $D$ is a predilator.
	
	Now, we construct a $1$-fixed point that differs from the canonical one. We define a term set $X^+$ and a map $\pi^+: X^+ \to D(X^+)$ simultaneously:
	\begin{itemize}
		\item $l \in L$ with $\pi^+(l) := l$,
		\item $t \in N$ with $\pi^+(t) := t(t * \langle 0 \rangle, t * \langle 1 \rangle)$,
		\item $P(t, x, y)$ for $t \in T$, $x, y \in X^+$ with $\pi^+(P(t, x, y)) := t(x, y)$ if $x \neq t * \langle 0 \rangle$ or $y \neq t * \langle 1 \rangle$.
	\end{itemize}
	Notice that the existence of and the restrictions on terms $P(t, x, y)$ ensure that $\pi^+$ is bijective.
	Next, we provide an order on this set: For the elements from $T$, we simply copy the order. For all other comparisons, we recursively reflect the order from $D(X^+)$ via $\pi^+$. Now, we restrict this term set to a suborder $X \subset X^+$: This new order $X$ contains every element from $T$ but only those terms $P(t, x, y)$ with $x < P(t, x, y)$, $y < P(t, x, y)$, and $x, y \in X$. We define $\pi: X \to D(X)$ to be the unique morphism with $D(\iota_X^{X^+}) \circ \pi = \pi^+ \circ \iota_X^{X^+}$.
	
	We check all the requirements for $1$-fixed points: First, it is clear that $\pi$ is an embedding. Second, $\triangleleft$ is well founded: Assume that we have an infinite descending sequence in $X$ with respect to this relation. Eventually, all elements in this sequence will stem from $T$. However, an infinite descending sequence in $T$ with respect to $\triangleleft$ corresponds to an infinite path in $T$, which does not exist by assumption.
	Third, we have to satisfy the requirement on the range of $\pi$: 
	
	``$\subseteq$'': For $t \in N$, we have $t * \langle 0 \rangle \triangleleft t$ and $t * \langle 1 \rangle \triangleleft t$ because of the order on $T$. For $P(t, x, y)$, we have $x \triangleleft \pi(P(t, x, y))$ and $y \triangleleft \pi(P(t, x, y))$ simply by definition of $X$.
	
	``$\supseteq$'' Consider some term $\sigma \in D(X)$ such that $\pi(s) < \sigma$ holds for any ${s \in \supp_X(\sigma)}$. If $\sigma$ is an element from $l \in L$ (or of the form $t(t * \langle 0 \rangle, t * \langle 1 \rangle)$ for $t \in T$), then we simply have $\pi(l) = \sigma$ (or $\pi(t) = \sigma$). Otherwise, $\sigma$ is of the form $t(x, y)$ for $t \in T$ and $x, y \in X$ with $\pi(x) < t(x, y)$, $\pi(y) < t(x, y)$ and one of $x \neq t * \langle 0 \rangle$ or $y \neq t * \langle 1 \rangle$ holds. Thus, it can be mapped to from $P(t, x, y)$ using $\pi$.
	
	Finally, assume that there is some $1$-fixed point isomorphism $f: X \to \psi_1(D)$ to the canonical $1$-fixed point $\psi_1(D)$. In order to distinguish the order $\triangleleft$ for both fixed points, let us indicate their origin using a subscript.
	By construction of $\psi_1(D)$ (cf.~\cite[Section~2]{FR21}), there is a finite amount of elements in $\psi_1(D)$ strictly below $f(\langle \rangle)$ with respect to the transitive closure of $\triangleleft_{\psi_1(D)}$. However, from the definition of $\pi$, we can see that any node from $T$ (except for the root node itself) lies strictly below the root $\langle \rangle$ with respect to the transitive closure of $\triangleleft_X$. Since $T$ is infinite, $f$ cannot have been an isomorphism.
\end{proof}
Omitting all requirements on $\triangleleft$ altogether completely removes uniqueness. We switch over to $D$-sequence termination points and prove the short result in this context.
\begin{lemma}[$\rca_0$]
	Assume that strong $D$-sequence termination points (Definition \ref{def:soa_d_sequence_fixed_point}) are defined without requirement (3), i.e., without a height function. Then, we can find two distinct strong $\id$-sequence termination points.
\end{lemma}
\begin{proof}
	Since we have $\id(\emptyset) = \emptyset$, Example \ref{ex:various_d_sequence_fixed_points} suggests that the linear order $\emptyset$ is an $\id$-sequence termination point. For $\emptyset$, all requirements (even the third) collapse and become trivial.
	
	However, there is another termination point: $\z$. We define $A_{z}^{D} := z - 1$ for all $z \in \z$ and check the requirements:
	\begin{enumerate}
		\item Consider $z \in \z$. We want to show that $z - 1$ is the smallest value in $\restr{\z}{z}$ greater than $y - 1$ for all $y < z$. It is immediately clear that $z - 1$ is the smallest value in $\z$ with this property. Moreover, $z - 1 < z$ holds, which implies $z - 1 \in \restr{\z}{z}$.
		\item[(2/4)] For any element $\sigma \in \z$, we can find a smallest $z \in \z$ with $\sigma \leq z - 1$ by simply choosing $z := \sigma + 1$.\qedhere
	\end{enumerate}
\end{proof}
According to the claimed uniqueness of strong termination points, the third point has to fail: We have $x \prec y \Leftrightarrow x \in \supp_{\restr{\z}{y}}(A_{y}^{D}) = \{y - 1\}$. Thus, $x \prec y$ is equivalent to $x = y - 1$. We conclude that the decreasing enumeration of all negative values in $\z$ induces an infinite descending sequence in the relation $\prec$.

We present the last big result of this paper: A characterization of $\pica$ using the uniqueness of regular $D$-sequence termination points.
\begin{theorem}[$\rca_0$]\label{thm:pica_uniqueness_regular_fixed_point}
	The following are equivalent:
	\begin{enumerate}[label=(\alph*)]
		\item $\Pi^1_1$-comprehension,
		\item For any dilator $D$ any two regular $D$-sequence termination points are isomorphic,
		\item For any dilator $D$ given any two regular $D$-sequence termination points on $X$ and $Y$, we have that $X$ and $Y$ are isomorphic.
	\end{enumerate}
\end{theorem}
Note, the difference between (b) and (c) lies in the fact that an isomorphism between $D$-sequence termination points not only requires the underlying linear orders to be isomorphic but also the $D$-sequences that are defined on them (cf.~Definition~\ref{def:d_seq_homomorphism}).
The direction from (b) to (c) is trivial. The direction from (a) to (b) is covered by Theorem~\ref{thm:pica_wf_fp} together with the following proposition:
\begin{proposition}[$\rca_0$]\label{prop:wf_fp_implies_all_regular_fp_iso}
	Given a predilator $D$ with a well founded $D$-sequence termination point, any two regular $D$-sequence termination points are isomorphic.
\end{proposition}
\begin{proof}
	Let $(\psi_1(D), \pi)$ be the canonical $1$-fixed point of $D$. Let $(X, (A^D_x)_{x \in X})$ be any regular $D$-sequence termination point with height function $h: X \to \n$. Let $\pi_+: \psi^+_1(D) \to D(\psi^+_1(D))$ be the extension of $\pi$ such that $\pi_+$ is an isomorphism. We define a map $f: X \to \psi_1^+(D)$ using recursion along $h$:
	\begin{equation*}
		f(x) := \pi^{-1}_+(D(f_S)(\sigma))\comma
	\end{equation*}
	where $S := \supp_{\restr{X}{x}}(A^D_x)$ holds, $f_S: S \to \psi_1^+(D)$ is the restriction of $f$ on $S$, and where $\sigma \in D(S)$ is the unique element with $D(\iota_S^{\restr{X}{x}})(\sigma) = A^D_x$.
	
	Given two elements $x, y \in X$, we prove along $h(x) + h(y)$ that $x < y$ implies $f(x) < f(y)$.	
	Let $S := \supp_{\restr{X}{x}}(A^D_x)$ and $T := \supp_{\restr{X}{y}}(A^D_y)$. Also, we need $\sigma \in D(S)$ and $\tau \in D(T)$ with $D(\iota_S^{\restr{X}{x}})(\sigma) = A^D_x$ and $D(\iota_T^{\restr{X}{y}})(\tau)$. From $x < y$, we have $D(\iota^X_{\restr{X}{x}})(A^D_x) < D(\iota^X_{\restr{X}{y}})(A^D_y)$, which implies $D(\iota_S^{S \cup T})(\sigma) < D(\iota_T^{S \cup T})(\tau)$.	
	By induction hypothesis, we know that $f_{S \cup T}$ is already defined and that it is an embedding. We conclude
	\begin{equation*}
		D(f_S)(\sigma) = D(f_{S \cup T} \circ \iota_S^{S \cup T})(\sigma) < D(f_{S \cup T} \circ \iota_T^{S \cup T})(\tau) = D(f_T)(\tau)\comma
	\end{equation*}
	where the inequality holds since $D(f_{S \cup T})$ is an embedding. Moreover, $\pi^{-1}_+$ is also an embedding and, hence, we arrive at $f(x) < f(y)$.
	
	Next, we show that the range of $f$ is a suborder of $\psi_1(D)$. Again, this is done using induction along the height of an element $x \in X$: If we assume the claim for all elements of smaller height, we immediately have $\supp_{\pi_+(f(x))} \subseteq S \subseteq \psi_1(D)$, where $S$ is the support of $A^D_x$ and the second suborder relation follows from induction. Moreover, every element in $S$ is smaller than $x$. Using the fact that $f$ is an embedding, we conclude $G_0(f(x)) <_{\fin} f(x)$ and, thus, $f(x) \in \psi_1(D)$.
	
	Recall that $D$ has a well founded $D$-sequence termination point (that is strong by Lemma \ref{lem:wf_fp_strong}). Furthermore, there is a strong $D$-sequence termination point induced by $\psi_1(D)$ using Theorem \ref{thm:d_sequence_fp_is_1_fp}. Since strong $D$-sequence termination points are unique by Corollary \ref{cor:strong_fp_unique}, we conclude that $\psi_1(D)$ is well founded. Now, this transports to $X$ via $f$. Finally, this implies that $(X, (A^D_x)_{x \in X})$ is also strong. Now that any two regular $D$-sequence termination points are already strong, they are also isomorphic via Corollary \ref{cor:strong_fp_unique}.
\end{proof}
Finally, for the direction (c) to (a) of Theorem~\ref{thm:pica_uniqueness_regular_fixed_point}, we use Theorem~\ref{thm:pica_wf_fp} together with the following proposition:
\begin{proposition}[$\rca_0$]\label{prop:if_D_sequence_non_iso_fp_linear_orders}
	Given a (pre)dilator $D$ with an ill founded regular $D$-sequence termination point, we can construct a (pre)dilator $E$ with two regular $E$-sequence termination points on two linear orders such that already these linear orders are not isomorphic.
\end{proposition}
\begin{proof}
	Given the predilator $D$, consider the following predilator $E$:
	\begin{equation*}
		E(X) := (D(X) + 1 + D(X)) \times \omega
	\end{equation*}
	for linear orders $X$. To be more precise, $E$ maps any such $X$ to the linear order consisting of tuples $\langle 0, \sigma, n \rangle$, $\langle 1, n \rangle$, and $\langle 2, \sigma, n \rangle$ for $\sigma \in D(X)$ and $n \in \n$, ordered lexicographically. Given $f: X \to Y$, the resulting morphism $E(f)$ maps these tuples to $\langle 0, f(\sigma), n \rangle$, $\langle 1, n \rangle$, and $\langle 2, f(\sigma), n \rangle$, respectively. Clearly, if $D$ is a dilator, then so is $E$.
	
	Consider the canonical $E$-sequence termination point given by $\psi_1(E)$ together with a map ${\pi: \psi_1(E) \to E(\psi_1(E))}$ using Theorem \ref{thm:d_sequence_fp_is_1_fp}. In $\psi_1(E)$, any element has a successor. This can either be seen directly from the definition (the multiplication ``$\times \omega$'' ensures this property) or more formally as follows:
	Consider, w.l.o.g., ${\sigma := \langle 0, \sigma', n \rangle \in \pi(\psi_1(E))}$. Clearly, $\tau := \langle 0, \sigma', n + 1 \rangle$ is greater than $\sigma$ and, since it shares its support with $\sigma$, we have $\tau \in \pi(\psi_1(E))$. Assume that we find some $\rho \in \pi(\psi_1(E))$ with $\sigma < \rho < \tau$. The latter inequality implies that $\rho$ is of the form $\langle 0, \rho', m \rangle$. From $\sigma < \rho$, we conclude $\sigma' \leq \rho'$ and, similarly from $\rho < \tau$, we have $\rho' \leq \sigma'$. Thus, $\rho' = \sigma'$ Finally, this implies $n < m < n + 1$, which cannot be. The arguments for $\sigma$ of any of the other possible two forms work similar.
	
	For the rest of the proof, we want to construct another $E$-sequence termination point that contains an element without successor. Since this property is preserved by isomorphisms, the underlying order of this termination point cannot be isomorphic to $\psi_1(E)$.
	
	By Proposition \ref{prop:wf_fp_implies_all_regular_fp_iso}, we know that the existence of an ill founded regular $D$-sequence termination point also makes the canonical $D$-sequence termination point $\psi_1(D)$ with ${\kappa: \psi_1(D) \to D(\psi_1(D))}$ ill founded. Thus, we have a descending sequence $(\sigma_n)_{n \in \n} \subseteq \kappa(\psi_1(D))$. We want to transport this descending sequence to the \emph{second} copy of $D$ in $E$. For this, we begin by considering the following general embedding $f$ from $\kappa(\psi_1(D))$ into $\pi(\psi_1(E))$:
	\begin{equation*}
		f(\sigma) := E(f)(\langle 0, \sigma, 0 \rangle)\period
	\end{equation*}
	Of course, to be precise, we have to consider $f$ with codomain $E(\psi^+_1(E))$. Also, we define $f(\sigma)$ recursively using $E(f_S)(\sigma')$ for $S := \supp_{\psi_1(D)}(\sigma)$ and the unique $\sigma' \in E(S)$ with $E(\iota_S^{\psi_1(E)})(\sigma') = \langle 0, \sigma, 0 \rangle$, where $f_S: S \to \psi^+_1(E)$ is defined like $f$ on $S$. After this, we prove that $f$ is an embedding. Finally, we show that $f$ already maps into $\pi(\psi_1(E))$. We only sketch these steps as they are very similar to those in the proof of Proposition \ref{prop:wf_fp_implies_all_regular_fp_iso}.
	
	The descending sequence in the second copy of $D$ in the predilator $E$ is then given by ${(\tau_n)_{n \in \n} \subseteq \psi_1(E)}$ with $\pi(\tau_n) := \langle 2, \sigma_n, 0 \rangle$ for any $n \in \n$. Clearly, each such element lives in $\psi_1(E)$ since the support elements are all in the first copy of $D$ in $E$ and, hence, smaller than the sequence element itself.
	
	Now, we define the underlying order of our second $E$-sequence termination point. The order $X$ is given by the following sequence of suborders $(X_n)_{n \in \n}$ of $\psi_1(E)$:
	\begin{itemize}
		\item $X_0$ contains all elements $\sigma \in \psi_1(E)$ of the form $\pi(\sigma) = \langle 0, \sigma', m\rangle$ or ${\pi(\sigma) = \langle 1, 0 \rangle}$. It is our goal to construct $X$ in such a way that this latter element does not have a successor.
		\item $X_{n}$ for $n > 0$ contains all elements $\sigma \in \psi_1(E)$ of the form $\pi(\sigma) = \langle 2, \sigma', m \rangle$ such that
		\begin{itemize}
			\item $\sigma \notin X_{n'}$ for all $n' < n$,
			\item $\sigma > \tau_n$,
			\item $\supp_{\psi_1(E)}(\pi(\sigma))$ is a finite suborder of $\bigcup_{n' < n} X_n'$.
		\end{itemize}
	\end{itemize}
	With this, we can define $X := \{(n, \sigma) \mid \sigma \in X_n \text{ and } n \in \n\}$. Let $g: X \to \psi_1(E)$ be the injection given by $g((n, \sigma)) := \sigma$ for all $(n, \sigma) \in X$. The order on $X$ is given by the order on $\psi_1(E)$, i.e., $x < y$ holds if and only if $g(x) < g(y)$ does for any $x, y \in X$. By definition, the orders $X_n$ are all pairwise disjoint, which entails that $g$ is injective and that the order defined on $X$ is anti-symmetric (and thus linear).
	Notice that for any $n \in \n$, there is some $m \in \n$ with $(m, \tau_n) \in X$. This is because $\tau_n > \tau_k$ holds for all $k > n$ and all elements in the support of $\tau_n$ are included in~$X_0$.
	
	Let $(A^E_\sigma)_{\sigma \in \psi_1(E)}$ be the canonical $E$-sequence on $\psi_1(E)$. We define the $E$-sequence $(B^E_x)_{x \in X}$ on $X$ as follows:
	\begin{equation*}
		E(\restr{g}{x})(B^E_x) := A^E_{g(x)}
	\end{equation*}
	for all $x \in X$. The existence of such a $B^E_x$ can be argued as follows: Let $x \in X$ be arbitrary. Informally, any element in $X$ is only constructed using, again, elements from $X$. By definition of $X$, we therefore have that $\supp_{\psi_1(E)}(\pi(g(x))) \subseteq \rng(g)$ holds. Since $A^E_{g(x)}$ has the same support as $\pi(g(x))$, we can apply the support condition of our $1$-fixed point $\psi_1(E)$ and find some $B^E_x$ satisfying our claim.
	Since $E(\restr{g}{x})$ is an embedding, $B^E_x$ must be unique.

	Let us now verify that this sequence $(B^E_x)_{x \in X}$ satisfies all requirements of a regular $E$-sequence.
	\begin{enumerate}[leftmargin=*, label=(\arabic*)]
		\item Let $x \in X$ be arbitrary. The fact that $B^E_x$ is greater than any $E(\iota_{\restr{X}{y}}^{\restr{X}{x}})(B^E_y)$ for $y < x$ reflects from the similar property of $(A^E_\sigma)_{\sigma \in \psi_1(E)}$. Now, we only have to show that each $B^E_x$ is the lowest bound in $E(\restr{X}{x})$ with this property. Assume that for some $x \in X$, there is such a counterexample, i.e., we have $\sigma \in E(\restr{X}{x})$ with $\sigma < B^E_x$ and $E(\iota_{\restr{X}{y}}^{\restr{X}{x}})(B^E_y) < \sigma$ for all $y < x$. From $E(\restr{g}{x})(\sigma) < E(\restr{g}{x})(B^E_x) = A^E_{g(x)}$, we conclude that there is some element $\gamma \in \psi_1(E)$ with $\gamma < g(x)$ and $E(\iota_{\restr{\psi_1(E)}{\gamma}}^{\restr{\psi_1(E)}{g(x)}})(A^E_\gamma) \geq E(\restr{g}{x})(\sigma)$. If $B^E_x$ was of the form $\langle 0, \tau, m\rangle$ or $\langle 1, 0 \rangle$, then we know by definition of $X$ that there must be some $y \in X$ with $g(y) = \gamma$. Moreover, we have $y < x$ since $g$ is an embedding and
		\begin{align*}
			(E(\restr{g}{x}) \circ E(\iota_{\restr{X}{y}}^{\restr{X}{x}}))(B^E_y) &= (E(\iota_{\restr{\psi_1(E)}{g(y)}}^{\restr{\psi_1(E)}{g(x)}}) \circ E(\restr{g}{y}))(B^E_y)\\
			&= E(\iota_{\restr{\psi_1(E)}{g(y)}}^{\restr{\psi_1(E)}{g(x)}})(A^E_{g(y)}) \geq E(\restr{g}{x})(\sigma)\period
		\end{align*}
		Since $E(\restr{g}{x})$ is an embedding this contradicts the assumption on $\sigma$.
		
		The only case left is that $B^E_x$ is of the form $\langle 2, \tau, m\rangle$. Here, there might not be a $y \in X$ with $g(y) = \gamma$.
		Assume that this is the case since some element in the internal support set of $\gamma$ (i.e., the support of $A^E_\gamma$) is not in the range of $g$. W.l.o.g., we can assume that $\gamma$ has been chosen like in Remark \ref{rem:cond_1_cases}. This remark tells us that either, $\gamma$ satisfies the equality $E(\iota_{\restr{X}{\gamma}}^{\restr{X}{g(x)}})(A^E_\gamma) = E(\restr{g}{x})(\sigma)$, which leads to the contradiction $\supp_{\restr{\psi_1(E)}{\gamma}}(A^E_\gamma) = g(\supp_{\restr{X}{x}}(\sigma)) \subseteq \rng(g)$.
		Or, that $\gamma$ lies in the support of $E(\restr{g}{x})(\sigma)$. Thus, $\gamma$ is in the range of $g$. From the contradiction in both cases, we conclude that the support of $A^E_\gamma$ lies in the range of $g$. The definition of $X$, therefore, entails that $\gamma \leq \tau_n$ holds for all $n \in \n$. We know that there is some $n \in \n$ with $\tau_n < g(x)$. As previously discussed, $\tau_n$ is in the range of $g$, i.e., there is some $z \in X$ with $g(z) = \tau_n$. Also, $\tau_n < g(x)$ entails $z < x$. We put everything together:
		\begin{align*}
			(E(\restr{g}{x}) \circ E(\iota_{\restr{X}{z}}^{\restr{X}{x}}))(B^E_z) &= E(\iota_{\restr{\psi_1(E)}{\tau_n}}^{\restr{\psi_1(E)}{g(x)}})(A^E_{\tau_n})\\
			&\geq E(\iota_{\restr{\psi_1(E)}{\tau_n}}^{\restr{\psi_1(E)}{g(x)}} \circ \iota_{\restr{\psi_1(E)}{\gamma}}^{\restr{\psi_1(E)}{\tau_n}})(A^E_{\gamma})\\
			&= E(\iota_{\restr{\psi_1(E)}{\gamma}}^{\restr{\psi_1(E)}{g(x)}})(A^E_{\gamma}) \geq E(\restr{g}{x})(\sigma)\period
		\end{align*}
		We conclude that $B^E_z$ contradicts our assumption on $\sigma$.
		\item Let $\sigma \in E(X)$ be arbitrary. We need to find $x \in X$ with $E(\iota_{\restr{X}{x}}^X)(B^E_x) \geq \sigma$. W.l.o.g., we can assume that $E(g)(\sigma)$ is greater than $\pi(\tau_0)$. Then, the only reason for there not being an $x \in X$ with $E(\iota_{\restr{X}{x}}^X)(B^E_x) = \sigma$ can be that $E(g)(\sigma)$ is not in the range of $\pi$. This entails the existence of $\tau \in \psi_1(E)$ with $\tau \in \supp_{\psi_1(E)}(E(g)(\sigma))$ and $\pi(\tau) \geq E(g)(\sigma)$. Since $\tau$ is in the support of $E(g)(\sigma)$, we conclude that there is some $x \in X$ with $g(x) = \tau$. Finally, this implies
		\begin{align*}
			E(g \circ \iota_{\restr{X}{x}}^X)(B^E_x) &= E(\iota_{\restr{\psi_1(E)}{g(x)}}^{\psi_1(E)} \circ (\restr{g}{x}))(B^E_x)
			\\&= E(\iota_{\restr{\psi_1(E)}{g(x)}}^{\psi_1(E)})(A^E_{g(x)})\\
			&= \pi(g(x)) = \pi(\tau) \geq E(g)(\sigma)\period
		\end{align*}
		Since $E(g)$ is an embedding, we have $E(\iota_{\restr{X}{x}}^X)(B^E_x) \geq \sigma$.
		\item For the required height function, we can simply copy the one from $\psi_1(E)$.
	\end{enumerate}
	In order to conclude this proof, we only need to show that, in contrast to $\psi_1(E)$, we can find some element in $X$ without successor. This element is given by $x \in X$ with $\pi(g(x)) = \langle 1, 0 \rangle$. Assume that $y$ is a successor of $x$ in $X$. By definition of $X$, we know that $g(y) > \tau_n$ holds for some $n \in \n$. Moreover, we have $\tau_n \in \rng(g)$ and $\tau_n > g(x)$. Thus, $y$ cannot be the successor of $x$. Finally, we conclude that with $\psi_1(E)$ and $X$, we have found two linear orders on which you can define $E$-sequences such that $\psi_1(E)$ and $X$ are not isomorphic.
\end{proof}

It is not clear whether this can be improved to $E := D$ for dilators $D$. For predilators $D$, we can find an easy counterexample:

\begin{remark}\label{rem:z_fp_all_iso}
	Let $D$ be the dilator that constantly maps to the ill founded order~$\z$. (Note that morphisms $D(f)$ for $f: X \to Y$ are always the identity function since all supports are empty for this predilator). Clearly, $\z$ itself is a $1$-fixed point of $D$ and, thus, a strong $D$-sequence termination point using Theorem \ref{thm:d_sequence_fp_is_1_fp}. In the following, we will show that for this predilator any two $D$-sequence termination points are isomorphic, i.e., they do not even have to be regular and the isomorphism is not only restricted to the underlying linear orders.
	
	Assume a second $D$-sequence termination point given by some linear order $X$ together with $(A^D_x)_{x \in X}$. This sequence by itself is an embedding from $X$ to $\z$. We now show that it must also be surjective: Assume that there is some $z \in \z$ such that $A^D_x \neq z$ for any $x \in X$. By requirement (2) of $D$-sequence termination points, there is, however, some element $x \in X$ with $A^D_x > z$. Assume that $x$ is chosen such that $A^D_x$ has the smallest value with this property. (This is possible since $\z$ restricted to values above $z$ is well founded and this property is reflected by $x \mapsto A^D_x$.) We conclude that there is no $y \in X$ with $A^D_y = A^D_x - 1$ since such a value would either be equal to $z$ or a smaller candidate than $x$. We have $A^D_y < A^D_x$ for any $y < x$. Assume $A^D_y \geq A^D_x - 1$ for some $y < x$. Either, $A^D_y = A^D_x - 1$ holds, which cannot be because of our considerations from before, or $A^D_y > A^D_x - 1$ holds, which entails $A^D_y \geq A^D_x$ and is, thus, a contradiction to $A^D_y < A^D_x$ for all $y < x$. We conclude that $A^D_x - 1$ is a smaller value than $A^D_x$ with the property that $A^D_y < A^D_x - 1$ holds for any $y < x$. This contradicts requirement (1) of $D$-sequence termination points. Finally, this means that our embedding $x \mapsto A^D_x$ from $X$ to $\z$ is an isomorphism. This is even a $D$-sequence isomorphism since we have
	\begin{equation*}
		D(\restr{f}{(\restr{X}{x})})(A^D_x) = A^D_x = f(x) = B^D_{f(x)}
	\end{equation*}
	for any $x \in X$, where $f: X \to Y$ with $f(x) := A^D_x$ is our isomorphism and $(B^D_z)_{z \in \z}$ with $B^D_z := z$ is the sequence associated with our trivial $D$-sequence termination point on $\z$.
\end{remark}
We conclude this paper with a remark on how strongly the amount of linear orders and actual termination points can differ for regular $D$-sequences. In order to avoid any considerations on how to deal with uncountability in second order arithmetic, we argue in set theory.
\begin{remark}
	The predilator $D$ that constantly maps to $\q$ has uncountably many pairwise non-isomorphic regular $D$-sequence termination points, but the underlying linear orders of any two $D$-sequence termination points are isomorphic. (Similar to Remark~\ref{rem:z_fp_all_iso}, morphisms $D(f)$ for $f: X \to Y$ are always the identity.)
	
	For the first claim, let $r \in \rr$ be some arbitrary real number. We define
	\begin{equation*}
		X := \q \setminus [r, r + 1]\period
	\end{equation*}
	The sequence that makes this a termination point is given by the identity, i.e., we have $(A^D_q)_{q \in X}$ with $A^D_q := q$. Let us show that this results in a regular $D$-sequence termination point. Let $q \in X$ be arbitrary. Since $X$ is an open subset of $\q$, the element $q$ is the supremum of all $p \in X$ with $p < q$. Thus, requirement $(1)$ is satisfied. For $(2)$, let $q \in \q$ be arbitrary. Clearly, we can find a rational $p \in \q$ that is greater than both $q$ and $r + 1$. We conclude $q < p \in X$. Finally, requirement $(3)$ is clear since all supports are empty.
	
	Now, consider two distinct reals $r_1, r_2 \in \rr$. Let $X$ with $(A^D_x)_{x \in X}$ be constructed as above for $r_1$, and let $Y$ with $(B^D_y)_{y \in Y}$ be constructed as above for $r_2$. We can find some rational $q \in X \setminus Y$. Then, a $D$-sequence isomorphism $f: X \to Y$ from the former termination point to the latter one would imply $D(\restr{f}{(\restr{X}{q})})(A^D_q) = A^D_q = q \neq f(q) = B^D_{f(q)}$, where the inequality simply holds because $q$ is not an element of $Y$ and, thus, cannot be mapped to by $f$. We conclude that there are uncountably many pairwise non-isomorphic regular $D$-sequence termination points.
	
	For the second claim, i.e., that the underlying linear orders of any two $D$-sequence termination points are isomorphic, we simply show for any possible termination point on $X$ with sequence $(A^D_x)_{x \in X}$, that $X$ is non-empty, dense, and has no end points. The usual back-and-forth argument then yields an isomorphism with $\q$.
	
	For simplicity, let us identify $X$ with its range with respect to $x \mapsto A^D_x$. This means, we assume $X \subseteq \q$ and $A^D_q = q$ for any $q \in X$. By requirement $(2)$, we immediately have that $X$ cannot be empty. Now, assume that $X$ contains two rationals $p < q$. From requirement $(1)$, we know that $q$ must be the limit of some strictly increasing sequence in $X$. Thus, we can find a third rational $r \in X$ with $p < r < q$. Assume that $X$ has a right end point $q \in X$. Requirement $(2)$ tells us that there must be some element in $X$ greater or equal to $q + 1$, which contradicts this immediately. Assume that $X$ has a left end point $q \in X$. Like before, $q$ is the limit of some strictly increasing sequence in $X$. It can, thus, not be a left end point. Using the classical back-and-forth-argument, we conclude that $X$ is isomorphic to~$\q$.
\end{remark}

\bibliographystyle{amsplain}
\bibliography{goodstein}
\end{document}